\documentclass{article}

\usepackage{arxiv}

\usepackage{microtype}
\usepackage{graphicx}
\usepackage{subfigure}
\usepackage{booktabs}
\usepackage{hyperref}

\usepackage{algorithm}
\usepackage{amsmath}
\usepackage{amssymb}
\usepackage{mathtools}
\usepackage{amsthm}
\usepackage[capitalise]{cleveref}
\theoremstyle{plain}
\newtheorem{theorem}{Theorem}[section]
\newtheorem{proposition}[theorem]{Proposition}
\newtheorem{lemma}[theorem]{Lemma}
\newtheorem{corollary}[theorem]{Corollary}
\theoremstyle{definition}

\newtheorem{assumption}[theorem]{Assumption}
\theoremstyle{remark}
\newtheorem{remark}[theorem]{Remark}

\usepackage[textsize=tiny]{todonotes}

\usepackage[utf8]{inputenc} 
\usepackage[T1]{fontenc}    
\usepackage{hyperref}       
\usepackage{url}            
\usepackage{booktabs}       
\usepackage{amsfonts}       
\usepackage{nicefrac}       
\usepackage{microtype}      
\usepackage{xcolor}         
\usepackage{algorithmic}

\usepackage{mathtools}
\usepackage{amsmath,amsfonts,amsthm}
\usepackage{todonotes}
\usepackage{natbib}
\usepackage{makecell}
\usepackage{soul}
\usepackage{ifthen}
\usepackage{pifont}
\usepackage{xcolor}

\allowdisplaybreaks

\usepackage{xspace}

\newcommand{\range}{\mathrm{range}}

\newcommand{\E}[1]{\mathbb{E}\left[#1\right]}

\newcommand{\Ei}[2]{\mathbb{E}_{#1}\left[#2\right]}

\newcommand{\g}{\nabla}

\newcommand{\bg}{\mathrm{D}}
\newcommand{\bb}{\mathrm{G}}

\newcommand{\R}{\mathbb{R}}

\def\<#1,#2>{\langle #1,#2\rangle}

\newcommand{\norm}[1]{\|#1\|}
\newcommand{\sqn}[1]{\norm{#1}^2}
\newcommand{\Norm}[1]{\left\|#1\right\|}
\newcommand{\sqN}[1]{\Norm{#1}^2}

\newcommand{\cbraces}[1]{\left( #1 \right)}
\newcommand{\sbraces}[1]{\left[ #1 \right]}
\newcommand{\braces}[1]{\left\{ #1 \right\}}

\newcommand{\cH}{\mathcal{H}}

\newcommand{\cG}{\mathcal{G}}

\newcommand{\cV}{\mathcal{V}}
\newcommand{\cE}{\mathcal{E}}
\newcommand{\cL}{\mathcal{L}}

\newcommand{\cO}{\mathcal{O}}

\newcommand{\cD}{\mathcal{D}}

\newcommand{\mI}{\mathbf{I}}

\newcommand{\mM}{\mathbf{M}}

\newcommand{\mP}{\mathbf{P}}

\newcommand{\mV}{v}
\newcommand{\mW}{\mathbf{W}}
\newcommand{\mX}{x}
\newcommand{\mY}{y}

\newcommand{\eqdef}{\coloneqq}
\DeclareMathOperator*{\argmin}{arg\,min}

\newcommand{\eps}{\varepsilon}

\usepackage{bm}
\usepackage{cleveref}

\def\vone{{\bm{1}}}

\def\vv{v}

\def\vx{x}

\newcommand{\N}{\mathbb{N}}

\newcommand{\dotp}[2]{\left<#1, #2\right>}

\newcommand{\oL}{\overline{L}}

\def\eqref#1{(\ref{#1})}

\newcommand{\z}{\hat{z}}
\newcommand{\x}{\hat{x}}

\title{Decentralized Finite-Sum Optimization over Time-Varying Networks}

\author{
Dmitry Metelev \\ MIPT\thanks{Moscow Institute of Physics and Technology, Moscow, Russia}\\
\And
Savelii Chezhegov \\ MIPT \\
\And
Alexander Rogozin \\ MIPT, HSE University\thanks{the National Research University Higher School of Economics, Moscow, Russia} \\
\And
Aleksandr Beznosikov \\ MIPT \\
\And
Alexander Sholokhov \\ MIPT \\
\And Alexander Gasnikov \\ MIPT, University Innopolis\thanks{University Innopolis, Innopolis, Russia}, ISP RAS\thanks{Ivannikov Institute for System Programming of the Russian Academy of Sciences, Moscow, Russia}\\
\And
Dmitry Kovalev \\ Yandex\thanks{Yandex Research, Moscow, Russia}}

\begin{document}

\maketitle

\begin{abstract}
	We consider decentralized time-varying stochastic optimization problems where each of the functions held by the nodes has a finite sum structure. Such problems can be efficiently solved using variance reduction techniques. Our aim is to explore the lower complexity bounds (for communication and number of stochastic oracle calls) and find optimal algorithms. The paper studies strongly convex and nonconvex scenarios. To the best of our knowledge, variance reduced schemes and lower bounds for time-varying graphs have not been studied in the literature. For nonconvex objectives, we obtain lower bounds and develop an optimal method GT-PAGE. For strongly convex objectives, we propose the first decentralized time-varying variance-reduction method ADOM+VR and establish lower bound in this scenario, highlighting the open question of matching the algorithms complexity and lower bounds even in static network case.
\end{abstract}

\section{Introduction}
\subsection{Decentralized optimization}

The classic problem statement of distributed optimization is considering a sum-type problem
\begin{equation}\label{main_problem_non_vr}
	\min_{x\in\R^d}~ F(x) := \sum_{i=1}^m F_i(x), 
\end{equation}
where $F_i(x)$ is loss function of $i$-\textit{th} device. The hallmark of distributed optimization is that each device only has access to its own data. One standard approach to solve such a problem is centralized optimization, where devices are called \textit{clients} and there is a \textit{server} responsible for device communication. Other frequent case of distributed optimization is the decentralized paradigm, in which the key factor is that there is no server, and information exchange takes place over a \textit{communication network}, which is described as a graph. This approach has some advantages over centralized optimization due to the fault possibility and privacy issues.

Decentralized optimization has many applications, such as power system control \citep{ram2009distributed,gan2012optimal}, distributed statistical inference \citep{forero2010consensus,nedic2017fast}, vehicle coordination and control \citep{ren2008distributed} and distributed sensing \citep{rabbat2004distributed,bazerque2009distributed}. 
Moreover, this approach is also applicable to machine learning, in particular, to federated learning \citep{konecny2016federated,mcmahan2017communication}, as this automatically leads to a restriction on the presence of the server in the communication network. 

Depending on the structure of the communication graph, decentralized optimization studies can be divided into two types: \textit{static} and \textit{time-varying} cases. The static case means that the communication network is retained throughout the training process. In this path, studies have been conducted in which both primal \citep{kovalev2020optimal} and dual \citep{scaman2017optimal} optimal algorithms have been obtained. Nevertheless, in the case of a fixed communication network, there are other ways of averaging between nodes that are more popular in distributed optimization. One such procedure is AllReduce \citep{chan2007collective}, which allows decentralized behavior to be modeled through a centralized process. As a consequence, one can obtain that the gossip protocol can be replaced by an approach that ostensibly has a server.

On the other hand, for time-varying networks, where the topology is altered from one moment to another, the idea of switching to the “centralized” setting does not work -- changes in communication graph prevent the existence of averaging guarantees for procedures as AllReduce (at least, the synchrony of information transfer is disrupted). As a result, that makes the gossip protocol one of the few options for information exchange. Moreover, cases where the topology of a distributed model changes over time are more frequent than the static case. The instability of links practically happens due to malfunctions in communication, such as a loss of wireless connection between sensors or drones, or human behavior modeling. This path has already been covered quite extensively in previous papers \citep{zadeh1961time,nedic2017achieving,nedic2020distributed, maros2018panda, kovalev2021lower, kovalev2021adom, li2021accelerated, koloskova2020unified, huang2022optimal}, where both dual and primal approaches were considered, and lower bounds were obtained.
\subsection{Variance reduction}
As said earlier, decentralized optimization has applications in machine learning. Thus, as in ML minimization problems, it is natural to consider the function $f(x)$ as an empirical risk function, i.e.
\begin{align}
    \label{int: finite-sum}
    f(x) := \frac{1}{n}\sum\limits_{j=1}^n f_{j}(x),
\end{align}
where $f_{j}(x)$ is the empirical risk for a slice of data. Although the most classical approach for minimizing $f$ is Stochastic Gradient Descent \citep{robbins1951stochastic}, an advanced paradigm for that problem is \textit{variance reduction}. The essence of this technique is a “tricky” update of the step direction, which allows to gradually reduce the variance of the stochastic selection of the set of $\{f_{j}\}_{j=1}^n$ which is involved in the update. That technique was widely studied for both strongly-convex \citep{johnson2013accelerating,defazio2014saga,
 allen2016katyusha} and nonconvex \citep{reddi2016stochastic, allen2016variance,nguyen2017sarah, fang2018spider, li2021page} scenarios. 
\subsection{Decentralized optimization and Variance reduction} 
Since decentralized optimization is rather important for machine learning as well, it is quite natural that the loss function $F_i$ on each device is constructed from its own data, and hence has the form of \eqref{int: finite-sum}. This immediately leads to the problem statement we consider in the paper:
\begin{align}
    \label{main_problem}
    \min_{x\in\R^d}~ F(x) := \sum_{i=1}^m F_i(x),
\end{align}
where $F_i = \frac{1}{n}\sum\limits_{j=1}^n f_{ij}(x)$. Such a problem has already been investigated in the literature for both strongly convex \citep{hendrikx2021optimal, li2022variance} and nonconvex \citep{xin2021fast, xin2022fast, li2022destress, luo2022optimal} settings. Nevertheless, that studies were devoted only to the case of static networks, which indicates the theoretical gap in the case of time-varying networks for the problem \eqref{main_problem}. To close this issue, we formulate the contribution of our paper as follows: 

$\bullet$ \textbf{Strongly convex case.} For the strongly convex problem \eqref{main_problem} over time-varying networks, we propose a method \textsc{ADOM+VR} (Algorithm~\ref{scary:alg}). The method is based on the combination of \textsc{ADOM+} algorithm for non-stochastic decentralized optimization over time-varying networks \citep{kovalev2021lower} and loopless \textsc{Katyusha} approach for finite-sum problems \citep{kovalev2020don}.

$\bullet$ \textbf{Nonconvex case.} For nonconvex problem \eqref{main_problem} over time-varying graphs, we propose an algorithm \textsc{GT-PAGE} (Algorithm~\ref{alg:yup}). The main idea is to implement the \textsc{PAGE} gradient estimator \citep{li2021page} for finite-sum problem into the gradient tracking technique \citep{nedic2017achieving}.

$\bullet$ \textbf{Optimality.} We propose first lower bounds for decentralized optimization on time-varying networks, which show that the presented methods are not only state-of-the-art in terms of communication and oracle complexity, but also \textit{optimal} under some conditions.



\section{Related Work} 
Decentralized optimization \eqref{main_problem_non_vr} over static and time-varying networks has been actively developing in recent years. For the strongly convex case, dual-based methods and lower bounds for (non-stochastic) strongly convex optimization over static graphs were proposed \citep{scaman2017optimal}. Optimal primal methods were obtained in \citep{kovalev2020optimal}. For time-varying networks, non-accelerated primal \citep{nedic2017achieving} and dual \citep{maros2018panda} methods were proposed. After that, accelerated algorithms were given in \citep{kovalev2021adom} (\textsc{ADOM}) for dual oracle and -- \textsc{ADOM+} and \textsc{Acc-GT} -- in \citep{kovalev2021lower} and \citep{li2021accelerated} for primal oracle. These methods match the lower complexity bounds for time-varying graphs developed in \citep{kovalev2021lower}. For the nonconvex case, primal algorithms were provided: \textsc{GT-DSGD} \citep{xin2021improved}, \textsc{DSGT} \citep{zhang2019decentralized}, \textsc{D-PSGD} \citep{lian2017can}, \textsc{DeTAG} \citep{lu2021optimal}, \textsc{MG-DSGD} \citep{yuan2022revisiting} for the static case, and \textsc{DSGD} \citep{koloskova2020unified}, \textsc{DSGT} \citep{zhang2019decentralized}, \textsc{MC-DSGT} \citep{huang2022optimal} for time-varying networks. Moreover, for both scenarios lower bounds were also provided in \citep{yuan2022revisiting} and \citep{huang2022optimal}, respectively.

Under finite-sum problem \eqref{int: finite-sum}, classical variance reduction methods such as SAGA \citep{defazio2014saga} and SVRG \citep{johnson2013accelerating} allow to enhance the rates for stochastic optimization problems with finite-sum structure. Accelerated variance reduced schemes require adding a negative momentum, also referred to as Katyusha momentum \citep{allen2016katyusha}. Considering nonconvex problems, recent development was started with \citep{reddi2016stochastic} and \citep{allen2016variance}, where algorithms based on SVRG were proposed. More recently, other modifications of SVRG scheme with the same gradient complexity $\mathcal{O}\left(n + n^{2/3}/{\epsilon^2}\right)$ were proposed in \citep{li2018simple}, \citep{ge2019stabilized} and \citep{horvath2019nonconvex}. Moreover, optimal algorithms were presented, such as Spider~\citep{fang2018spider}, SNVRG~\citep{zhou2020stochastic}, methods based on SARAH~\citep{nguyen2017sarah} (e.g. SpiderBoost~\citep{wang2018spiderboost}, ProxSARAH~\citep{pham2020proxsarah}, Geom-SARAH~\citep{horvath2022adaptivity}) and PAGE~\citep{li2021page}, which have $\mathcal{O}\left(n + \sqrt{n}/{\epsilon^2}\right)$ gradient estimation complexity. Lower bounds were also presented \citep{fang2018spider, li2021page}.

In strongly-convex decentralized optimization over static graphs, optimal variance reduced methods ADFS \citep{hendrikx2019accelerated} and Acc-VR-EXTRA \citep{li2022variance} were proposed. The corresponding lower bounds were given in \citep{hendrikx2021optimal}. 
In the nonconvex case, the result of first application of variance reduction and gradient tracking to decentralized optimization for static graphs was the method  D-GET~\citep{sun2020improving}. Later, algorithms GT-SAGA~\citep{xin2021fast}, GT-HSGD~\citep{xin2021hybrid} and GT-SARAH~\citep{xin2022fast} were proposed, which improve the complexity of communication rounds and local computations comparing to D-GET. A relatively new result was achieved by the method DESTRESS~\citep{li2022destress}, which is optimal in terms of local computations, but ineffective in terms of number of communications in case of static graphs. This method was improved into DEAREST~\citep{luo2022optimal},  which is optimal. Nevertheless, the application of variance reduction has not been studied for the case of nonconvex decentralized optimization over time-varying graphs. In Table~\ref{tab:complexity_nonconvex} we present an overview of methods for which it is possible to explicitly write out complexities in terms of constants of smoothness and $\chi$. For an overview of other algorithms, see Table 1 in \citep{xin2022fast} and Table 1 in \citep{xin2021hybrid}. 
\begin{remark}
    It is necessary to clarify that optimality of DESTRESS and DEAREST is not clear in terms of dependence on smoothness constants. Indeed, mentioned constants $L_s, \hat{L}$ and $L$ are sensitive to $n$. In \cref{proof_lower} we show that ratios $\sqrt{n}L = \hat{L}$ and $nL = L_s$ can be achieved.
\end{remark}
\begin{table}[h]
	\label{table:complexity_str_convex}
	\centering
        \resizebox{0.75\columnwidth}{!}{
	\begin{tabular}{|c|c|c|}
		\hline
		\bf Algorithm & {\bf Comp.} & {\bf Comm.} \\
		\hline
		ADFS  \citep{hendrikx2021optimal}& $n + \sqrt{n\max_i \frac{\oL_i}{\mu_i}}$ & $\sqrt{\chi} \sqrt{\max_i \frac{L_i}{\mu_i}}$ $^{\text{\textcolor{blue}{(a)}}}$ \\
		\hline
		Acc-VR-EXTRA  \citep{li2020optimal}& $n + \sqrt{n\frac{\oL}{\mu}}$ & $\sqrt{\chi\frac{L}{\mu}}$  \\
		\hline
		ADOM+VR  (Alg.~\ref{scary:alg}, this paper) & $n + \sqrt{n\frac{\oL}{\mu}}$ & $\chi\sqrt{\frac{L}{\mu}}$ $^{\text{\textcolor{blue}{(b)}}}$\\
		\hline
		Lower bound (Th.~\ref{thm:lower_strong_conv}, this paper) & $n + \sqrt{n \max_i \frac{\oL_i}{\mu_i}}$ & $\chi \sqrt{\max_i \frac{L_i}{\mu_i}}$ $^{\text{\textcolor{blue}{(b)}}}$\\
		\hline
	\end{tabular}
    }
    \caption{Computational (the number of stochastic oracle calls per node) and communication complexities of decentralized methods for finite-sum \textbf{strongly convex} optimization. $O(\cdot)$ notation and $\log(1/\eps)$ factor are omitted. Notation:\\
    \textcolor{blue}{(a)} -- dual method; 
    \textcolor{blue}{(b)} -- time-varying network is considered;\\
    Here $L$ is taken from \cref{assum:smoothness_F}, $\oL_i$ from \cref{assum:smoothness_f} and $L_i$ from \cref{assum:kappa_lower}. For notation and more details, see Sections \ref{sec:notation_and_assumptions} and \ref{sec:lower_bounds}. 
    }
\end{table}
\begin{table}[h]
    \centering
    \resizebox{0.75\columnwidth}{!}{
	\begin{tabular}{|c|c|c|}
	\hline
	\textbf{Algorithm} & \textbf{Comp}. & \textbf{Comm}. \\ \hline
        {DESTRESS  \citep{li2022destress}} & $n + \frac{\sqrt{n}L_{s}\Delta}{\eps^2}$ & $\sqrt{\chi}\big(\sqrt{mn} + \frac{L_{s}\Delta}{\eps^2}\big)$ \\ \hline
        {DEAREST \citep{luo2022optimal}} & $n + \frac{\sqrt{n}\hat{L}\Delta}{\eps^2}$ & $\sqrt{\chi}\frac{\hat{L}\Delta}{\eps^2}$ \\ \hline
	{GT-PAGE (Alg.~\ref{alg:yup}, this paper)} & $n + \frac{\sqrt{n}\hat{L}\Delta}{\eps^2}$ & $\chi\frac{L\Delta}{\eps^2}$ $^{\text{\textcolor{blue}{(a)}}}$ \\ \hline
        {Lower bound (Th.~\ref{thm: lower-bound-sq-norm}, this paper)} & $n + \frac{\sqrt{n}\hat{L}\Delta}{\eps^2}$ & $\chi\frac{L\Delta}{\eps^2}$ $^{\text{\textcolor{blue}{(a)}}}$ \\ \hline
\end{tabular}
}
\caption{Computational (the number of stochastic oracle calls per node) and communication complexities of decentralized methods for finite-sum \textbf{nonconvex} optimization over time-varying graphs. $O(\cdot)$ notation is omitted. Notation:\\
    \textcolor{blue}{(a)} -- time-varying network is considered; \\
Here $L_{s} = \max_{i,j} L_{ij}$ from \cref{assum:smoothness_f}, $L$ from \cref{assum:smoothness_F} and $\hat{L}$ from \cref{assum:smoothness_F_average}. For notation, see \cref{sec:notation_and_assumptions}.} 
\label{tab:complexity_nonconvex}
\end{table}

\section{Notation and Assumptions}\label{sec:notation_and_assumptions}

Throughout this paper, we adopt the following notations: We denote by \(||\cdot||=||\cdot||_2\) the norm in \(L_2\) space. The Kronecker product of two matrices is denoted as \(A \otimes B\). We use \(\cD(X)\) to denote some distribution over a finite set \(X\). The sets of batch indices are denoted by \(S\), expressed as \(S = (\xi^1, \ldots, \xi^b)\), where \(\xi^j\) is a tuple of \(m\) elements, each corresponding to a node, specifically \(\xi^j = (\xi^j_1, \ldots, \xi^j_m)\), with \(\xi_i^j\) being the index of the local function on \(i\)-th node in \(j\)-th element of the batch. Also for each $i=1,\ldots,m$ define $S_i = (\xi_i^1, \ldots, \xi_i^b)$. Each node maintains its own copy of a variable corresponding to a specific variable in the algorithm. The variables in the algorithm are aggregations of the corresponding node variables:
\begin{equation*}
    x = (x_1, x_2, \ldots, x_m)\in \R^{d\times m}.
\end{equation*}
With a slight abuse of notation we denote $F(x) = \sum_{i=1}^m F_i(x_i)$ and $\g F(x) = \left(\g F_1(x_1), \ldots, \g F_m(x_m)\right).$
Linear operations and scalar products are performed component-wise in a decentralized way. Let us introduce an auxiliary subspace $\cL = \{x\in\R^{d\times m}|x_1=\ldots=x_m\}$, respectively $\cL^\perp = \{x\in\R^{d\times m}|x_1+\ldots+x_m=0\}$. We also let $x^* = \argmin_{x\in\R^d} F(x)$ or $x^* = \argmin_{x\in\R^{d\times m}} F(x)$, depending on the context.
 
Let us pass to assumptions on objective functions. We introduce different concepts of smoothness: Assumptions~\ref{assum:smoothness_F}, \ref{assum:smoothness_f} and \ref{assum:strong_convexity} are used in Algorithm~\ref{scary:alg}; Assumptions~\ref{assum:smoothness_F} and \ref{assum:smoothness_F_average} are for Algorithm~\ref{alg:yup}.
\begin{assumption}\label{assum:smoothness_F}
	For each $i = 1, \ldots, m$ function $F_i$ is $L$-smooth, i.e.
	\begin{align*}
        \norm{\nabla F_{i}(y) - \nabla F_{i}(x)} \leq L\norm{y - x}.
	\end{align*}
\end{assumption}
\begin{assumption}\label{assum:smoothness_f}
	For each $i = 1, \ldots, m$ and $j = 1, \ldots, n$ function $f_{ij}$ is convex and $L_{ij}$-smooth, i.e.
	\begin{align*}
        \norm{\nabla f_{ij}(y) - \nabla f_{ij}(x)} \leq L_{ij}\norm{y - x}.
	\end{align*}
    For each $i = 1, \ldots, m$ let us define
    \begin{equation*} 
        \oL_i = \frac{1}{n}\sum_{j=1}^n L_{ij},\text{  }
        \oL = \max_{i}\left\{\oL_i\right\}.
    \end{equation*}
\end{assumption}
\begin{assumption}\label{assum:smoothness_F_average}
	For each $i = 1, \ldots, m$ function $F_i$ is $\hat{L}$-average smooth, i.e.
	\begin{align*}
		\frac{1}{n}\sum\limits_{j=1}^n\norm{\nabla f_{ij}(y) - \nabla f_{ij}(x)}^2 \leq \hat{L}^2\norm{y - x}^2.
	\end{align*}
\end{assumption}
Let us briefly discuss Assumptions \ref{assum:smoothness_F}, \ref{assum:smoothness_f} and \ref{assum:smoothness_F_average}. \cref{assum:smoothness_F} is a basic one for the consideration of decentralized approaches; it means that $L$ is an uniform bound for effective constants of smoothness of each $F_i$. Nevertheless, assumptions involving $\{f_{ij}\}$ differ depending on the convex/nonconvex scenarios. Indeed, \cref{assum:smoothness_f} arises from the structure of proposed method for the strongly convex setup, where we avoid the assumption on uniform smoothness of each $f_{ij}, j = \overline{1, n}$ \citep{kovalev2020don} and propose a slightly tighter analysis since we engage the \textit{effective} constants of smoothness compared to the previous results. At the same time, \cref{assum:smoothness_F_average} is the classic one, especially for the nonconvex case, since it allows to measure the constant of smoothness effectively in stochastic approaches for the finite-sum problem \eqref{int: finite-sum}. Moreover, optimal results for the nonconvex setup already exist under \cref{assum:smoothness_F_average} \citep{fang2018spider, li2021page}, hence we decide not to deviate from it.
\begin{remark}
    Moreover, it can be shown that there are some ratios between $L$ and other measures of smoothness. Thus, for example,
    \begin{align*}
        L &\leq \oL \leq nL; \qquad
        L \leq \hat{L} \leq \sqrt{n}L.
    \end{align*}
    These correlations can play an important role, for example, in constructing of lower bounds. For more details, see \cref{sec:lower_bounds}.
\end{remark}
Also we introduce an assumption on strong convexity.
\begin{assumption}\label{assum:strong_convexity}
	For each $i = 1, \ldots, m$ function $F_i$ is $\mu$-strongly convex, i.e.
	\begin{align*}
		F_i(y) \geq F_i(x) + \<\g F_i(x), y - x> + \frac{\mu}{2}\sqn{y - x}_2.
	\end{align*}
\end{assumption}
Decentralized communication is represented by a sequence of graphs $\braces{\cG^k = (\cV, \cE^k)}_{k=0}^\infty$. With each graph, we associate a gossip matrix $\mW(k)$.
\begin{assumption}\label{assum:gossip_matrix_sequence}
For each $k = 0, 1, 2, \ldots$ it holds
\begin{enumerate}
	\item $[\mW(k)]_{i,j} \neq 0$  if and only if $(i,j) \in \cE^k$ or $i=j$,
	\item $\ker \mW(k) \supset  \left\{ (x_1,\ldots,x_m) \in \R^n : x_1 = \ldots = x_m \right\}$,
	\item $\range \mW(k) \subset \left\{(x_1,\ldots,x_m) \in\R^n : \sum_{i=1}^m x_i = 0\right\}$,
	\item There exists $\chi \geq 1$, such that 
	\begin{align}\label{eq:chi}
		&\sqn{\mW(k) x - x} \leq (1-\chi^{-1})\sqn{x} \\
		&\text{ for all } x  \in \left\{ (x_1,\ldots,x_m) \in \R^{d\times m} :\sum_{i =1}^m x_i = 0\right\}. \nonumber
	\end{align}	
\end{enumerate}
\end{assumption}
The main distinguishing property in \cref{assum:gossip_matrix_sequence} is the $4^{th}$ one, which asserts the connectivity of each communication network into a sequence $\{\cG^k\}_{k=0}^\infty$.
What is more, matrices $\mW(k)$ can be chosen as $\mW(k) = \mathbf{L}(\cG^k) / \lambda_{\max}(\mathbf{L}(\cG^k))$, where $\mathbf{L}(\cG^k)$ denotes a graph Laplacian matrix. Moreover, if the network is constant ($\cG^k \equiv \cG$), we have $\chi = \lambda_{\max}(\mathbf{L}(\cG))/\lambda_{\min}^+(\mathbf{L}(\cG))$, i.e. $\chi$ equals the graph condition number.




\section{Algorithms}\label{sec:optimal_algorithms}

In this section we propose two algorithms, \textsc{ADOM+VR} and \textsc{GT-PAGE} for the strongly convex and nonconvex cases, respectively, and give theoretical guarantees of the convergence. It is worth mentioning that considered complexities for decentralized optimization and variance reduction approaches are different. To be more precise, in decentralized techniques we measure the number of communications, while for variance reduction methods we estimate the number of oracle calls. Therefore, these two complexities should be measured for the comprehensive analysis. For more details, see \cref{corollary:adom+_conv} and \cref{cor:nonconvex_setup}.


\subsection{Strongly Convex Case}\label{subsec:optimal_alg_str_convex}

For the strongly convex case, we propose a method, which consists of ADOM+~\citep{kovalev2021lower} as a base decentralized method and Katyusha~\citep{allen2016katyusha,kovalev2020don} gradient estimator with negative momentum.

\textbf{ADOM+.}
 The main idea of ADOM+ is briefly described as follows. The given optimization problem can be written in decentralized form as
\begin{equation*}
\min_{x\in \cL} F(x).
\end{equation*}
This can be further reformulated in the following way, which is the basis for the ADOM+ method:
\begin{equation*}
\min_{x\in \R^{d\times m}}\max_{y\in\R^{d\times m}}\max_{z\in \cL^{\perp}} F(x)-\frac{\nu}{2}\sqn{x}-\<y, x>-\frac{1}{2\nu}\sqn{y+z}.
\end{equation*}
It is not difficult to show that in case $\nu<\mu$, this saddle point problem is strongly convex, which means that it has a single solution $(x^*, y^*, z^*)$ satisfying the optimality conditions:
\begin{align}
0 &= \g F(x^*)-\nu x^* - y^*,\label{opt:x}\\
0 &= \nu^{-1}(y^*+z^*)+x^*,\label{opt:y}\\
0 &\ni y^* + z^*. \label{opt:z}
\end{align}

The idea is described in more detail in \citep{kovalev2021lower}.

\textbf{Katyusha.} The method Katyusha \citep{allen2016katyusha} contains two ideas -- negative momentum and special gradient estimator. If the problem \eqref{int: finite-sum} is considered, then at step $k$, instead of the gradient $\nabla f(x^k)$ one uses an estimator
\begin{align}\label{eq:grad_estimator_katyusha}
	\nabla^k = \frac{1}{b} \sum_{i\in S} [\g f_i(x^k) - \g f_i(w^k)] + \nabla f(w^k),
\end{align}
where $S$ is a random batch of indices of size $b$, $x^k$ is the current iterate and $w^k$ is a reference point at which the full gradient is computed. However, we integrate a loopless version of Katyusha \citep{kovalev2020don} into our algorithm rather than the original one. 
\begin{algorithm}[h]
	\caption{\textsc{ADOM+VR}}
	\label{scary:alg}
	\begin{algorithmic}[1]
		\STATE {\bf input:} $x^0, y^0,m^0,\omega^0 \in (\R^d)^\cV$, $z^0 \in \cL^\perp$
		\STATE $x_f^0 = \omega^0 = x^0$, $y_f^0 = y^0$, $z_f^0 = z^0$
		\FOR{$k = 0,1,\ldots,N-1$}
		\STATE $x^k_g = \tau_1 x^k + \tau_0 \omega^k + (1 - \tau_1-\tau_0)x^k_f$\label{scary:line:x:1}
		\STATE{$S_i^k \sim \cD^b_i\left(\{1,2,\ldots,n\}\right), p_{ij}=\frac{L_{ij}}{n\oL_i}$}
		\STATE{$\left(\nabla^k\right)_i = \frac{1}{b}\sum_{j \in S^k_i} \frac{1}{np_{ij}}\big[
			\g f_{ij}(x^k_{g,i}) - \g f_{ij} (\omega^k_i) \big] + \g F_i(\omega^k_i)$}
		\STATE $x^{k+1} = x^k + \eta\alpha(x_g^k - x^{k+1}) - \eta\left[\nabla^k - \nu x_g^k - y^{k+1}\right] $\label{scary:line:x:2}
		\STATE $x_f^{k+1} = x_g^k + \tau_2 (x^{k+1} - x^k)$\label{scary:line:x:3}
		\STATE{$\omega_i^{k+1} = \begin{cases} x^k_{f, i}, &\text{with prob. } p_1 \\
				x^k_{g, i}, &\text{with prob. } p_2
				\\ \omega_i^k, &\text{with prob. } 1-p_1-p_2 \end{cases}$}\label{scary:line_omega}
		\STATE $y_g^k = \sigma_1 y^k + (1-\sigma_1)y_f^k$\label{scary:line:y:1}
		\STATE $y^{k+1} = y^k + \theta\beta (\nabla^k - \nu x_g^k - y^{k+1}) -\theta\left[\nu^{-1}(y_g^k + z_g^k) + x^{k+1}\right]$\label{scary:line:y:2}
		\STATE $y_f^{k+1} = y_g^k + \sigma_2 (y^{k+1} - y^k)$\label{scary:line:y:3}
		\STATE $z_g^k = \sigma_1 z^k + (1-\sigma_1)z_f^k$\label{scary:line:z:1}
		\STATE $z^{k+1} = z^k + \gamma \delta(z_g^k - z^k) - (\mW(k)\otimes \mI_d)\left[\gamma\nu^{-1}(y_g^k+z_g^k) + m^k\right]$\label{scary:line:z:2}
		\STATE $m^{k+1} = \gamma\nu^{-1}(y_g^k+z_g^k) + m^k - (\mW(k)\otimes \mI_d)\left[\gamma\nu^{-1}(y_g^k+z_g^k) + m^k\right]$\label{scary:line:m}
		\STATE $z_f^{k+1} = z_g^k - \zeta (\mW(k)\otimes \mI_d)(y_g^k + z_g^k)$\label{scary:line:z:3}
		\ENDFOR
		\STATE {\bfseries Output:} $x^N$
	\end{algorithmic}
\end{algorithm}

\begin{theorem}\label{thm:adom+_conv}
	Let Assumptions \ref{assum:smoothness_f}, \ref{assum:smoothness_F}, \ref{assum:strong_convexity}, \ref{assum:gossip_matrix_sequence} and $b \geq \oL/L$ hold. Then Algorithm~\ref{scary:alg} requires $N$ iterations to yield $x^N$ such that $\norm{x^N - x^*}^2\leq \varepsilon$, where
	\begin{equation*}
		N = \cO \left(\left(\frac{n}{b}+\left(\frac{\sqrt{n}}{b}+\frac{n\oL}{b^2 L}+\chi\right)\sqrt{\frac{L}{\mu}}\right)\log \frac{1}{\epsilon}\right).
	\end{equation*}
\end{theorem}

\begin{corollary}\label{corollary:adom+_conv}
    In the setting of Theorem~\ref{thm:adom+_conv}, with \(b\sim\max\left\{\sqrt{n\oL/L}, n\sqrt{\mu/L}\right\}\) and the number of communications per iteration $\sim\chi$, the algorithm requires
    \begin{equation*}
        \cO\left(n + \sqrt{\frac{n\oL}{\mu}}\right)\log\frac{1}{\epsilon}
    \end{equation*} oracle calls per node and
    \begin{equation*}
        \cO\left(\chi\sqrt{\frac{L}{\mu}}\log\frac{1}{\epsilon}\right)
    \end{equation*} communication iterations to reach \(\norm{x^N - x^*}^2 \leq \epsilon\).
\end{corollary}
Proofs of \cref{thm:adom+_conv} and \cref{corollary:adom+_conv} can be found in Appendix~\ref{appendix:A}.

\subsection{Nonconvex Case}\label{subsec:optimal_alg_nonconvex}
For the nonconvex setup, we propose a method based on a combination of gradient tracking and PAGE gradient estimator \citep{li2021page}. The main idea of this approach consists of two parts.

\textbf{Gradient Tracking.} Gradient tracking scheme can be written as in \citep{nedic2017achieving}:
\begin{align*}
    x^{k+1} &= \mW^kx^k - \eta y^k;\\
    y^{k+1} &= \mW^ky^k + \nabla F(x^{k+1}) - \nabla F(x^{k}).
\end{align*}
Such an algorithm leads to $y_i^k$ being an approximation of the average gradient from all devices in the network at each iteration.

\textbf{PAGE.} The key meaning of PAGE is as follows. Calculating the full gradient can be expensive, but finite-sum construction allows to count the batched gradient, which is clearly lower in computational cost. Moreover, PAGE update does not have any loops (as, for example, in SVRG~\citep{johnson2013accelerating}) and can be computed recursively as follows:
\begin{align}\label{eq:grad_estimator_page}
	\nabla^{k+1} = \frac{1}{b} \sum_{i\in S} \sbraces{\nabla f_i(x^{k+1}) - \nabla f_i(x^k)} + \nabla^k,
\end{align}
where $S$ denotes a random set of indices of size $b$. Note that unlike estimator \eqref{eq:grad_estimator_katyusha} for strongly convex case, PAGE estimator \eqref{eq:grad_estimator_page} stores the gradient from previous iteration, not only the gradient snapshot.\\ 
\begin{algorithm}[h]
	\caption{\textsc{GT-PAGE}}
	\label{alg:yup}
	\begin{algorithmic}[1]
		\STATE {\bfseries Input:} Initial point $ \mX^0= (\vone_m \otimes \mI_d)\vx_0$, $\mY^0 = \nabla F(\mX^0)$, $\mV^0 = \frac{1}{m} (\vone_m\vone^\top_m \otimes \mI_d) \mY^0$, step size $\eta$, minibatch size $b$. \\
		\FOR{$k=0, 1, \ldots, N - 1$}
		\STATE $\mX^{k+1} = ((\mI_m - \mW(k))\otimes \mI_d)\mX^k - \eta \mV^k$
        \STATE $S_i^k \sim \cD_i^b(\{1, 2, \ldots, n\})$, $p_{ij} = \frac{1}{n}$ 
		\STATE $\big(\nabla^k\big)_i = \mY^k_i + \frac{1}{b}\sum\limits_{j \in S^k_i} \big(\nabla f_{ij}(x_i^{k+1}) - \nabla f_{ij}(x_i^{k})\big)$
		\STATE $\mY^{k+1} = 
		\begin{cases}
			\nabla F(\mX^{k+1}), &\text{with prob. } p,\\
			\nabla^k, &\text{with prob. } 1-p
		\end{cases}$
		\STATE $\mV^{k+1} = ((\mI_m - \mW(k))\otimes \mI_d)\mV^{k} + \mY^{k+1} - \mY^k$
		\ENDFOR
        \STATE {\bfseries Output:} {$x$ chosen uniformly from $\{x^k\}_{k=0}^{N-1}$}
	\end{algorithmic}
\end{algorithm}

\begin{theorem}
    \label{thm:nonconvex-setup}
    Let Assumptions \ref{assum:smoothness_F}, \ref{assum:smoothness_F_average} and \ref{assum:gossip_matrix_sequence} hold. Then, Algorithm \ref{alg:yup} requires $N$ iterations to yield $\hat{x}^N$, which is randomly taken from $\{\bar\vx^k\}_{k=0}^{N-1}$ such that $\E{\norm{\nabla F(\hat{x}^N)}^2} \leq \epsilon^2$, where
    \begin{align*}
        N =\mathcal{O}\left(\frac{\chi^3 L\Delta\left(1 + \sqrt{\frac{(1-p)\hat{L}^2}{bpL^2}}\right)}{ \epsilon^2}\right),
    \end{align*}
    where $\Delta = F(x_0) - F^*$.
\end{theorem}
\cref{thm:nonconvex-setup} says that the number of iterations required to reach $\varepsilon$-accuracy is proportional to $\chi^3$. Nevertheless, the \textit{multi-stage consensus} technique can be applied, since it is a universal way to divide oracle and communication complexities of a decentralized optimization method. Instead of performing one synchronized communication, let us perform several iterations in a row. Following \citep{kovalev2021lower}, we introduce
\begin{align*}
	\mW(k; T) = \mI_m - \Pi_{q=kT}^{(k+1)T - 1} (\mI_m - \mW(q)).
\end{align*}
It can be shown that if we take $T = \lceil\chi\rceil$, then condition number of $\mW(k; T)$ reduces to $O(1)$. To see that, note that for all $x\in\cL^\bot$ it holds
\begin{align*}
	\norm{\mW(k; T)x - x}^2
	&\leq (1 - \chi^{-1})^T \norm{x}^2 \\
	&\leq \exp(-T\chi^{-1}) \norm{x}^2\leq e^{-1} \norm{x}^2.
\end{align*}
In other words, by using multi-stage consensus we reduce $\chi$ to $O(1)$ in iteration complexity by paying a $\lceil\chi\rceil$ times more communications per iteration. 
\begin{remark}
For static networks, Chebyshev acceleration replaces multi-stage consensus \citep{scaman2017optimal}. Term $\chi$ in complexity is reduced to $O(1)$ at the cost of performing $\lceil\sqrt\chi\rceil$ communications per iteration. (Static) gossip matrix $\mW$ is replaced by a Chebyshev polynomial $P(\mW)$.
\end{remark}
\begin{corollary}
	\label{cor:nonconvex_setup}
	In the setting of Theorem~\ref{thm:nonconvex-setup}, let $b = \frac{\sqrt{n}\hat{L}}{L},~ p = \frac{b}{n + b}$ and number of communications per iteration $\chi$. Then Algorithm~\ref{alg:yup} with multi-stage consensus requires
	\begin{align*}
		\mathcal{O}\cbraces{n+ \frac{\sqrt{n}\hat{L}\Delta}{\eps^2}} \quad \text{ and } \quad \mathcal{O}\cbraces{\frac{\chi L\Delta}{\eps^2}}
	\end{align*}
	oracle calls per node and
	communications, respectively, to reach accuracy $\eps$, i.e. $\E{\norm{\nabla F(\hat{x}^N)}^2} \leq \epsilon^2$.
\end{corollary}
Proofs of \cref{thm:nonconvex-setup} and \cref{cor:nonconvex_setup} can be found in \cref{upper} and \cref{upper-separate} respectively.
\begin{remark}
\label{upper-static}
    It should be clarified that in the case of time-static graphs Chebyshev acceleration allows us to go from $\chi$ to $\sqrt{\chi}$ in the bound on the number of communications.
\end{remark}
\section{Lower Bounds}\label{sec:lower_bounds}
In this section, we present lower bounds for the strongly convex case in terms of \citep{hendrikx2021optimal} and for the nonconvex case. It is important to note that the setup for the strongly convex case in which lower bounds for stochastic problems are considered is different from the class of problems analyzed in this article, which will be discussed in more detail later.

The high-level concept underlying lower bounds in decentralized optimization involves creating a decentralized counterexample problem, where information exchange between two vertex clusters is slow. More specifically, the vertices in the counterexample are divided into three types: the first type can potentially ``transfer'' the gradient from even positions to the next, introducing a new dimension, the second type can do so from odd positions, the third type does nothing. We take a ``bad'' function for the corresponding optimization problem and divide it by the corresponding node types in such a way that different clusters contain components of the "bad" function that can approach the solution only after ``communicating'' with nodes from another cluster.

It is also worth noting that the lower bounds in non-stochastic setting have a different interpretation than in stochastic one. Namely, in classical decentralized optimization, the lower bounds on the number of computational iterations are given in terms of the number of iterations when the network performs local computations without exchanging information. However, in stochastic setting, the lower bounds is interpreted as the maximum number of stochastic oracle calls per node, which forces almost the entire ``bad'' function to be put into a single node. This is exactly the problem that underlies the current divergence of upper and lower bounds in the strongly convex case.

\subsection{First-order Decentralized Algorithms}

Following \citep{kovalev2021adom} and \citep{hendrikx2021optimal}, let us formalize the concept of a decentralized optimization algorithm. The procedure will consist of two types of iterations: communicational iterations, in which nodes cannot access the oracle, but only exchange information with neighbors, and computational iterations, in which nodes do not communicate with each other, but only perform local computations in their memory. Let time be discrete, each iteration $k$ is either communicational or computational. For any vertex $i$, denote by $\mathcal{H}_i(k)$ the local memory at $k$th iteration. Then the following inclusions hold:

1. For all $i = 1, \ldots, m$, if nodes perform a local computation at step $k$, local information is updated as
\begin{align*}
\mathcal{H}_i(k + 1) \subseteq \text{span}\bigg(\bigcup_{j=1}^n\{x, \nabla f_{ij}(x), \nabla f_{ij}^*(x): x \in \mathcal{H}_i(k)\}\bigg).
\end{align*}

2. For all $i = 1, \ldots, m$, if nodes perform a communicational iteration at time step \( k \), local information is updated as

$$ \mathcal{H}_i(k + 1) \subseteq \text{span}\left( \bigcup_{j \in \mathcal{N}^k_i} \mathcal{H}_j(k) \cup \mathcal{H}_i(k) \right),$$

where $\mathcal{N}^k_i$ is neighbors of node $i$ at $k$-th step.

\subsection{Strongly Convex Case}

In the strongly convex case, we formulate the lower bounds under slightly different assumptions. We let each function $F_i$ have its own smoothness and strong convexity parameters.
\begin{assumption}\label{assum:strong_convexity_lower}
	For each $i = 1, \ldots, m$ function $F_i$ is $\mu_i$-strongly convex and $L_i$-smooth.
\end{assumption}

\begin{assumption}\label{assum:kappa_lower}
    For all $i = 1, \ldots, m$, we have $\kappa_b \geq L_i/\mu_i$ and $\kappa_s \geq \frac{1}{n}\sum_{j=1}^n L_{ij}/\mu_i$.
\end{assumption}

In this case, we allow functions on nodes to have different constants of strong convexity, preserving the constraints on condition numbers. This plays a role in lower bounds, because in the counterexample problem the strong convexity constants on the nodes can differ by a factor of $m$.

\begin{theorem}\label{thm:lower_strong_conv}
For any $\chi>24$, for any $\kappa_b > 0$, there exists a constant $\kappa_s > 0$, a time-varying network $\{\cG^k\}_{k=1}^{\infty}$ on $m$ nodes, the corresponding sequence of gossip matrices $\{\mW(k)\}_{k=1}^{\infty}$ satisfying Assumption~\ref{assum:gossip_matrix_sequence}, and functions $\{f_{ij}\}$, such that the problem~\eqref{main_problem} satisfies Assumptions~\ref{assum:smoothness_f}, \ref{assum:strong_convexity_lower}, \ref{assum:kappa_lower} and for any first-order decentralized algorithm holds
\begin{equation*}
    \frac{1}{nm}\sum_{i=1}^m\sum_{j=1}^n\frac{\sqn{x_{ij} - x^*}}{\sqn{x^0_{ij} - x^*}} \geq \max\left\{T_1, T_2\right\},
\end{equation*}
where
\begin{align*}
    T_1 &= \left(1 - \frac{2}{\sqrt{\frac{2}{3}\kappa_b+\frac{1}{3}} + 1}\right)^{2 + 16N_c/(\chi-24)},
    \\
    T_2 &= \left(1 - \frac{2n}{\sqrt{n}\sqrt{\frac{2}{3}\kappa_s+n/3} + n}\right)^{4N_s/n},
\end{align*}
 $N_c$ is the number of communication iterations, $N_s$ is the maximum number of stochastic oracle calls on any node, and $x_{ij}\in \cH_i(k)$, $k$ is the number of the last time step.
\end{theorem}
\begin{proof}
    The proof may be found in Appendix~\ref{appendix:B}.
\end{proof}
\begin{corollary}
For any $\chi > 0$ and any $\kappa_b > 0$, there exists a decentralized problem satisfying Assumptions~\ref{assum:smoothness_f}, \ref{assum:gossip_matrix_sequence}, \ref{assum:strong_convexity_lower}, and \ref{assum:kappa_lower}, such that for any first-order decentralized algorithm for each node to reach an $\epsilon$-solution of problem~\eqref{main_problem}, a minimum of $N_c$ communication iterations and $N_s$ stochastic oracle calls on some node are required, where
\begin{align*}
    N_s &= \Omega\left(\left(n+\sqrt{n\kappa_s}\right)\log\left(\frac{1}{\eps}\right)\right),\\
    N_c &= \Omega\left(\chi\sqrt{\kappa_b}\log\left(\frac{1}{\eps}\right)\right).
\end{align*}
\end{corollary}
As we can see, the obtained lower bound has different setting than the class of problems on which the work of Algorithm~\ref{scary:alg} is analysed, the same problem is present in \citep{li2020optimal} and \citep{kovalev2022optimal_1}. This difficulty appears to arise in a decentralised setup, so the question of how to make the lower bound correct, how to interpret it and what would be the optimal primal algorithm in the case of static and time-varying network remains open. 
\subsection{Nonconvex Case}
In the nonconvex case, we use the same assumptions that for Algorithm~\ref{alg:yup}.
\begin{theorem}
\label{thm: lower-bound-sq-norm}
    For any $L > 0$, $m \geq 3$, there exists a set $\{F_i\}_{i=1}^n$ which satisfy \cref{assum:smoothness_F} and \cref{assum:smoothness_F_average}, and a sequence of matrices $\{\mW(k)\}_{k=0}^{\infty}$ which satisfy \cref{assum:gossip_matrix_sequence}, such that for any output $\hat{x}^N$ of any first-order decentralized algorithm after $N$ communications and $K$ local computations we get:
    \begin{align*}
        \E {\norm{\nabla F(\hat{x}^N)}^2} &= \Omega\left(\frac{\chi L\Delta}{ N}\right),\\
        \E {\norm{\nabla F(\hat{x}^N)}^2} &= \Omega\left(\frac{\sqrt{n}\Delta\hat{L}}{K}\right).
    \end{align*}
\end{theorem}
\begin{proof}
    See \cref{proof_lower}.
\end{proof}
\begin{corollary}
\label{cor: lower-bound-required-num}
    In the setting of \cref{thm: lower-bound-sq-norm}, the number of communication rounds $N_c$ and local oracle calls $N_s$ required to reach $\epsilon$-accuracy ($\E {\norm{\nabla F(\hat{x}^N)}^2} \leq \epsilon^2$) is lower bounded as
    \begin{align*}
        N_s &= \Omega\left(n + \frac{\sqrt{n}\Delta\hat{L}}{\epsilon^2}\right); \qquad 
        N_c = \Omega\left(\frac{\chi L\Delta}{\epsilon^2}\right),
    \end{align*}
    respectively.
\end{corollary}
\begin{remark}
    The lower bound for oracle calls $N_s$ is obtained the following way. From Theorem~\ref{thm: lower-bound-sq-norm} we get $N_s = \Omega(\sqrt{n}\Delta\hat L/\eps^2)$. Additionally, in \citep{li2021page} it was shown that $N_s = \Omega(n)$ even for non-distributed optimization. Consequently, we have 
    \begin{align*}
    N_s = \Omega\cbraces{\max\cbraces{n, \frac{\sqrt{n}\Delta\hat L}{\eps^2}}} = \Omega\cbraces{n + \frac{\sqrt{n}\Delta\hat L}{\eps^2}}.
    \end{align*}
\end{remark}
\begin{remark}
    Since one of the main ideas of the proof of \cref{thm: lower-bound-sq-norm} is the selection of a special sequence of time-varying graphs, that is why we get an estimate on the number of communications $\sim\chi$. But, as has been shown in some papers (e.g., \citep{yuan2022revisiting}), a lower bound on the number of communications for decentralized optimization on static graphs is $\sim\sqrt{\chi}$. Applying the same topology to our proof and taking into account \cref{upper-static}, we can conclude that \textsc{GT-PAGE} with Chebyshev acceleration is optimal for the case of static graphs as well.
\end{remark}
\begin{minipage}{\textwidth}
	\begin{figure}[H]
        \centering
		\includegraphics[width=\textwidth]{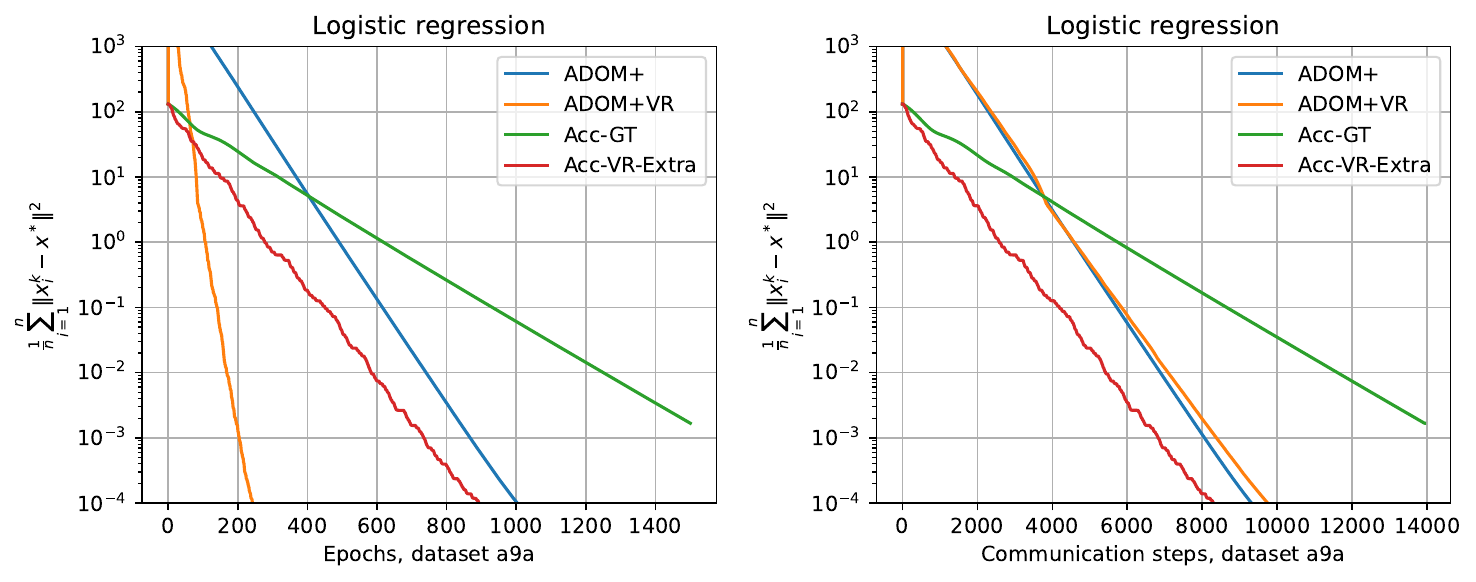}\\
		\centering (a) a9a
		\includegraphics[width=\textwidth]{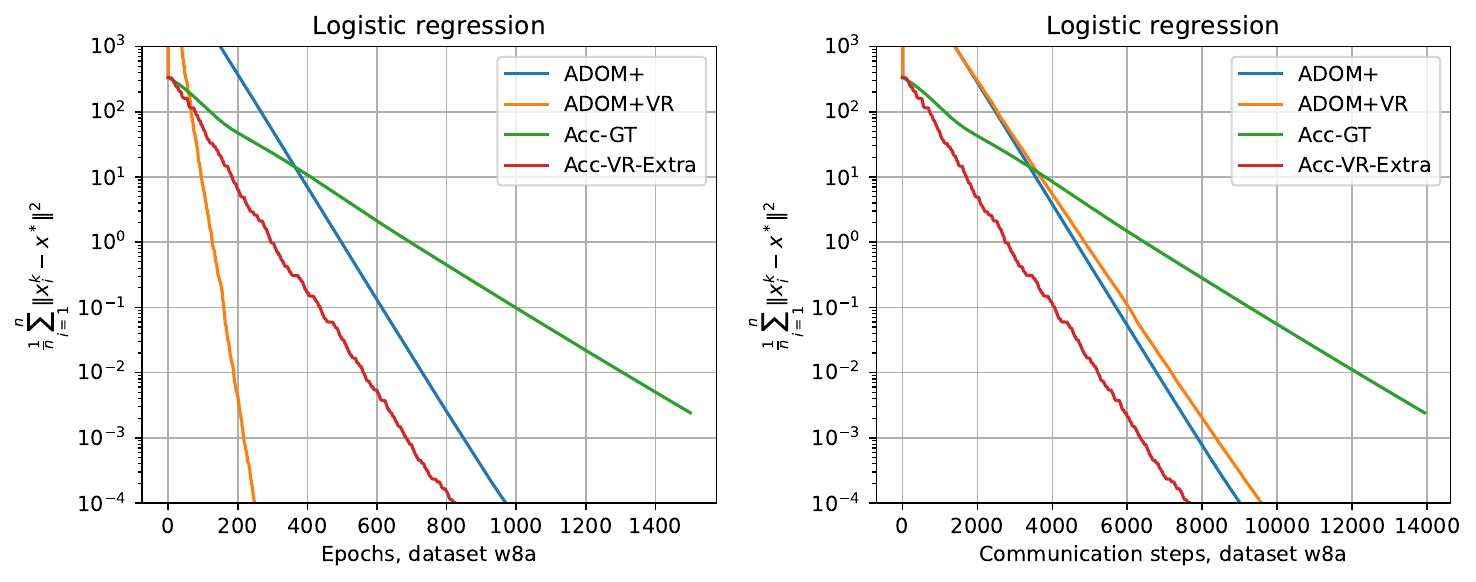}\\
		\centering (b) w8a
		\vspace{0.1cm}
		\caption{Comparison of communication and oracle complexities of Algorithm \ref{scary:alg} (ADOM+VR), ADOM+, Accelerated-GT (Acc-GT) and Accelerated-VR-Extra (Acc-VR-Extra) on logistic regression problem on LibSVM datasets.}
		\label{fig:strong}
	\end{figure}
\end{minipage}
\section{Numerical experiments}\label{sec:experiments}
In this section, we present numerical experiments comparing the proposed methods of this paper with state-of-the-art methods for both strongly convex and nonconvex problems.
\subsection{Setup}
\textbf{Datasets.} We utilize LibSVM \citep{chang2011libsvm} datasets in our experiments: a9a and w8a. Each dataset in an individual experiment is randomly distributed among the agents in the communication network.\\
\textbf{Topology.} We consider a random geometric graph with 50 vertices as the time-varying structure of the network.\\
\textbf{Loss function.} As an objective functions we choose logistic loss with $l_2$-regularization and non-linear least squares loss for strongly convex and nonconvex problems respectively.\\
\textbf{Optimization methods.} For our experiments we implemented proposed algorithms (\cref{scary:alg} and \cref{alg:yup}) with other existing approaches (see \cref{fig:strong} and \cref{fig:nonconvex}).
\subsection{Results}

\begin{minipage}{\textwidth}
\begin{figure}[H]
	\includegraphics[width=\textwidth]{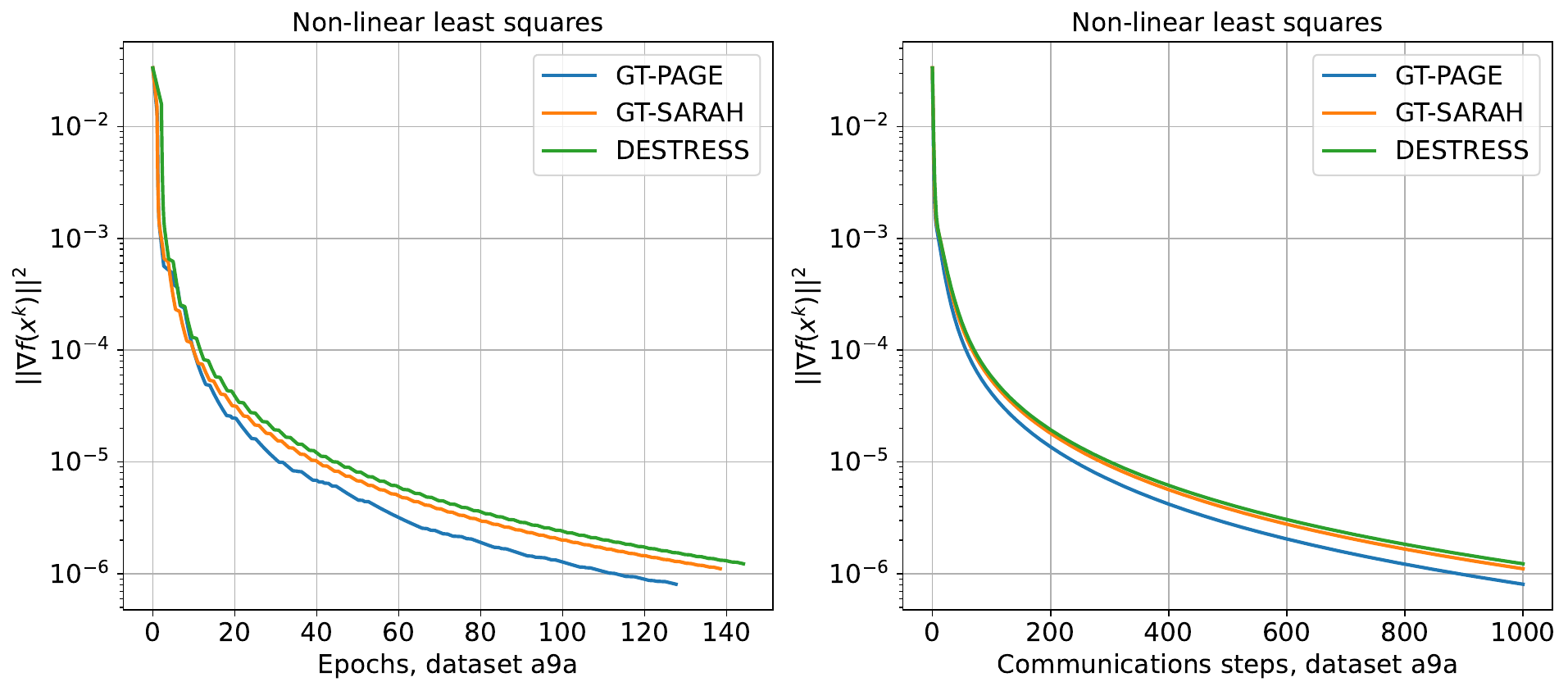}\\
	\centering (a) a9a \\
	\includegraphics[width=\textwidth]{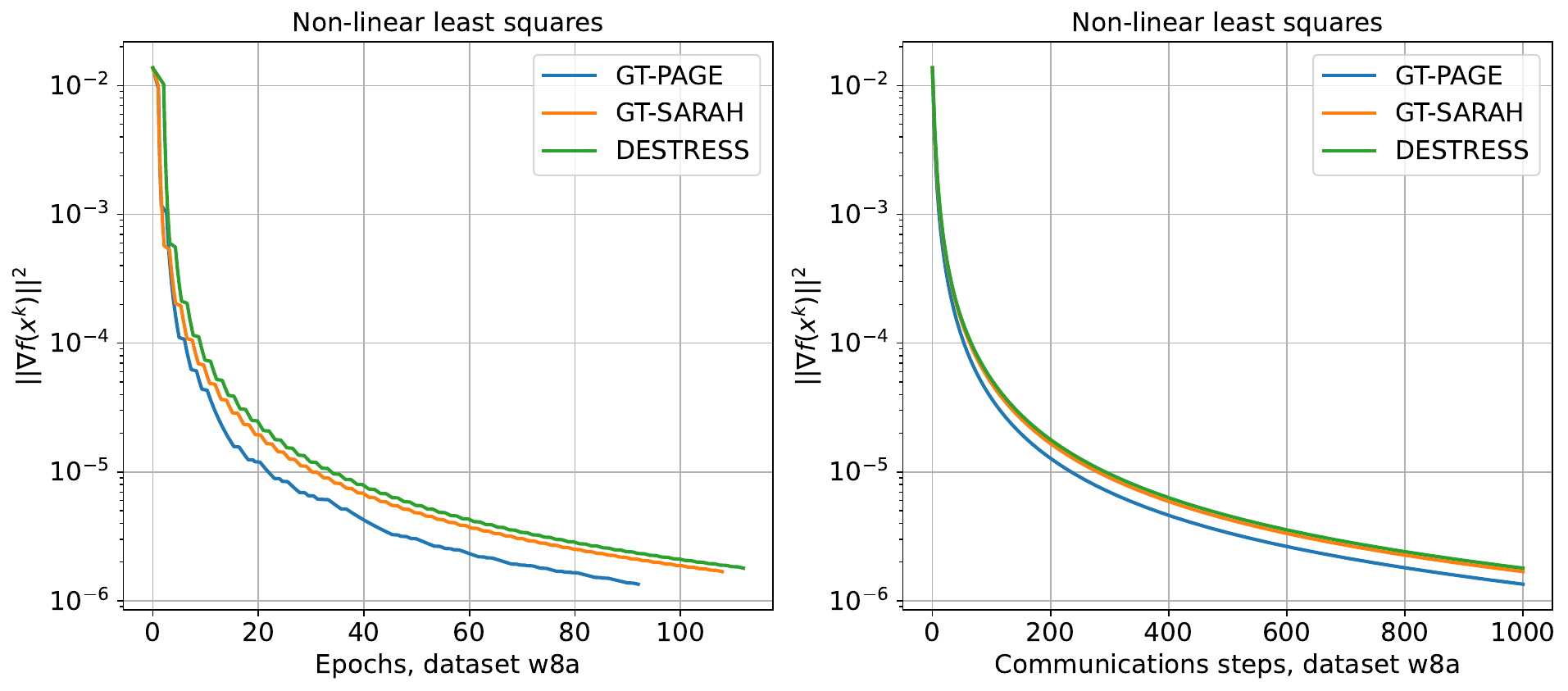} \\
	\centering (b) w8a
	\caption{Comparison of communication and oracle complexities of Algorithm \ref{alg:yup} (GT-PAGE), GT-SARAH and DESTRESS.}
	\label{fig:nonconvex}
\end{figure}
\end{minipage}

Experimental outcomes are shown in \cref{fig:strong} and \cref{fig:nonconvex}. Regarding the logistic regression problem, ADOM+VR outperforms other methods with respect to the number of epochs, i.e. the number of oracle calls. However, there is no gain in communication complexity compared to state-of-the-art approaches. At the same time, for the non-linear least squares problem, GT-PAGE behaves better with respect to other methods, but it does not demonstrate a strong superiority.
\bibliographystyle{apalike}
\bibliography{ref}

\newpage
\appendix
\onecolumn
{\centering \large \bf Supplementary Material}

We now establish the convergence rate of Algorithm~\ref{scary:alg}. This proof is for the most part a modified analysis of the ADOM+ algorithm with the addition of techniques corresponding to variance reduction setting. The parts not affected by the change were kept for the sake of completeness.

\section{Proof of Theorem~\ref{thm:adom+_conv}}\label{appendix:A}

By $\bg_F(x,y)$ we denote Bregman distance $\bg_F(x,y)\eqdef F(x) - F(y) - \<\g F(y),x-y>$.

By $\bb_F(x,y)$ we denote $\bb_F(x,y)\eqdef \bg_F(x,y) - \frac{\nu}{2}\sqn{x-y}$. \label{def:bb}

\begin{lemma}\label{eq:grad_upper}
	\begin{equation}
		\begin{aligned}
			\Ei{S^k}{\sqn{\nabla^k - \g F(x_g^k)}}
			&\leq
			\frac{2\oL}{b}\left(\bb_F(\omega^k, x^*) - \bb_F(x_g^k, x^*)\right)
			\\&-
			\frac{2\oL}{b}\<\g F(x_g^k) - \g F(x^*) -\nu x_g^k + \nu x^*, \omega^k - x_g^k>.
		\end{aligned}
	\end{equation}
	\begin{proof}
		Firstly note, that if $g^k =\g f_i(x^k_g) - \g f_i (\omega^k) + \g f_i(\omega^k)$, then
		\begin{equation}\label{scary:bregman_upper_lemma}
			\begin{aligned}
				\Ei{i}{\left\| g^k - \nabla f(x^k_g) \right\|^2} &= 
				\Ei{i}{\left\| \nabla f_i(x^k_g) - \nabla f_i(w^k) - \Ei{i}{\nabla f_i(x^k_g) - \nabla f_i(w^k)} \right\|^2} \\
				&\leq
				\Ei{i}{\left\| \nabla f_i(x^k) - \nabla f_i(w^k) \right\|^2} \\
				&\leq
				2\oL\left( f(w^k) - f(x^k) - \left\langle \nabla f(x^k), w^k - x^k \right\rangle \right).
			\end{aligned}
		\end{equation}
        Let us describe the main term
        \begin{align*}
            &\Ei{S^k_i}{\sqN{\left(\nabla^k\right)_i - \g F_i(x^k_{g,i})}} = \Ei{S^k_i}{\sqN{\frac{1}{b}\sum_{j \in S^k_i} \frac{1}{np_{ij}}\big[
			\g f_{ij}(x^k_{g, i}) - \g f_{ij} (\omega^k_i)\big] + \g F_i(\omega^k_i) - \g F_i(x^k_{g,i})}}
            \\&\overset{(1)}{=}
            \frac{1}{b}\Ei{j}{\sqN{\frac{1}{np_{ij}}\big[
			\g f_{ij}(x^k_{g, i}) - \g f_{ij} (\omega^k_i)\big] + \g F_i(\omega^k_i) - \g F_i(x^k_{g,i})}}
            \\&=
            \frac{1}{b}\Ei{j}{\sqN{\frac{1}{np_{ij}}\big[\left(
			\g f_{ij}(x^k_{g, i}) - \g f_{ij}(x^*)\right) - \left(\g f_{ij} (\omega^k_i) - \g f_{ij}(x^*)\right)\big] + \g F_i(\omega^k_i) - \g F_i(x^k_{g,i})}}
            \\&\overset{(2)}{\leq}
            \frac{1}{b}\Ei{j}{\sqN{\frac{1}{np_{ij}}\big[\left(
			\g f_{ij}(x^k_{g, i}) - \g f_{ij}(x^*) - \nu x^k_{g,i}+\nu x^*\right) - \left(\g f_{ij} (\omega^k_i) - \g f_{ij}(x^*) - \nu \omega_i+\nu x^*\right)\big]}}
            \\&\overset{(3)}{\leq}
            \sum_{j=1}^n \frac{p_{ij}}{b}\frac{2L_{ij}}{n^2p_{ij}^2}\bb_{f_{ij}}(x_{g,i}^k, x^*)
           \overset{(4)}{=}
            \frac{2\oL_i}{b}\left(\bb_{F_i}(\omega^k_i, x^*) - \bb_{F_i}(x_{g,i}^k, x^*) - \<\g\bb_{F_i}(x_{g,i}^k, x^*), \omega^k_i - x_{g,i}^k> \right),
        \end{align*}
		where
        \begin{itemize}
            \item[(1)] is due to independency of $(\xi^1_i, \xi^2_i,\ldots, \xi^b_i)$,
            \item[(2)] follows from the inequality $\E{\sqn{\xi}} \leq \E{\sqn{\xi + c}}$ if $\E{\xi}=0$ and $c$ is constant,
            \item[(3)] follows from \eqref{scary:bregman_upper_lemma} inequality,
            \item[(4)] follows from $p_{ij}=L_{ij}/(n\oL_i)$ definition.
        \end{itemize}
        The required inequality is the simple consequence of the previous statement.
	\end{proof}
\end{lemma}
Further we will assume that the basis of the expectation is clear from the context.
\begin{lemma}
	Let $\tau_2$ be defined as follows:
	\begin{equation}\label{scary:tau2}
		\tau_2 = \min\left\{\frac{1}{2}, \max\left\{1, \frac{\sqrt{n}}{b}\right\}\sqrt{\frac{\mu}{L}}\right\}.
	\end{equation}
	Let $\tau_1$ be defined as follows:
	\begin{equation}\label{scary:tau1}
		\tau_1 = (1-\tau_0)(1/\tau_2 + 1/2)^{-1}.
	\end{equation}
	Let $\tau_0$ be defined as follows:
	\begin{equation}\label{scary:tau0}
		\tau_0 = \frac{\oL}{2L b}.
	\end{equation}
	Let $\eta$ be defined as follows:
	\begin{equation}\label{scary:eta}
		\eta = \left(L\left(\tau_2+\frac{2\tau_1}{1-\tau_1}\right)\right)^{-1}.
	\end{equation}
	Let $\alpha$ be defined as follows:
	\begin{equation}\label{scary:alpha}
		\alpha = \mu/2.
	\end{equation}
	Let $\nu$ be defined as follows:
	\begin{equation}\label{scary:nu}
		\nu = \mu/2.
	\end{equation}
	Let $\Psi_x^k$ be defined as follows:
	\begin{equation}\label{scary:Psi_x}
		\Psi_x^k = \left(\frac{1}{\eta} + \alpha\right)\sqn{x^{k} - x^*} + \frac{2}{\tau_2}\left(\bg_f(x_f^{k},x^*)-\frac{\nu}{2}\sqn{x_f^{k} - x^*} \right)
	\end{equation}
	Then the following inequality holds:
	\begin{equation}\label{scary:eq:x}
		\begin{split}
			\Psi_x^{k+1} &\leq \left(1 - \frac{1}{20}\min\left\{\sqrt{\frac{\mu}{L}}, b\sqrt{\frac{\mu}{nL}}\right\}\right)\Psi_x^k
    		+
    		2\E{\< y^{k+1} - y^*,x^{k+1} - x^*>}
    		\\&+
            \frac{\oL}{Lb}\left(\frac{1}{\tau_1} - 1\right)\left(\bb_F(\omega^k,x^*)-\bb_F(x_g^k,x^*)\right)
            -
            \bb_F(x_g^k,x^*)
            -
            \frac{1}{2}\bb_F(x_f^k,x^*)
    		\\&+
    		\frac{\oL}{Lb}\<\g F(x_g^k) - \g F(x^*) -\nu x_g^k + \nu x^*, \omega^k - x_g^k>.
		\end{split}
	\end{equation}
\end{lemma}
\begin{proof}
	\begin{align*}
		\frac{1}{\eta}\sqn{x^{k+1}  - x^*}
		&=
		\frac{1}{\eta}\sqn{x^k - x^*}+\frac{2}{\eta}\<x^{k+1} - x^k,x^{k+1}- x^*> - \frac{1}{\eta}\sqn{x^{k+1} - x^k}.
	\end{align*}
	Using Line~\ref{scary:line:x:2} of Algorithm~\ref{scary:alg} we get
	\begin{align*}
		\frac{1}{\eta}\sqn{x^{k+1}  - x^*}
		&=
		\frac{1}{\eta}\sqn{x^k - x^*}
		+
		2\alpha\<x_g^k - x^{k+1},x^{k+1}- x^*>
		\\&-
		2\<\nabla^k - \nu x_g^k - y^{k+1},x^{k+1} - x^*>
		-
		\frac{1}{\eta}\sqn{x^{k+1} - x^k}
		\\&=
		\frac{1}{\eta}\sqn{x^k - x^*}
		+
		2\alpha\<x_g^k - x^*- x^{k+1} + x^*,x^{k+1}- x^*>
		\\&-
		2\<\nabla^k - \nu \x_g^k - y^{k+1},x^{k+1} - x^*>
		-
		\frac{1}{\eta}\sqn{x^{k+1} - x^k}
		\\&\leq
		\frac{1}{\eta}\sqn{x^k - x^*}
		-
		\alpha\sqn{x^{k+1} - x^*} + \alpha\sqn{x_g^k - x^*}
		-
		2\<\nabla^k - \nu x_g^k - y^{k+1},x^{k+1} - x^*>
		\\&-
		\frac{1}{\eta}\sqn{x^{k+1} - x^k}.
	\end{align*}
	Using optimality condition \eqref{opt:x} we get
	\begin{align*}
		\frac{1}{\eta}\sqn{x^{k+1}  - x^*}
		&\leq
		\frac{1}{\eta}\sqn{x^k - x^*}
		-
		\alpha\sqn{x^{k+1} - x^*} + \alpha\sqn{x_g^k - x^*}
		-
		\frac{1}{\eta}\sqn{x^{k+1} - x^k}
		\\&
		-2\<\g F(x_g^k) - \g F(x^*),x^{k+1} - x^*>
		+
		2\nu\< x_g^k - x^*,x^{k+1} - x^*>
		+
		2\< y^{k+1} - y^*,x^{k+1} - x^*>
		\\&
		-2\<\nabla^k - \g F(x_g^k),x^{k+1} - x^*>.
	\end{align*}
	Using Line~\ref{scary:line:x:3} of Algorithm~\ref{scary:alg} we get
	\begin{align*}
		\frac{1}{\eta}\sqn{x^{k+1}  - x^*}
		&\leq
		\frac{1}{\eta}\sqn{x^k - x^*}
		-
		\alpha\sqn{x^{k+1} - x^*} + \alpha\sqn{x_g^k - x^*}
		-
		\frac{1}{\eta\tau_2^2}\sqn{x_f^{k+1} - x_g^k}
		\\&
		-2\<\g F(x_g^k) - \g F(x^*),x^k - x^*>
		+
		2\nu\< x_g^k - x^*,x^k - x^*>
		+
		2\< y^{k+1} - y^*,x^{k+1} - x^*>
		\\&-
		\frac{2}{\tau_2}\<\g F(x_g^k) - \g F(x^*),x_f^{k+1} - x_g^k>
		+
		\frac{2\nu}{\tau_2}\< x_g^k - x^*,x_f^{k+1} - x_g^k>
		\\&-
		2\<\nabla^k - \g F(x_g^k) - \nu(\x_g^k-x_g^k),x^{k+1} - x^*>
		\\&=
		\frac{1}{\eta}\sqn{x^k - x^*}
		-
		\alpha\sqn{x^{k+1} - x^*} + \alpha\sqn{x_g^k - x^*}
		-
		\frac{1}{\eta\tau_2^2}\sqn{x_f^{k+1} - x_g^k}
		\\&
		-2\<\g F(x_g^k) - \g F(x^*),x^k - x^*>
		+
		2\nu\< x_g^k - x^*,x^k - x^*>
		+
		2\< y^{k+1} - y^*,x^{k+1} - x^*>
		\\&-
		\frac{2}{\tau_2}\<\g F(x_g^k) - \g F(x^*),x_f^{k+1} - x_g^k>
		+
		\frac{\nu}{\tau_2}\left(\sqn{x_f^{k+1} - x^*} - \sqn{x_g^k - x^*}-\sqn{x_f^{k+1} - x_g^k}\right)
		\\&
		-
		2\<\nabla^k - \g F(x_g^k),x^{k+1} - x^*>.
	\end{align*}
	Using the $L$-smoothness property of $\bg_F(x, x^*)$ with respect to $x$, which is derived from the $L$-smoothness of $F(x)$, we obtain
	\begin{align*}
		\frac{1}{\eta}\sqn{x^{k+1}  - x^*}
		&\leq
		\frac{1}{\eta}\sqn{x^k - x^*}
		-
		\alpha\sqn{x^{k+1} - x^*} + \alpha\sqn{x_g^k - x^*}
		-
		\frac{1}{\eta\tau_2^2}\sqn{x_f^{k+1} - x_g^k}
		\\&
		-2\<\g F(x_g^k) - \g F(x^*),x^k - x^*>
		+
		2\nu\< x_g^k - x^*,x^k - x^*>
		+
		2\< y^{k+1} - y^*,x^{k+1} - x^*>
		\\&-
		\frac{2}{\tau_2}\<\g F(x_g^k) - \g F(x^*),x_f^{k+1} - x_g^k>
		+
		\frac{\nu}{\tau_2}\left(\sqn{x_f^{k+1} - x^*} - \sqn{x_g^k - x^*}-\sqn{x_f^{k+1} - x_g^k}\right)
		\\&-
		2\<\nabla^k - \g F(x_g^k),x^{k+1} - x^*>
		\\&\leq
		\frac{1}{\eta}\sqn{x^k - x^*}
		-
		\alpha\sqn{x^{k+1} - x^*} + \alpha\sqn{x_g^k - x^*}
		-
		\frac{1}{\eta\tau_2^2}\sqn{x_f^{k+1} - x_g^k}
		\\&
		-2\<\g F(x_g^k) - \g F(x^*),x^k - x^*>
		+
		2\nu\< x_g^k - x^*,x^k - x^*>
		+
		2\< y^{k+1} - y^*,x^{k+1} - x^*>
		\\&-
		\frac{2}{\tau_2}\left(\bg_f(x_f^{k+1},x^*) - \bg_f(x_g^k,x^*) - \frac{L}{2}\sqn{x_f^{k+1} - x_g^k}\right)
		\\&+
		\frac{\nu}{\tau_2}\left(\sqn{x_f^{k+1} - x^*} - \sqn{x_g^k - x^*}-\sqn{x_f^{k+1} - x_g^k}\right)
		-
		2\<\nabla^k - \g F(x_g^k),x^{k+1} - x^*>
		\\&=
		\frac{1}{\eta}\sqn{x^k - x^*}
		-
		\alpha\sqn{x^{k+1} - x^*} + \alpha\sqn{x_g^k - x^*}
		+
		\left(\frac{L - \nu}{\tau_2}-\frac{1}{\eta\tau_2^2}\right)
		\sqn{x_f^{k+1} - x_g^k}
		\\&
		-2\<\g F(x_g^k) - \g F(x^*),x^k - x^*>
		+
		2\nu\< x_g^k - x^*,x^k - x^*>
		+
		2\< y^{k+1} - y^*,x^{k+1} - x^*>
		\\&-
		\frac{2}{\tau_2}\left(\bg_f(x_f^{k+1},x^*) - \bg_f(x_g^k,x^*)\right)
		+
		\frac{\nu}{\tau_2}\left(\sqn{x_f^{k+1} - x^*} - \sqn{x_g^k - x^*}\right)
		\\&-
		2\<\nabla^k - \g F(x_g^k),x^{k+1} - x^*>
	\end{align*}
	Using Line~\ref{scary:line:x:1} of Algorithm~\ref{scary:alg} we get
	\begin{align*}
		\frac{1}{\eta}\sqn{x^{k+1}  - x^*}
		&\leq
		\frac{1}{\eta}\sqn{x^k - x^*}
		-
		\alpha\sqn{x^{k+1} - x^*} + \alpha\sqn{x_g^k - x^*}
		+
		\left(\frac{L - \nu}{\tau_2}-\frac{1}{\eta\tau_2^2}\right)
		\sqn{x_f^{k+1} - x_g^k}
		\\&
		-2\<\g F(x_g^k) - \g F(x^*),x_g^k - x^*>
		+
		2\nu\sqn{x_g^k - x^*}
		+
		\frac{2(1-\tau_1-\tau_0)}{\tau_1}\<\g F(x_g^k) - \g F(x^*),x_f^k - x_g^k>
		\\&+
		\frac{2\tau_0}{\tau_1}\<\g F(x_g^k) - \g F(x^*),\omega^k - x_g^k>
		+
		\frac{2\nu(1-\tau_1-\tau_0)}{\tau_1}\<x_g^k - x_f^k,x_g^k - x^*>
		+
		\frac{2\nu\tau_0}{\tau_1}\<x_g^k - \omega^k,x_g^k - x^*>
		\\&+
		2\< y^{k+1} - y^*,x^{k+1} - x^*>
		-
		\frac{2}{\tau_2}\left(\bg_f(x_f^{k+1},x^*) - \bg_f(x_g^k,x^*)\right)
		+
		\frac{\nu}{\tau_2}\left(\sqn{x_f^{k+1} - x^*} - \sqn{x_g^k - x^*}\right)
		\\&-
		2\<\nabla^k - \g F(x_g^k),x^{k+1} - x^*>
		\\&=
		\frac{1}{\eta}\sqn{x^k - x^*}
		-
		\alpha\sqn{x^{k+1} - x^*} + \alpha\sqn{x_g^k - x^*}
		+
		\left(\frac{L - \nu}{\tau_2}-\frac{1}{\eta\tau_2^2}\right)
		\sqn{x_f^{k+1} - x_g^k}
		\\&
		-2\<\g F(x_g^k) - \g F(x^*),x_g^k - x^*>
		+
		2\nu\sqn{x_g^k - x^*}
		+
		\frac{2(1-\tau_1-\tau_0)}{\tau_1}\<\g F(x_g^k) - \g F(x^*),x_f^k - x_g^k>
		\\&+
		\frac{2\tau_0}{\tau_1}\<\g F(x_g^k) - \g F(x^*),\omega^k - x_g^k>
		+
		\frac{\nu(1-\tau_1-\tau_0)}{\tau_1}\left(\sqn{x_g^k- x_f^k} + \sqn{x_g^k - x^*} - \sqn{x_f^k - x^*}\right)
		\\&+
		\frac{2\nu\tau_0}{\tau_1}\<x_g^k - \omega^k,x_g^k - x^*>
		+
		2\< y^{k+1} - y^*,x^{k+1} - x^*>
		-
		\frac{2}{\tau_2}\left(\bg_f(x_f^{k+1},x^*) - \bg_f(x_g^k,x^*)\right)
		\\&+
		\frac{\nu}{\tau_2}\left(\sqn{x_f^{k+1} - x^*} - \sqn{x_g^k - x^*}\right)
		-
		2\<\nabla^k - \g F(x_g^k),x^{k+1} - x^*>.
	\end{align*}
	By applying $\mu$-strong convexity of $\bg_F(x, x^*)$ in $x$, following from $\mu$-strong convexity of $F(x)$, we obtain
	\begin{align*}
		\frac{1}{\eta}\sqn{x^{k+1}  - x^*}
		&\leq
		\frac{1}{\eta}\sqn{x^k - x^*}
		-
		\alpha\sqn{x^{k+1} - x^*} + \alpha\sqn{x_g^k - x^*}
		+
		\left(\frac{L - \nu}{\tau_2}-\frac{1}{\eta\tau_2^2}\right)
		\sqn{x_f^{k+1} - x_g^k}
		\\&
		-2\bg_F(x_g^k,x^*) - \mu\sqn{x_g^k - x^*}
		+
		2\nu\sqn{x_g^k - x^*}
		\\&+
		\frac{2(1-\tau_1-\tau_0)}{\tau_1}\left(\bg_F(x_f^k,x^*) - \bg_F(x_g^k,x^*) - \frac{\mu}{2}\sqn{x_f^k - x_g^k}\right)
		+
		\frac{2\tau_0}{\tau_1}\<\g F(x_g^k) - \g F(x^*),\omega^k - x_g^k>
		\\&+
		\frac{\nu(1-\tau_1-\tau_0)}{\tau_1}\left(\sqn{x_g^k- x_f^k} + \sqn{x_g^k - x^*} - \sqn{x_f^k - x^*}\right)
		+
		\frac{2\nu\tau_0}{\tau_1}\<x_g^k - \omega^k,x_g^k - x^*>
		\\&+
		2\< y^{k+1} - y^*,x^{k+1} - x^*>
		-
		\frac{2}{\tau_2}\left(\bg_f(x_f^{k+1},x^*) - \bg_f(x_g^k,x^*)\right)
		\\&+
		\frac{\nu}{\tau_2}\left(\sqn{x_f^{k+1} - x^*} - \sqn{x_g^k - x^*}\right)
		-
		2\<\nabla^k - \g F(x_g^k),x^{k+1} - x^*>.
		\\&=
		\frac{1}{\eta}\sqn{x^k - x^*}
		-
		\alpha\sqn{x^{k+1} - x^*}
		+
		\frac{2(1-\tau_1-\tau_0)}{\tau_1}\left(\bg_F(x_f^k,x^*) - \frac{\nu}{2}\sqn{x_f^k - x^*}\right)
		\\&-
		\frac{2}{\tau_2}\left(\bg_f(x_f^{k+1},x^*)-\frac{\nu}{2}\sqn{x_f^{k+1} - x^*} \right)
		+
		2\< y^{k+1} - y^*,x^{k+1} - x^*>
		\\&+
		2\left(\frac{1}{\tau_2}-\frac{1-\tau_0}{\tau_1}\right)\bg_F(x_g^k,x^*)
		+
		\left(\alpha - \mu + \nu+\frac{(1-\tau_0)\nu}{\tau_1}-\frac{\nu}{\tau_2}\right)\sqn{x_g^k - x^*}
		\\&+
		\left(\frac{L - \nu}{\tau_2}-\frac{1}{\eta\tau_2^2}\right)
		\sqn{x_f^{k+1} - x_g^k}
		+
		\frac{(1-\tau_1-\tau_0)(\nu-\mu)}{\tau_1}\sqn{x_f^k - x_g^k}
		\\&+
		\frac{2\tau_0}{\tau_1}\<\g F(x_g^k) - \g F(x^*),\omega^k - x_g^k>
		+
		\frac{2\nu\tau_0}{\tau_1}\<x_g^k - \omega^k,x_g^k - x^*>
		\\&-
		2\<\nabla^k - \g F(x_g^k),x^{k+1} - x^*>.
	\end{align*}
	Utilizing $\eta$ as defined in \eqref{scary:eta}, $\tau_1$ as defined in \eqref{scary:tau1}, and considering that $\nu < \mu$, we derive
	\begin{align*}
		\frac{1}{\eta}\sqn{x^{k+1}  - x^*}
		&\leq
		\frac{1}{\eta}\sqn{x^k - x^*}
		-
		\alpha\sqn{x^{k+1} - x^*}
		+
		\frac{2(1-\tau_2/2)}{\tau_2}\bb_F(x_f^k,x^*)
		\\&-
		\frac{2}{\tau_2}\bb_F(x_f^{k+1},x^*)
		+
		2\< y^{k+1} - y^*,x^{k+1} - x^*>
		\\&-\bg_F(x_g^k,x^*)
		+
		\left(\alpha - \mu + \frac{3\nu}{2}\right)\sqn{x_g^k - x^*}
		-
		\frac{2L\tau_1}{\tau_2^2(1-\tau_1)}\sqn{x_f^{k+1} - x_g^k}
		\\&+
		\frac{2\tau_0}{\tau_1}\<\left(\g F(x_g^k) - \nu x_g^k\right) - \left(\g F(x^*) - \nu x^*\right),\omega^k - x_g^k>
		\\&-
		2\<\nabla^k - \g F(x_g^k),x^{k+1} - x^*>.
	\end{align*}
	Using $\alpha$ defined by \eqref{scary:alpha} and $\nu$ defined by \eqref{scary:nu} we get
	\begin{align*}
		\frac{1}{\eta}\sqn{x^{k+1}  - x^*}
		&\leq
		\frac{1}{\eta}\sqn{x^k - x^*}
		-
		\alpha\sqn{x^{k+1} - x^*}
		+
		\frac{2(1-\tau_2/2)}{\tau_2}\bb_F(x_f^k,x^*)
		\\&-
		\frac{2}{\tau_2}\bb_F(x_f^{k+1},x^*)
		+
		2\< y^{k+1} - y^*,x^{k+1} - x^*>
		\\&-
		\left(\bg_F(x_g^k,x^*) - \frac{\nu}{2}\sqn{x_g^k - x^*}\right)
		-
		\frac{2L\tau_1}{\tau_2^2(1-\tau_1)}\sqn{x_f^{k+1} - x_g^k}
		\\&+
		\frac{2\tau_0}{\tau_1}\<\left(\g F(x_g^k) - \nu x_g^k\right) - \left(\g F(x^*) - \nu x^*\right),\omega^k - x_g^k>
		\\&-
		2\<\nabla^k - \g F(x_g^k),x^{k+1} - x^*>.
	\end{align*}
	Taking the expectation over $i$ at the $k$th step, using that $x^k-x^*$ is independent of $i$ and that $\E{\nabla^k - \g F(x_g^k)} = 0$ we get
	\begin{align*}
		\frac{1}{\eta}\E{\sqn{x^{k+1}  - x^*}}
		&\leq
		\frac{1}{\eta}\sqn{x^k - x^*}
		-
		\alpha\E{\sqn{x^{k+1} - x^*}}
		+
		\frac{2(1-\tau_2/2)}{\tau_2}\bb_F(x_f^k,x^*)
		\\&-
		\frac{2}{\tau_2}\E{\bb_F(x_f^{k+1},x^*)}
		+
		2\E{\< y^{k+1} - y^*,x^{k+1} - x^*>}
		\\&-
		\bb_F(x_g^k,x^*)
		-
		\frac{2L\tau_1}{\tau_2^2(1-\tau_1)}\E{\sqn{x_f^{k+1} - x_g^k}}
		\\&+
		\frac{2\tau_0}{\tau_1}\<\left(\g F(x_g^k) - \nu x_g^k\right) - \left(\g F(x^*) - \nu x^*\right),\omega^k - x_g^k>
		\\&-
		2\E{\<\nabla^k - \g F(x_g^k),x^{k+1} - x^k>}.
	\end{align*}
	Using Line~\ref{scary:line:x:3} of Algorithm~\ref{scary:alg} and the Cauchy–Schwarz inequality for $\<\nabla^k - \g F(x_g^k),x_f^{k+1} - x_g^k>$ we get
	\begin{align*}
		\frac{1}{\eta}\E{\sqn{x^{k+1}  - x^*}}
		&\leq
		\frac{1}{\eta}\sqn{x^k - x^*}
		-
		\alpha\E{\sqn{x^{k+1} - x^*}}
		+
		\frac{2(1-\tau_2/2)}{\tau_2}\bb_F(x_f^k,x^*)
		\\&-
		\frac{2}{\tau_2}\E{\bb_F(x_f^{k+1},x^*)}
		+
		2\E{\< y^{k+1} - y^*,x^{k+1} - x^*>}
		\\&-
		\bb_F(x_g^k,x^*)
		+
		\frac{1-\tau_1}{2L\tau_1}\E{\sqn{\nabla^k - \g F(x_g^k)}}
		\\&+
		\frac{2\tau_0}{\tau_1}\<\left(\g F(x_g^k) - \nu x_g^k\right) - \left(\g F(x^*) - \nu x^*\right),\omega^k - x_g^k>.
	\end{align*}
	Using lemma~\ref{eq:grad_upper} and $\tau_0$ definition \eqref{scary:tau0} we get
	\begin{align*}
		\frac{1}{\eta}\E{\sqn{x^{k+1}  - x^*}}
		&\leq
		\frac{1}{\eta}\sqn{x^k - x^*}
		-
		\alpha\E{\sqn{x^{k+1} - x^*}}
		+
		\frac{2(1-\tau_2/2)}{\tau_2}\bb_F(x_f^k,x^*)
		\\&-
		\frac{2}{\tau_2}\E{\bb_F(x_f^{k+1},x^*)}
		+
		2\E{\< y^{k+1} - y^*,x^{k+1} - x^*>}
		-
		\bb_F(x_g^k,x^*)
		\\&+\frac{\oL}{Lb}\left(\frac{1}{\tau_1}-1\right)\left(\bb_F(\omega^k, x^*) - \bb_F(x_g^k, x^*) - \<\g F(x_g^k) - \g F(x^*) -\nu x_g^k + \nu x^*, \omega^k - x_g^k>\right)
		\\&+
		\frac{2\tau_0}{\tau_1}\<\left(\g F(x_g^k) - \nu x_g^k\right) - \left(\g F(x^*) - \nu x^*\right),\omega^k - x_g^k>
		\\&=
		\frac{1}{\eta}\sqn{x^k - x^*}
		-
		\alpha\E{\sqn{x^{k+1} - x^*}}
		+
		\frac{2(1-\tau_2/2)}{\tau_2}\bb_F(x_f^k,x^*)
		-
		\frac{2}{\tau_2}\E{\bb_F(x_f^{k+1},x^*)}
		\\&+
		2\E{\< y^{k+1} - y^*,x^{k+1} - x^*>}
        +\frac{\oL}{Lb}\left(\frac{1}{\tau_1} - 1\right)\left(\bb_F(\omega^k,x^*)-\bb_F(x_g^k,x^*)\right)
        -
        \bb_F(x_g^k,x^*)
		\\&+
		\frac{\oL}{Lb}\<\g F(x_g^k) - \g F(x^*) -\nu x_g^k + \nu x^*, \omega^k - x_g^k>.
	\end{align*}
	After rearranging and using $\Psi_x^k $ definition \eqref{scary:Psi_x} we get
	\begin{align*}
		\E{\Psi_x^{k+1}}
		&\leq
		\max\left\{1 - \tau_2/4, 1/(1+\eta\alpha)\right\}\Psi_x^k
		+
		2\E{\< y^{k+1} - y^*,x^{k+1} - x^*>}
		\\&+
        \frac{\oL}{Lb}\left(\frac{1}{\tau_1} - 1\right)\left(\bb_F(\omega^k,x^*)-\bb_F(x_g^k,x^*)\right)
        -
        \bb_F(x_g^k,x^*)
        -
        \frac{1}{2}\bb_F(x_f^k,x^*)
		\\&+
		\frac{\oL}{Lb}\<\g F(x_g^k) - \g F(x^*) -\nu x_g^k + \nu x^*, \omega^k - x_g^k>
		\\&\leq 
		\left(1 - \frac{1}{20}\min\left\{\sqrt{\frac{\mu}{L}}, b\sqrt{\frac{\mu}{nL}}\right\}\right)\Psi_x^k
		+
		2\E{\< y^{k+1} - y^*,x^{k+1} - x^*>}
		\\&+
        \frac{\oL}{Lb}\left(\frac{1}{\tau_1} - 1\right)\left(\bb_F(\omega^k,x^*)-\bb_F(x_g^k,x^*)\right)
        -
        \bb_F(x_g^k,x^*)
        -
        \frac{1}{2}\bb_F(x_f^k,x^*)
		\\&+
		\frac{\oL}{Lb}\<\g F(x_g^k) - \g F(x^*) -\nu x_g^k + \nu x^*, \omega^k - x_g^k>.
	\end{align*}
	
	The last inequality follows from $\eta$, $\alpha$, $\tau_0$, $\tau_1$, $\tau_2$ definitions \eqref{scary:eta}, \eqref{scary:alpha}, \eqref{scary:tau0}, \eqref{scary:tau1} and \eqref{scary:tau2}. Estimating the second term:
	\begin{align*}
		\frac{1}{1+\eta\alpha}
		&\leq
		1 - \frac{\eta\alpha}{2}
		\leq
		1 - \frac{\mu}{4}\left(L\left(\tau_2+\frac{2\tau_1}{1-\tau_1}\right)\right)^{-1}
		\leq
		1 - \frac{\mu}{4}\left(L\left(\tau_2+\frac{2\tau_2}{1-\tau_2}\right)\right)^{-1}
		\\&\leq
		1 - \frac{\mu}{4}\left(L\left(\tau_2+4\tau_2\right)\right)^{-1}
		=
		1 - \frac{\mu}{20L\tau_2}
		\leq
    	1 - \frac{1}{20\max\left\{1, \frac{\sqrt{n}}{b}\right\}}\sqrt{\frac{\mu}{L}}
        \leq
        1 - \frac{1}{20}\min\left\{\sqrt{\frac{\mu}{L}}, b\sqrt{\frac{\mu}{nL}}\right\}.
	\end{align*}
    Estimating the first term:
    \begin{align*}
        1-\tau_2/4 \leq 1-\min\left\{\frac{1}{8}, \frac{1}{4}\sqrt{\frac{\mu}{L}}\right\}.
    \end{align*}
\end{proof}

\begin{lemma}
	The following inequality holds:
	\begin{equation}
		\begin{split}\label{scary:eq:1}
			-\sqn{y^{k+1}- y^*}
			&\leq 
			\frac{(1-\sigma_1)}{\sigma_1}\sqn{y_f^k - y^*}
			- \frac{1}{\sigma_2}\sqn{y_f^{k+1} - y^*}
			\\&-
			\left(\frac{1}{\sigma_1} - \frac{1}{\sigma_2}\right)\sqn{y_g^k - y^*}
			+
			\left(\sigma_2- \sigma_1\right)\sqn{y^{k+1} - y^k}.
		\end{split}
	\end{equation}
\end{lemma}
\begin{proof}
	Lines~\ref{scary:line:y:1} and~\ref{scary:line:y:3} of Algorithm~\ref{scary:alg} imply
	\begin{align*}
		y_f^{k+1} &= y_g^k + \sigma_2(y^{k+1} - y^k)\\&=
		y_g^k + \sigma_2 y^{k+1} - \frac{\sigma_2}{\sigma_1}\left(y_g^k - (1-\sigma_1)y_f^k\right)
		\\&=
		\left(1 - \frac{\sigma_2}{\sigma_1}\right)y_g^k + \sigma_2 y^{k+1} + \left(\frac{\sigma_2}{\sigma_1}- \sigma_2\right) y_f^k.
	\end{align*}
	After subtracting $y^*$ and rearranging we get
	\begin{align*}
		(y_f^{k+1}- y^*)+ \left(\frac{\sigma_2}{\sigma_1} - 1\right) (y_g^k - y^*)
		=
		\sigma_2( y^{k+1} - y^*)+ \left(\frac{\sigma_2}{\sigma_1} - \sigma_2\right)(y_f^k - y^*).
	\end{align*}
	Multiplying both sides by $\frac{\sigma_1}{\sigma_2}$ gives
	\begin{align*}
		\frac{\sigma_1}{\sigma_2}(y_f^{k+1}- y^*)+ \left(1-\frac{\sigma_1}{\sigma_2}\right) (y_g^k - y^*)
		=
		\sigma_1( y^{k+1} - y^*)+ \left(1 - \sigma_1\right)(y_f^k - y^*).
	\end{align*}
	Squaring both sides gives
	\begin{align*}
		\frac{\sigma_1}{\sigma_2}\sqn{y_f^{k+1} - y^*} + \left(1- \frac{\sigma_1}{\sigma_2}\right)\sqn{y_g^k - y^*} - \frac{\sigma_1}{\sigma_2}\left(1-\frac{\sigma_1}{\sigma_2}\right)\sqn{y_f^{k+1} - y_g^k}
		\leq
		\sigma_1\sqn{y^{k+1} - y^*} + (1-\sigma_1)\sqn{y_f^k - y^*}.
	\end{align*}
	Rearranging gives
	\begin{align*}
		-\sqn{y^{k+1}- y^*} \leq -\left(\frac{1}{\sigma_1} - \frac{1}{\sigma_2}\right)\sqn{y_g^k - y^*} + \frac{(1-\sigma_1)}{\sigma_1}\sqn{y_f^k - y^*} - \frac{1}{\sigma_2}\sqn{y_f^{k+1} - y^*}+
		\frac{1}{\sigma_2}\left(1 - \frac{\sigma_1}{\sigma_2}\right)\sqn{y_f^{k+1} - y_g^k}.
	\end{align*}
	Using Line~\ref{scary:line:y:3} of Algorithm~\ref{scary:alg} we get
	\begin{align*}
		-\sqn{y^{k+1}- y^*} \leq -\left(\frac{1}{\sigma_1} - \frac{1}{\sigma_2}\right)\sqn{y_g^k - y^*} + \frac{(1-\sigma_1)}{\sigma_1}\sqn{y_f^k - y^*} - \frac{1}{\sigma_2}\sqn{y_f^{k+1} - y^*}+
		\left(\sigma_2 - \sigma_1\right)\sqn{y^{k+1} - y^k}.
	\end{align*}
\end{proof}

\begin{lemma}
	Let $\beta$ be defined as follows:
	\begin{equation}\label{scary:beta}
		\beta = 1/(2L).
	\end{equation}
	Let $\sigma_1$ be defined as follows:
	\begin{equation}\label{scary:sigma1}
		\sigma_1 = (1/\sigma_2 + 1/2)^{-1}.
	\end{equation}
	Then the following inequality holds:
	\begin{equation}
		\begin{split}\label{scary:eq:y}
			\MoveEqLeft[4]
			\left(\frac{1}{\theta} + \frac{\beta}{2}\right)\E{\sqn{y^{k+1} - y^*}}
			+
			\frac{\beta}{2\sigma_2}\E{\sqn{y_f^{k+1} - y^*}}\\
			&\leq
			\frac{1}{\theta}\sqn{y^k - y^*}
			+
			\frac{\beta(1-\sigma_2/2)}{2\sigma_2}\sqn{y_f^k - y^*}
			-
			2\E{\<x^{k+1} - x^*, y^{k+1} - y^*>}
            +
            \bb_F(x_g^k, x^*)
			\\&-
			2\nu^{-1}\E{\<y_g^k + z_g^k - (y^* + z^*), y^{k+1} - y^*>}
			-
			\frac{\beta}{4}\sqn{y_g^k - y^*}
			+
			\left(\frac{\beta\sigma_2^2}{4} - \frac{1}{\theta}\right)\E{\sqn{y^{k+1} - y^k}}
			\\&+
			\frac{\oL}{Lb}\left(\bb_F(\omega^k, x^*) - \bb_F(x_g^k, x^*) - \<\g F(x_g^k) - \g F(x^*) -\nu x_g^k + \nu x^*, \omega^k - x_g^k>\right).
		\end{split}
	\end{equation}
\end{lemma}
\begin{proof}
	\begin{align*}
		\frac{1}{\theta}\sqn{y^{k+1} - y^*}
		&=
		\frac{1}{\theta}\sqn{y^k - y^*} + \frac{2}{\theta}\<x^{k+1} - x^k , x^{k+1} - x^*> - \frac{1}{\theta}\sqn{y^{k+1} - y^k}. 
	\end{align*}
	Using Line~\ref{scary:line:y:2} of Algorithm~\ref{scary:alg} we get
	\begin{align*}
		\frac{1}{\theta}\sqn{y^{k+1} - y^*}
		&=
		\frac{1}{\theta}\sqn{y^k - y^*}
		+
		2\beta\<\nabla^k - \nu x_g^k - y^{k+1}, y^{k+1} - y^*>
		\\&-
		2\<\nu^{-1}(y_g^k + z_g^k) + x^{k+1}, y^{k+1} - y^*>
		-
		\frac{1}{\theta}\sqn{y^{k+1} - y^k}.
	\end{align*}
	Using optimality condition \eqref{opt:x} we get
	\begin{align*}
		\frac{1}{\theta}\sqn{y^{k+1} - y^*}
		&=
		\frac{1}{\theta}\sqn{y^k - y^*}
		+
		2\beta\<\nabla^k - \nu x_g^k - (\g F(x^*) - \nu x^*) + y^*- y^{k+1}, y^{k+1} - y^*>
		\\&-
		2\<\nu^{-1}(y_g^k + z_g^k) + x^{k+1}, y^{k+1} - y^*>
		-
		\frac{1}{\theta}\sqn{y^{k+1} - y^k}
		\\&=
		\frac{1}{\theta}\sqn{y^k - y^*}
		+
		2\beta\<\nabla^k - \nu x_g^k - (\g F(x^*) - \nu x^*), y^{k+1} - y^*>
		-
		2\beta\sqn{y^{k+1} - y^*}
		\\&-
		2\<\nu^{-1}(y_g^k + z_g^k) + x^{k+1}, y^{k+1} - y^*>
		-
		\frac{1}{\theta}\sqn{y^{k+1} - y^k}
		\\&\leq
		\frac{1}{\theta}\sqn{y^k - y^*}
		+
		\beta\sqn{\nabla^k - \nu x_g^k - (\g F(x^*) - \nu x^*)}
		-
		\beta\sqn{y^{k+1} - y^*}
		\\&-
		2\<\nu^{-1}(y_g^k + z_g^k) + x^{k+1}, y^{k+1} - y^*>
		-
		\frac{1}{\theta}\sqn{y^{k+1} - y^k}.
	\end{align*}
	Taking expectation over $i$ and using the property $\E{\sqn{\xi}}=\E{\sqn{\xi-\E{\xi}}}+\sqn{\E{\xi}}$ we get
	\begin{align*}
		\frac{1}{\theta}\E{\sqn{y^{k+1} - y^*}}
		&\leq
		\frac{1}{\theta}\sqn{y^k - y^*}
		+
		\beta\sqn{\g F(x_g^k) - \nu x_g^k - (\g F(x^*) - \nu x^*)}
		-
		\beta\sqn{y^{k+1} - y^*}
		\\&-
		2\<\nu^{-1}(y_g^k + z_g^k) + x^{k+1}, y^{k+1} - y^*>
		-
		\frac{1}{\theta}\sqn{y^{k+1} - y^k}
		+
		\beta\E{\sqn{\nabla^k - \g F(x_g^k)}}.
	\end{align*}
	Function $F(x) - \frac{\nu}{2}\sqn{x}$ is convex and $L$-smooth, together with \eqref{eq:grad_upper} it implies
	\begin{align*}
		\frac{1}{\theta}\sqn{y^{k+1} - y^*}
		&\leq
		\frac{1}{\theta}\sqn{y^k - y^*}
		+
		2\beta L\left(\bg_F(x_g^k, x^*) - \frac{\nu}{2}\sqn{x_g^k - x^*}\right)
		-
		\beta\E{\sqn{y^{k+1} - y^*}}
		\\&-
		2\E{\<\nu^{-1}(y_g^k + z_g^k) + x^{k+1}, y^{k+1} - y^*>}
		-
		\frac{1}{\theta}\E{\sqn{y^{k+1} - y^k}}
		\\&+
		\frac{2\oL\beta}{b}\left(\bb_F(\omega^k, x^*) - \bb_F(x_g^k, x^*) - \<\g F(x_g^k) - \g F(x^*) -\nu x_g^k + \nu x^*, \omega^k - x_g^k>\right).
	\end{align*}
	Using $\beta$ definition \eqref{scary:beta} we get
	\begin{align*}
		\frac{1}{\theta}\E{\sqn{y^{k+1} - y^*}}
		&\leq
		\frac{1}{\theta}\sqn{y^k - y^*}
		+
		\bb_F(x_g^k,x^*)
		-
		\beta\E{\sqn{y^{k+1} - y^*}}
		\\&-
		2\E{\<\nu^{-1}(y_g^k + z_g^k) + x^{k+1}, y^{k+1} - y^*>}
		-
		\frac{1}{\theta}\E{\sqn{y^{k+1} - y^k}}
		\\&+
		\frac{\oL}{Lb}\left(\bb_F(\omega^k, x^*) - \bb_F(x_g^k, x^*) - \<\g F(x_g^k) - \g F(x^*) -\nu x_g^k + \nu x^*, \omega^k - x_g^k>\right).
	\end{align*}
	Using optimality condition \eqref{opt:y} we get
	\begin{align*}
		\frac{1}{\theta}\E{\sqn{y^{k+1} - y^*}}
		&\leq
		\frac{1}{\theta}\sqn{y^k - y^*}
		-
		\beta\E{\sqn{y^{k+1} - y^*}}
		-
		2\nu^{-1}\E{\<y_g^k + z_g^k - (y^* + z^*), y^{k+1} - y^*>}
		\\&-
		2\E{\<x^{k+1} - x^*, y^{k+1} - y^*>}
		-
		\frac{1}{\theta}\E{\sqn{y^{k+1} - y^k}}
        +
        \bb_F(x_g^k, x^*)
		\\&+
		\frac{\oL}{Lb}\left(\bb_F(\omega^k, x^*) - \bb_F(x_g^k, x^*) - \<\g F(x_g^k) - \g F(x^*) -\nu x_g^k + \nu x^*, \omega^k - x_g^k>\right).
	\end{align*}
	Using \eqref{scary:eq:1} together with $\sigma_1$ definition \eqref{scary:sigma1} we get
	\begin{align*}
		\frac{1}{\theta}\E{\sqn{y^{k+1} - y^*}}
		&\leq
		\frac{1}{\theta}\sqn{y^k - y^*}
		-
		\frac{\beta}{2}\E{\sqn{y^{k+1} - y^*}}
		+
		\frac{\beta(1-\sigma_2/2)}{2\sigma_2}\sqn{y_f^k - y^*}
		\\&-
		\frac{\beta}{2\sigma_2}\E{\sqn{y_f^{k+1} - y^*}}
		-
		\frac{\beta}{4}\sqn{y_g^k - y^*}
		+
		\frac{\beta\left(\sigma_2- \sigma_1\right)}{2}\E{\sqn{y^{k+1} - y^k}}
        +
        \bb_F(x_g^k, x^*)
		\\&-
		2\nu^{-1}\E{\<y_g^k + z_g^k - (y^* + z^*), y^{k+1} - y^*>}
		-
		2\E{\<x^{k+1} - x^*, y^{k+1} - y^*>}
		-
		\frac{1}{\theta}\E{\sqn{y^{k+1} - y^k}}
		\\&+
		\frac{\oL}{Lb}\left(\bb_F(\omega^k, x^*) - \bb_F(x_g^k, x^*) - \<\g F(x_g^k) - \g F(x^*) -\nu x_g^k + \nu x^*, \omega^k - x_g^k>\right).
		\\&\leq
		\frac{1}{\theta}\sqn{y^k - y^*}
		-
		\frac{\beta}{2}\E{\sqn{y^{k+1} - y^*}}
		+
		\frac{\beta(1-\sigma_2/2)}{2\sigma_2}\sqn{y_f^k - y^*}
		\\&-
		\frac{\beta}{2\sigma_2}\E{\sqn{y_f^{k+1} - y^*}}
		-
		\frac{\beta}{4}\sqn{y_g^k - y^*}
		+
		\left(\frac{\beta\sigma_2^2}{4} - \frac{1}{\theta}\right)\E{\sqn{y^{k+1} - y^k}}
        +
        \bb_F(x_g^k, x^*)
		\\&-
		2\nu^{-1}\E{\<y_g^k + z_g^k - (y^* + z^*), y^{k+1} - y^*>}
		-
		2\E{\<x^{k+1} - x^*, y^{k+1} - y^*>}
		\\&+
		\frac{\oL}{Lb}\left(\bb_F(\omega^k, x^*) - \bb_F(x_g^k, x^*) - \<\g F(x_g^k) - \g F(x^*) -\nu x_g^k + \nu x^*, \omega^k - x_g^k>\right).
	\end{align*}
	Rearranging gives 
	\begin{align*}
		\MoveEqLeft[4]
		\left(\frac{1}{\theta} + \frac{\beta}{2}\right)\E{\sqn{y^{k+1} - y^*}}
		+
		\frac{\beta}{2\sigma_2}\E{\sqn{y_f^{k+1} - y^*}}\\
		&\leq
		\frac{1}{\theta}\sqn{y^k - y^*}
		+
		\frac{\beta(1-\sigma_2/2)}{2\sigma_2}\sqn{y_f^k - y^*}
		-
		2\E{\<x^{k+1} - x^*, y^{k+1} - y^*>}
        +
        \bb_F(x_g^k, x^*)
		\\&-
		2\nu^{-1}\E{\<y_g^k + z_g^k - (y^* + z^*), y^{k+1} - y^*>}
		-
		\frac{\beta}{4}\sqn{y_g^k - y^*}
		+
		\left(\frac{\beta\sigma_2^2}{4} - \frac{1}{\theta}\right)\E{\sqn{y^{k+1} - y^k}}
		\\&+
		\frac{\oL}{Lb}\left(\bb_F(\omega^k, x^*) - \bb_F(x_g^k, x^*) - \<\g F(x_g^k) - \g F(x^*) -\nu x_g^k + \nu x^*, \omega^k - x_g^k>\right).
	\end{align*}
\end{proof}
\begin{lemma}
	The following inequality holds:
	\begin{equation}\label{scary:eq:2}
		\sqn{m^k}_\mP
		\leq
		8\chi^2\gamma^2\nu^{-2}\sqn{y_g^k + z_g^k}_\mP + 4\chi(1 - (4\chi)^{-1})\sqn{m^k}_\mP - 4\chi\sqn{m^{k+1}}_\mP.
	\end{equation}
\end{lemma}
\begin{proof}
	Using Line~\ref{scary:line:m} of Algorithm~\ref{scary:alg} we get
	\begin{align*}
		\sqn{m^{k+1}}_\mP
		&=
		\sqn{\gamma\nu^{-1}(y_g^k+z_g^k) + m^k - (\mW(k)\otimes \mI_d)\left[\gamma\nu^{-1}(y_g^k+z_g^k) + m^k\right]}_{\mP}\\
		&=
		\sqn{\mP\left[\gamma\nu^{-1}(y_g^k+z_g^k) + m^k\right]- (\mW(k)\otimes \mI_d)\mP\left[\gamma\nu^{-1}(y_g^k+z_g^k) + m^k\right]}.
	\end{align*}
	Using  property \eqref{eq:chi} we obtain
	\begin{align*}
		\sqn{m^{k+1}}_\mP
		&\leq (1 - \chi^{-1})
		\sqn{m^k + \gamma\nu^{-1}(y_g^k + z_g^k)}_\mP.
	\end{align*}
	Using inequality $\sqn{a+b} \leq (1+c)\sqn{a} + (1+c^{-1})\sqn{b}$ with $c = \frac{1}{2(\chi - 1)}$ we get
	\begin{align*}
		\sqn{m^{k+1}}_\mP
		&\leq (1 - \chi^{-1})
		\left[
		\left(1 + \frac{1}{2(\chi - 1)}\right)\sqn{m^k}_\mP
		+ 
		\left(1 +  2(\chi - 1)\right)\gamma^2\nu^{-2}\sqn{y_g^k + z_g^k}_\mP
		\right]
		\\&\leq
		(1 - (2\chi)^{-1})\sqn{m^k}_\mP
		+
		2\chi\gamma^2\nu^{-2}\sqn{y_g^k + z_g^k}_\mP.
	\end{align*}
	Rearranging gives
	\begin{align*}
		\sqn{m^k}_\mP
		&\leq
		8\chi^2\gamma^2\nu^{-2}\sqn{y_g^k + z_g^k}_\mP + 4\chi(1 - (4\chi)^{-1})\sqn{m^k}_\mP - 4\chi\sqn{m^{k+1}}_\mP.
	\end{align*}
\end{proof}

\begin{lemma}
	Let $\z^k$ be defined as follows:
	\begin{equation}\label{scary:zhat}
		\z^k = z^k - \mP m^k.
	\end{equation}
	Then the following inequality holds:
	\begin{equation}
		\begin{split}\label{scary:eq:z}
			\MoveEqLeft[4]\frac{1}{\gamma}\sqn{\z^{k+1} - z^*}
			+
			\frac{4}{3\gamma}\sqn{m^{k+1}}_\mP
			\leq
			\left(\frac{1}{\gamma} - \delta\right)\sqn{\z^k - z^*}
			+
			\left(1-(4\chi)^{-1} +\frac{3\gamma\delta}{2}\right)\frac{4}{3\gamma}\sqn{m^k}_\mP
			\\&-
			2\nu^{-1}\<y_g^k + z_g^k - (y^*+z^*),z^k - z^*>
			+
			\gamma\nu^{-2}\left(1 + 6\chi\right)\sqn{y_g^k+z_g^k}_\mP
			\\&+
			2\delta\sqn{z_g^k  - z^*}
			+
			\left(2\gamma\delta^2-\delta\right)\sqn{z_g^k - z^k}.
		\end{split}
	\end{equation}
\end{lemma}
\begin{proof}
	\begin{align*}
		\frac{1}{\gamma}\sqn{\z^{k+1} - z^*}
		&=
		\frac{1}{\gamma}\sqn{\z^k - z^*}
		+
		\frac{2}{\gamma}\<\z^{k+1} - \z^k,\z^k - z^*>
		+
		\frac{1}{\gamma}\sqn{\z^{k+1} - \z^k}.
	\end{align*}
	The combination of Lines~\ref{scary:line:z:2} and~\ref{scary:line:m} in Algorithm~\ref{scary:alg}, coupled with the definition of $\z^k$ in \eqref{scary:zhat}, imply
	\begin{equation*}
		\z^{k+1} - \z^k = \gamma\delta(z_g^k - z^k) - \gamma\nu^{-1}\mP(y_g^k + z_g^k).
	\end{equation*}
	Hence,
	\begin{align*}
		\frac{1}{\gamma}\sqn{\z^{k+1} - z^*}
		&=
		\frac{1}{\gamma}\sqn{\z^k - z^*}
		+
		2\delta\<z_g^k - z^k,\z^k - z^*>
		-
		2\nu^{-1}\<\mP(y_g^k + z_g^k),\z^k - z^*>
		+
		\frac{1}{\gamma}\sqn{\z^{k+1} - \z^k}
		\\&=
		\frac{1}{\gamma}\sqn{\z^k - z^*}
		+
		\delta\sqn{z_g^k - \mP m^k - z^*} - \delta\sqn{\z^k - z^*} - \delta\sqn{z_g^k - z^k}
		\\&-
		2\nu^{-1}\<\mP(y_g^k + z_g^k),\z^k - z^*>
		+
		\gamma\sqn{\delta(z_g^k - z^k) - \nu^{-1}\mP(y_g^k+z_g^k)}
		\\&\leq
		\left(\frac{1}{\gamma} - \delta\right)\sqn{\z^k - z^*}
		+
		2\delta\sqn{z_g^k  - z^*}
		+
		2\delta\sqn{m^k}_\mP
		-
		\delta\sqn{z_g^k - z^k}
		\\&-
		2\nu^{-1}\<\mP(y_g^k + z_g^k),\z^k - z^*>
		+
		2\gamma\delta^2\sqn{z_g^k - z^k}
		+
		\gamma\sqn{\nu^{-1}\mP(y_g^k+z_g^k)}
		\\&\leq
		\left(\frac{1}{\gamma} - \delta\right)\sqn{\z^k - z^*}
		+
		2\delta\sqn{z_g^k  - z^*}
		+
		\left(2\gamma\delta^2-\delta\right)\sqn{z_g^k - z^k}
		\\&-
		2\nu^{-1}\<\mP(y_g^k + z_g^k),z^k - z^*>
		+
		\gamma\sqn{\nu^{-1}\mP(y_g^k+z_g^k)}
		+
		2\delta\sqn{m^k}_\mP
		+
		2\nu^{-1}\<\mP(y_g^k + z_g^k),m^k>.
	\end{align*}
	Using the fact that $z^k \in \cL^\perp$ for all $k=0,1,2\ldots$ and optimality condition \eqref{opt:z} we get
	\begin{align*}
		\frac{1}{\gamma}\sqn{\z^{k+1} - z^*}
		&\leq
		\left(\frac{1}{\gamma} - \delta\right)\sqn{\z^k - z^*}
		+
		2\delta\sqn{z_g^k  - z^*}
		+
		\left(2\gamma\delta^2-\delta\right)\sqn{z_g^k - z^k}
		\\&-
		2\nu^{-1}\<y_g^k + z_g^k - (y^*+z^*),z^k - z^*>
		+
		\gamma\nu^{-2}\sqn{y_g^k+z_g^k}_\mP
		\\&+
		2\delta\sqn{m^k}_\mP
		+
		2\nu^{-1}\<\mP(y_g^k + z_g^k),m^k>.
	\end{align*}
	Using Young's inequality we get
	\begin{align*}
		\frac{1}{\gamma}\sqn{\z^{k+1} - z^*}
		&\leq
		\left(\frac{1}{\gamma} - \delta\right)\sqn{\z^k - z^*}
		+
		2\delta\sqn{z_g^k  - z^*}
		+
		\left(2\gamma\delta^2-\delta\right)\sqn{z_g^k - z^k}
		\\&-
		2\nu^{-1}\<y_g^k + z_g^k - (y^*+z^*),z^k - z^*>
		+
		\gamma\nu^{-2}\sqn{y_g^k+z_g^k}_\mP
		\\&+
		2\delta\sqn{m^k}_\mP
		+
		3\gamma\chi\nu^{-2}\sqn{y_g^k + z_g^k}_\mP + \frac{1}{3\gamma\chi}\sqn{m^k}_\mP.
	\end{align*}
	Using \eqref{scary:eq:2} we get
	\begin{align*}
		\frac{1}{\gamma}\sqn{\z^{k+1} - z^*}
		&\leq
		\left(\frac{1}{\gamma} - \delta\right)\sqn{\z^k - z^*}
		+
		2\delta\sqn{z_g^k  - z^*}
		+
		\left(2\gamma\delta^2-\delta\right)\sqn{z_g^k - z^k}
		\\&-
		2\nu^{-1}\<y_g^k + z_g^k - (y^*+z^*),z^k - z^*>
		+
		\gamma\nu^{-2}\sqn{y_g^k+z_g^k}_\mP
		\\&+
		2\delta\sqn{m^k}_\mP
		+
		6\gamma\nu^{-2}\chi\sqn{y_g^k + z_g^k}_\mP + \frac{4(1 - (4\chi)^{-1})}{3\gamma}\sqn{m^k}_\mP - \frac{4}{3\gamma}\sqn{m^{k+1}}_\mP
		\\&=
		\left(\frac{1}{\gamma} - \delta\right)\sqn{\z^k - z^*}
		+
		2\delta\sqn{z_g^k  - z^*}
		+
		\left(2\gamma\delta^2-\delta\right)\sqn{z_g^k - z^k}
		\\&-
		2\nu^{-1}\<y_g^k + z_g^k - (y^*+z^*),z^k - z^*>
		+
		\gamma\nu^{-2}\left(1 + 6\chi\right)\sqn{y_g^k+z_g^k}_\mP
		\\&+
		\left(1-(4\chi)^{-1}+\frac{3\gamma\delta}{2}\right)\frac{4}{3\gamma}\sqn{m^k}_\mP - \frac{4}{3\gamma}\sqn{m^{k+1}}_\mP.
	\end{align*}
\end{proof}

\begin{lemma}
	The following inequality holds:
	\begin{equation}
		\begin{split}\label{scary:eq:3}
			\MoveEqLeft[3]2\<y_g^k + z_g^k - (y^*+z^*),y^k + z^k - (y^*+ z^*)>
			\\&\geq
			2\sqn{y_g^k + z_g^k - (y^*+z^*)}
			+
			\frac{(1-\sigma_2/2)}{\sigma_2}\left(\sqn{y_g^k + z_g^k - (y^*+z^*)}  - \sqn{y_f^k + z_f^k - (y^*+z^*)}\right).
		\end{split}
	\end{equation}
\end{lemma}
\begin{proof}
	\begin{align*}
		\MoveEqLeft[4]2\<y_g^k + z_g^k - (y^*+z^*),y^k + z^k - (y^*+ z^*)>
		\\&=
		2\sqn{y_g^k + z_g^k - (y^*+z^*)}
		+
		2\<y_g^k + z_g^k - (y^*+z^*),y^k + z^k - (y_g^k + z_g^k)>.
	\end{align*}
	Using Lines~\ref{scary:line:y:1} and~\ref{scary:line:z:1} of Algorithm~\ref{scary:alg} we get
	\begin{align*}
		\MoveEqLeft[4]2\<y_g^k + z_g^k - (y^*+z^*),y^k + z^k - (y^*+ z^*)>
		\\&=
		2\sqn{y_g^k + z_g^k - (y^*+z^*)}
		+
		\frac{2(1-\sigma_1)}{\sigma_1}\<y_g^k + z_g^k - (y^*+z^*), y_g^k + z_g^k - (y_f^k + z_f^k)>
		\\&=
		2\sqn{y_g^k + z_g^k - (y^*+z^*)}
		\\&+
		\frac{(1-\sigma_1)}{\sigma_1}\left(\sqn{y_g^k + z_g^k - (y^*+z^*)} + \sqn{y_g^k + z_g^k - (y_f^k + z_f^k)} - \sqn{y_f^k + z_f^k - (y^*+z^*)}\right)
		\\&\geq
		2\sqn{y_g^k + z_g^k - (y^*+z^*)}
		+
		\frac{(1-\sigma_1)}{\sigma_1}\left(\sqn{y_g^k + z_g^k - (y^*+z^*)}  - \sqn{y_f^k + z_f^k - (y^*+z^*)}\right).
	\end{align*}
	Using $\sigma_1$ definition \eqref{scary:sigma1} we get
	\begin{align*}
		\MoveEqLeft[4]2\<y_g^k + z_g^k - (y^*+z^*),y^k + z^k - (y^*+ z^*)>
		\\&\geq
		2\sqn{y_g^k + z_g^k - (y^*+z^*)}
		+
		\frac{(1-\sigma_2/2)}{\sigma_2}\left(\sqn{y_g^k + z_g^k - (y^*+z^*)}  - \sqn{y_f^k + z_f^k - (y^*+z^*)}\right).
	\end{align*}
\end{proof}
\begin{lemma}
	Let $\zeta$ be defined by
	\begin{align}\label{scary:zeta}
		\zeta = 1/2.
	\end{align}
	Then the following inequality holds:
	\begin{align}
		\label{scary:eq:4}
			\MoveEqLeft[6]-2\<y^{k+1} - y^k,y_g^k + z_g^k - (y^*+z^*)>
			\nonumber\\&\leq
			\frac{1}{\sigma_2}\sqn{y_g^k + z_g^k - (y^*+z^*)}
			-
			\frac{1}{\sigma_2}\sqn{y_f^{k+1} + z_f^{k+1} - (y^*+z^*)}
			\nonumber\\&+
			2\sigma_2\sqn{y^{k+1} - y^k}
			-
			\frac{1}{2\sigma_2\chi}\sqn{y_g^k + z_g^k}_{\mP}.
	\end{align}
\end{lemma}
\begin{proof}
	\begin{align*}
		\MoveEqLeft[4]\sqn{y_f^{k+1} + z_f^{k+1} - (y^*+z^*)}
		\\&=
		\sqn{y_g^k + z_g^k - (y^*+z^*)} + 2\<y_f^{k+1} + z_f^{k+1} - (y_g^k + z_g^k),y_g^k + z_g^k - (y^*+z^*)>
		\\&+
		\sqn{y_f^{k+1} + z_f^{k+1} - (y_g^k + z_g^k)}
		\\&\leq
		\sqn{y_g^k + z_g^k - (y^*+z^*)} + 2\<y_f^{k+1} + z_f^{k+1} - (y_g^k + z_g^k),y_g^k + z_g^k - (y^*+z^*)>
		\\&+
		2\sqn{y_f^{k+1} - y_g^k}
		+
		2\sqn{z_f^{k+1} - z_g^k}.
	\end{align*}
	Using Line~\ref{scary:line:y:3} of Algorithm~\ref{scary:alg} we get
	\begin{align*}
		\MoveEqLeft[4]\sqn{y_f^{k+1} + z_f^{k+1} - (y^*+z^*)}
		\\&\leq
		\sqn{y_g^k + z_g^k - (y^*+z^*)}
		+
		2\sigma_2\<y^{k+1} - y^k,y_g^k + z_g^k - (y^*+z^*)>
		+
		2\sigma_2^2\sqn{y^{k+1} - y^k}
		\\&+
		2\<z_f^{k+1} - z_g^k,y_g^k + z_g^k - (y^*+z^*)>
		+
		2\sqn{z_f^{k+1} - z_g^k}.
	\end{align*}
	Using Line~\ref{scary:line:z:3} of Algorithm~\ref{scary:alg} and optimality condition \eqref{opt:z} we get
	\begin{align*}
		\MoveEqLeft[4]\sqn{y_f^{k+1} + z_f^{k+1} - (y^*+z^*)}
		\\&\leq
		\sqn{y_g^k + z_g^k - (y^*+z^*)}
		+
		2\sigma_2\<y^{k+1} - y^k,y_g^k + z_g^k - (y^*+z^*)>
		+
		2\sigma_2^2\sqn{y^{k+1} - y^k}
		\\&-
		2\zeta\< (\mW(k)\otimes \mI_d)(y_g^k + z_g^k),y_g^k + z_g^k - (y^*+z^*)>
		+
		2\zeta^2\sqn{(\mW(k)\otimes \mI_d)(y_g^k + z_g^k)}
		\\&=
		\sqn{y_g^k + z_g^k - (y^*+z^*)}
		+
		2\sigma_2\<y^{k+1} - y^k,y_g^k + z_g^k - (y^*+z^*)>
		+
		2\sigma_2^2\sqn{y^{k+1} - y^k}
		\\&-
		2\zeta\< (\mW(k)\otimes \mI_d)(y_g^k + z_g^k),y_g^k + z_g^k>
		+
		2\zeta^2\sqn{(\mW(k)\otimes \mI_d)(y_g^k + z_g^k)}.
	\end{align*}
	Using $\zeta$ definition \eqref{scary:zeta} we get
	\begin{align*}
		\MoveEqLeft[4]\sqn{y_f^{k+1} + z_f^{k+1} - (y^*+z^*)}
		\\&\leq
		\sqn{y_g^k + z_g^k - (y^*+z^*)}
		+
		2\sigma_2\<y^{k+1} - y^k,y_g^k + z_g^k - (y^*+z^*)>
		+
		2\sigma_2^2\sqn{y^{k+1} - y^k}
		\\&-
		\< (\mW(k)\otimes \mI_d)(y_g^k + z_g^k),y_g^k + z_g^k>
		+
		\frac{1}{2}\sqn{(\mW(k)\otimes \mI_d)(y_g^k + z_g^k)}
		\\&=
		\sqn{y_g^k + z_g^k - (y^*+z^*)}
		+
		2\sigma_2\<y^{k+1} - y^k,y_g^k + z_g^k - (y^*+z^*)>
		+
		2\sigma_2^2\sqn{y^{k+1} - y^k}
		\\&-
		\frac{1}{2}\sqn{(\mW(k)\otimes \mI_d)(y_g^k + z_g^k)}
		-
		\frac{1}{2}\sqn{y_g^k + z_g^k}
		+
		\frac{1}{2}\sqn{(\mW(k)\otimes \mI_d)(y_g^k + z_g^k) - (y_g^k + z_g^k)}
		\\&+
		\frac{1}{2}\sqn{(\mW(k)\otimes \mI_d)(y_g^k + z_g^k)}
		\\&\leq
		\sqn{y_g^k + z_g^k - (y^*+z^*)}
		+
		2\sigma_2\<y^{k+1} - y^k,y_g^k + z_g^k - (y^*+z^*)>
		+
		2\sigma_2^2\sqn{y^{k+1} - y^k}
		\\&-
		\frac{1}{2}\sqn{y_g^k + z_g^k}_\mP
		+
		\frac{1}{2}\sqn{(\mW(k)\otimes \mI_d)(y_g^k + z_g^k) - (y_g^k + z_g^k)}_\mP.
		\\&=
		\sqn{y_g^k + z_g^k - (y^*+z^*)}
		+
		2\sigma_2\<y^{k+1} - y^k,y_g^k + z_g^k - (y^*+z^*)>
		+
		2\sigma_2^2\sqn{y^{k+1} - y^k}
		\\&-
		\frac{1}{2}\sqn{y_g^k + z_g^k}_\mP
		+
		\frac{1}{2}\sqn{(\mW(k)\otimes \mI_d)\mP(y_g^k + z_g^k) - \mP(y_g^k + z_g^k)}.
	\end{align*}
	Using condition \eqref{eq:chi} we get
	\begin{align*}
		\MoveEqLeft[4]\sqn{y_f^{k+1} + z_f^{k+1} - (y^*+z^*)}
		\\&\leq
		\sqn{y_g^k + z_g^k - (y^*+z^*)}
		+
		2\sigma_2\<y^{k+1} - y^k,y_g^k + z_g^k - (y^*+z^*)>
		+
		2\sigma_2^2\sqn{y^{k+1} - y^k}
		\\&-
		(2\chi)^{-1}\sqn{y_g^k + z_g^k}_\mP.
	\end{align*}
	Rearranging gives
	\begin{align*}
		\MoveEqLeft[6]-2\<y^{k+1} - y^k,y_g^k + z_g^k - (y^*+z^*)>
		\\&\leq
		\frac{1}{\sigma_2}\sqn{y_g^k + z_g^k - (y^*+z^*)}
		-
		\frac{1}{\sigma_2}\sqn{y_f^{k+1} + z_f^{k+1} - (y^*+z^*)}
		\\&+
		2\sigma_2\sqn{y^{k+1} - y^k}
		-
		\frac{1}{2\sigma_2\chi}\sqn{y_g^k + z_g^k}_{\mP}.
	\end{align*}
\end{proof} 

\begin{lemma}
	Let $\delta$ be defined as follows:
	\begin{equation}\label{scary:delta}
		\delta = \frac{1}{17L}.
	\end{equation}
	Let $\gamma$ be defined as follows:
	\begin{equation}\label{scary:gamma}
		\gamma = \frac{\nu}{14\sigma_2\chi^2}.
	\end{equation}
	Let $\theta$ be defined as follows:
	\begin{equation}\label{scary:theta}
		\theta = \frac{\nu}{4\sigma_2}.
	\end{equation}
	Let $\sigma_2$ be defined as follows:
	\begin{equation}\label{scary:sigma2}
		\sigma_2 = \frac{\sqrt{\mu}}{16\chi\sqrt{L}}.
	\end{equation}
	Let $\Psi_{yz}^k$ be the following Lyapunov function
	\begin{align}
		\label{scary:Psi_yz}
			\Psi_{yz}^k &= \left(\frac{1}{\theta} + \frac{\beta}{2}\right)\sqn{y^{k} - y^*}
			+
			\frac{\beta}{2\sigma_2}\sqn{y_f^{k} - y^*}
			+
			\frac{1}{\gamma}\sqn{\z^{k} - z^*}
			+
			\frac{4}{3\gamma}\sqn{m^{k}}_\mP
			+
			\frac{\nu^{-1}}{\sigma_2}\sqn{y_f^{k} + z_f^{k} - (y^*+z^*)}.
	\end{align}
	Then the following inequality holds:
	\begin{align}\label{scary:eq:yz}
		  \E{\Psi_{yz}^{k+1}} &\leq \left(1 - \frac{\sqrt{\mu}}{32\chi\sqrt{L}}\right)\Psi_{yz}^k
    		-
    		2\E{\<x^{k+1} - x^*, y^{k+1} - y^*>}
            +
            \bb_F(x_g^k, x^*)
    		\nonumber\\&+
    		\frac{\oL}{Lb}\left(\bb_F(\omega^k, x^*) - \bb_F(x_g^k, x^*) - \<\g F(x_g^k) - \g F(x^*) -\nu x_g^k + \nu x^*, \omega^k - x_g^k>\right).
	\end{align}
\end{lemma}
\begin{proof}
	Combining \eqref{scary:eq:y} and \eqref{scary:eq:z} gives
	\begin{align*}
		\left(\frac{1}{\theta} + \frac{\beta}{2}\right)&\E{\sqn{y^{k+1} - y^*}}
		+
		\frac{\beta}{2\sigma_2}\E{\sqn{y_f^{k+1} - y^*}}
		+
		\frac{1}{\gamma}\sqn{\z^{k+1} - z^*}
		+
		\frac{4}{3\gamma}\sqn{m^{k+1}}_\mP
		\\&\leq
		\left(\frac{1}{\gamma} - \delta\right)\sqn{\z^k - z^*}
		+
		\left(1-(4\chi)^{-1}+\frac{3\gamma\delta}{2}\right)\frac{4}{3\gamma}\sqn{m^k}_\mP
		+
		\frac{1}{\theta}\sqn{y^k - y^*}
		+
		\frac{\beta(1-\sigma_2/2)}{2\sigma_2}\sqn{y_f^k - y^*}
		\\&-
		2\nu^{-1}\<y_g^k + z_g^k - (y^*+z^*),y^k + z^k - (y^*+ z^*)>
		-
		2\nu^{-1}\E{\<y_g^k + z_g^k - (y^* + z^*), y^{k+1} - y^k>}
		\\&+
		\gamma\nu^{-2}\left(1 + 6\chi\right)\sqn{y_g^k+z_g^k}_\mP
		+
		\left(\frac{\beta\sigma_2^2}{4} - \frac{1}{\theta}\right)\E{\sqn{y^{k+1} - y^k}}
		+
		2\delta\sqn{z_g^k  - z^*}
		-
		\frac{\beta}{4}\sqn{y_g^k - y^*}
		\\&-
		2\E{\<x^{k+1} - x^*, y^{k+1} - y^*>}
		+\left(2\gamma\delta^2-\delta\right)\sqn{z_g^k - z^k}
        +
        \bb_F(x_g^k, x^*)
		\\&+
		\frac{\oL}{Lb}\left(\bb_F(\omega^k, x^*) -   \bb_F(x_g^k, x^*) - \<\g F(x_g^k) - \g F(x^*) -\nu x_g^k + \nu x^*, \omega^k - x_g^k>\right).
	\end{align*}
	Using \eqref{scary:eq:3} and \eqref{scary:eq:4} we get
	\begin{align*}
		\MoveEqLeft[4]\left(\frac{1}{\theta} + \frac{\beta}{2}\right)\E{\sqn{y^{k+1} - y^*}}
		+
		\frac{\beta}{2\sigma_2}\E{\sqn{y_f^{k+1} - y^*}}
		+
		\frac{1}{\gamma}\sqn{\z^{k+1} - z^*}
		+
		\frac{4}{3\gamma}\sqn{m^{k+1}}_\mP
		\\&\leq
		\left(\frac{1}{\gamma} - \delta\right)\sqn{\z^k - z^*}
		+
		\left(1- (4\chi)^{-1}+\frac{3\gamma\delta}{2}\right)\frac{4}{3\gamma}\sqn{m^k}_\mP
		+
		\frac{1}{\theta}\sqn{y^k - y^*}
		+
		\frac{\beta(1-\sigma_2/2)}{2\sigma_2}\sqn{y_f^k - y^*}
		\\&-
		2\nu^{-1}\sqn{y_g^k + z_g^k - (y^*+z^*)}
		+
		\frac{\nu^{-1}(1-\sigma_2/2)}{\sigma_2}\left(\sqn{y_f^k + z_f^k - (y^*+z^*)} - \sqn{y_g^k + z_g^k - (y^*+z^*)}\right)
		\\&+
		\frac{\nu^{-1}}{\sigma_2}\sqn{y_g^k + z_g^k - (y^*+z^*)}
		-
		\frac{\nu^{-1}}{\sigma_2}\E{\sqn{y_f^{k+1} + z_f^{k+1} - (y^*+z^*)}}
		+
		2\nu^{-1}\sigma_2\E{\sqn{y^{k+1} - y^k}}
		\\&-
		\frac{\nu^{-1}}{2\sigma_2\chi}\sqn{y_g^k + z_g^k}_{\mP}
		+
		\gamma\nu^{-2}\left(1 + 6\chi\right)\sqn{y_g^k+z_g^k}_\mP
		+
		\left(\frac{\beta\sigma_2^2}{4} - \frac{1}{\theta}\right)\E{\sqn{y^{k+1} - y^k}}
		+
		2\delta\sqn{z_g^k  - z^*}
		\\&-
		\frac{\beta}{4}\sqn{y_g^k - y^*}
		-
		2\E{\<x^{k+1} - x^*, y^{k+1} - y^*>}
		+\left(2\gamma\delta^2-\delta\right)\sqn{z_g^k - z^k}
        +
        \bb_F(x_g^k, x^*)
		\\&+
		\frac{\oL}{Lb}\left(\bb_F(\omega^k, x^*) -   \bb_F(x_g^k, x^*) - \<\g F(x_g^k) - \g F(x^*) -\nu x_g^k + \nu x^*, \omega^k - x_g^k>\right)
		\\&=
		\left(\frac{1}{\gamma} - \delta\right)\sqn{\z^k - z^*}
		+
		\left(1-(4\chi)^{-1}+\frac{3\gamma\delta}{2}\right)\frac{4}{3\gamma}\sqn{m^k}_\mP
		+
		\frac{1}{\theta}\sqn{y^k - y^*}
		+
		\frac{\beta(1-\sigma_2/2)}{2\sigma_2}\sqn{y_f^k - y^*}
		\\&+
		\frac{\nu^{-1}(1-\sigma_2/2)}{\sigma_2}\sqn{y_f^k + z_f^k - (y^*+z^*)}
		-
		\frac{\nu^{-1}}{\sigma_2}\E{\sqn{y_f^{k+1} + z_f^{k+1} - (y^*+z^*)}}
		\\&+
		2\delta\sqn{z_g^k  - z^*}
		-
		\frac{\beta}{4}\sqn{y_g^k - y^*}
		+
		\nu^{-1}\left(\frac{1}{\sigma_2} - \frac{(1-\sigma_2/2)}{\sigma_2} - 2\right)\sqn{y_g^k + z_g^k - (y^*+z^*)}
		\\&+
		\left(\gamma\nu^{-2}\left(1 + 6\chi\right) - \frac{\nu^{-1}}{2\sigma_2\chi}\right)\sqn{y_g^k+z_g^k}_\mP
		+
		\left(\frac{\beta\sigma_2^2}{4} + 	2\nu^{-1}\sigma_2 - \frac{1}{\theta}\right)\E{\sqn{y^{k+1} - y^k}}
		\\&+
		\left(2\gamma\delta^2-\delta\right)\sqn{z_g^k - z^k}
		-
		2\E{\<x^{k+1} - x^*, y^{k+1} - y^*>}
        +
        \bb_F(x_g^k, x^*)
		\\&+
		\frac{\oL}{Lb}\left(\bb_F(\omega^k, x^*) - \bb_F(x_g^k, x^*) - \<\g F(x_g^k) - \g F(x^*) -\nu x_g^k + \nu x^*, \omega^k - x_g^k>\right)
		\\&=
		\left(\frac{1}{\gamma} - \delta\right)\sqn{\z^k - z^*}
		+
		\left(1-(4\chi)^{-1}+\frac{3\gamma\delta}{2}\right)\frac{4}{3\gamma}\sqn{m^k}_\mP
		+
		\frac{1}{\theta}\sqn{y^k - y^*}
		+
		\frac{\beta(1-\sigma_2/2)}{2\sigma_2}\sqn{y_f^k - y^*}
		\\&+
		\frac{\nu^{-1}(1-\sigma_2/2)}{\sigma_2}\sqn{y_f^k + z_f^k - (y^*+z^*)}
		-
		\frac{\nu^{-1}}{\sigma_2}\E{\sqn{y_f^{k+1} + z_f^{k+1} - (y^*+z^*)}}
		\\&+
		2\delta\sqn{z_g^k  - z^*}
		-
		\frac{\beta}{4}\sqn{y_g^k - y^*}
		-
		\frac{3\nu^{-1}}{2}\sqn{y_g^k + z_g^k - (y^*+z^*)}
		+
		\left(2\gamma\delta^2-\delta\right)\sqn{z_g^k - z^k}
		\\&+
		\left(\gamma\nu^{-2}\left(1 + 6\chi\right) - \frac{\nu^{-1}}{2\sigma_2\chi}\right)\sqn{y_g^k+z_g^k}_\mP
		+
		\left(\frac{\beta\sigma_2^2}{4} + 	2\nu^{-1}\sigma_2 - \frac{1}{\theta}\right)\E{\sqn{y^{k+1} - y^k}}
		\\&+
		2\E{\<x^{k+1} - x^*, y^{k+1} - y^*>}
		+
        \bb_F(x_g^k, x^*)
        \\&+
		\frac{\oL}{Lb}\left(\bb_F(\omega^k, x^*) - \bb_F(x_g^k, x^*) - \<\g F(x_g^k) - \g F(x^*) -\nu x_g^k + \nu x^*, \omega^k - x_g^k>\right).
	\end{align*}
	Using $\beta$ definition \eqref{scary:beta} and $\nu$ definition \eqref{scary:nu} we get
	\begin{align*}
		\MoveEqLeft[4]\left(\frac{1}{\theta} + \frac{\beta}{2}\right)\E{\sqn{y^{k+1} - y^*}}
		+
		\frac{\beta}{2\sigma_2}\E{\sqn{y_f^{k+1} - y^*}}
		+
		\frac{1}{\gamma}\sqn{\z^{k+1} - z^*}
		+
		\frac{4}{3\gamma}\sqn{m^{k+1}}_\mP
		\\&\leq
		\left(\frac{1}{\gamma} - \delta\right)\sqn{\z^k - z^*}
		+
		\left(1-(4\chi)^{-1}+\frac{3\gamma\delta}{2}\right)\frac{4}{3\gamma}\sqn{m^k}_\mP
		+
		\frac{1}{\theta}\sqn{y^k - y^*}
		+
		\frac{\beta(1-\sigma_2/2)}{2\sigma_2}\sqn{y_f^k - y^*}
		\\&+
		\frac{\nu^{-1}(1-\sigma_2/2)}{\sigma_2}\sqn{y_f^k + z_f^k - (y^*+z^*)}
		-
		\frac{\nu^{-1}}{\sigma_2}\E{\sqn{y_f^{k+1} + z_f^{k+1} - (y^*+z^*)}}
		\\&+
		2\delta\sqn{z_g^k  - z^*}
		-
		\frac{1}{8L}\sqn{y_g^k - y^*}
		-
		\frac{3}{\mu}\sqn{y_g^k + z_g^k - (y^*+z^*)}
		+
		\left(2\gamma\delta^2-\delta\right)\sqn{z_g^k - z^k}
		\\&+
		\left(\gamma\nu^{-2}\left(1 + 6\chi\right) - \frac{\nu^{-1}}{2\sigma_2\chi}\right)\sqn{y_g^k+z_g^k}_\mP
		+
		\left(\frac{\beta\sigma_2^2}{4} + 	2\nu^{-1}\sigma_2 - \frac{1}{\theta}\right)\E{\sqn{y^{k+1} - y^k}}
		\\&-
		2\E{\<x^{k+1} - x^*, y^{k+1} - y^*>}
		+
        \bb_F(x_g^k, x^*)
        \\&+
		\frac{\oL}{Lb}\left(\bb_F(\omega^k, x^*) - \bb_F(x_g^k, x^*) - \<\g F(x_g^k) - \g F(x^*) -\nu x_g^k + \nu x^*, \omega^k - x_g^k>\right).
	\end{align*}
	Using $\delta$ definition \eqref{scary:delta} we get
	\begin{align*}
		\MoveEqLeft[4]\left(\frac{1}{\theta} + \frac{\beta}{2}\right)\E{\sqn{y^{k+1} - y^*}}
		+
		\frac{\beta}{2\sigma_2}\E{\sqn{y_f^{k+1} - y^*}}
		+
		\frac{1}{\gamma}\sqn{\z^{k+1} - z^*}
		+
		\frac{4}{3\gamma}\sqn{m^{k+1}}_\mP
		\\&\leq
		\left(\frac{1}{\gamma} - \delta\right)\sqn{\z^k - z^*}
		+
		\left(1-(4\chi)^{-1}+\frac{3\gamma\delta}{2}\right)\frac{4}{3\gamma}\sqn{m^k}_\mP
		+
		\frac{1}{\theta}\sqn{y^k - y^*}
		+
		\frac{\beta(1-\sigma_2/2)}{2\sigma_2}\sqn{y_f^k - y^*}
		\\&+
		\frac{\nu^{-1}(1-\sigma_2/2)}{\sigma_2}\sqn{y_f^k + z_f^k - (y^*+z^*)}
		-
		\frac{\nu^{-1}}{\sigma_2}\E{\sqn{y_f^{k+1} + z_f^{k+1} - (y^*+z^*)}}
		\\&+
		\left(\gamma\nu^{-2}\left(1 + 6\chi\right) - \frac{\nu^{-1}}{2\sigma_2\chi}\right)\sqn{y_g^k+z_g^k}_\mP
		+
		\left(\frac{\beta\sigma_2^2}{4} + 	2\nu^{-1}\sigma_2 - \frac{1}{\theta}\right)\E{\sqn{y^{k+1} - y^k}}
		\\&+
		\left(2\gamma\delta^2-\delta\right)\sqn{z_g^k - z^k}
		-
		2\E{\<x^{k+1} - x^*, y^{k+1} - y^*>}
        +
        \bb_F(x_g^k, x^*)
		\\&+
		\frac{\oL}{Lb}\left(\bb_F(\omega^k, x^*) - \bb_F(x_g^k, x^*) - \<\g F(x_g^k) - \g F(x^*) -\nu x_g^k + \nu x^*, \omega^k - x_g^k>\right).
	\end{align*}
	Using $\gamma$ definition \eqref{scary:gamma} we get
	\begin{align*}
		\MoveEqLeft[4]\left(\frac{1}{\theta} + \frac{\beta}{2}\right)\E{\sqn{y^{k+1} - y^*}}
		+
		\frac{\beta}{2\sigma_2}\E{\sqn{y_f^{k+1} - y^*}}
		+
		\frac{1}{\gamma}\sqn{\z^{k+1} - z^*}
		+
		\frac{4}{3\gamma}\sqn{m^{k+1}}_\mP
		\\&\leq
		\left(\frac{1}{\gamma} - \delta\right)\sqn{\z^k - z^*}
		+
		\left(1-(4\chi)^{-1}+\frac{3\gamma\delta}{2}\right)\frac{4}{3\gamma}\sqn{m^k}_\mP
		+
		\frac{1}{\theta}\sqn{y^k - y^*}
		+
		\frac{\beta(1-\sigma_2/2)}{2\sigma_2}\sqn{y_f^k - y^*}
		\\&+
		\frac{\nu^{-1}(1-\sigma_2/2)}{\sigma_2}\sqn{y_f^k + z_f^k - (y^*+z^*)}
		-
		\frac{\nu^{-1}}{\sigma_2}\E{\sqn{y_f^{k+1} + z_f^{k+1} - (y^*+z^*)}}
		\\&+
		\left(\frac{\beta\sigma_2^2}{4} + 	2\nu^{-1}\sigma_2 - \frac{1}{\theta}\right)\E{\sqn{y^{k+1} - y^k}}
		+
		\left(2\gamma\delta^2-\delta\right)\sqn{z_g^k - z^k}
		\\&-
		2\E{\<x^{k+1} - x^*, y^{k+1} - y^*>}
		+
        \bb_F(x_g^k, x^*)
        \\&+
		\frac{\oL}{Lb}\left(\bb_F(\omega^k, x^*) - \bb_F(x_g^k, x^*) - \<\g F(x_g^k) - \g F(x^*) -\nu x_g^k + \nu x^*, \omega^k - x_g^k>\right).
	\end{align*}
	Using $\theta$ definition together with \eqref{scary:nu}, \eqref{scary:beta} and \eqref{scary:sigma2} gives
	\begin{align*}
		\MoveEqLeft[4]\left(\frac{1}{\theta} + \frac{\beta}{2}\right)\E{\sqn{y^{k+1} - y^*}}
		+
		\frac{\beta}{2\sigma_2}\E{\sqn{y_f^{k+1} - y^*}}
		+
		\frac{1}{\gamma}\sqn{\z^{k+1} - z^*}
		+
		\frac{4}{3\gamma}\sqn{m^{k+1}}_\mP
		\\&\leq
		\left(\frac{1}{\gamma} - \delta\right)\sqn{\z^k - z^*}
		+
		\left(1-(4\chi)^{-1}+\frac{3\gamma\delta}{2}\right)\frac{4}{3\gamma}\sqn{m^k}_\mP
		+
		\frac{1}{\theta}\sqn{y^k - y^*}
		+
		\frac{\beta(1-\sigma_2/2)}{2\sigma_2}\sqn{y_f^k - y^*}
		\\&+
		\frac{\nu^{-1}(1-\sigma_2/2)}{\sigma_2}\sqn{y_f^k + z_f^k - (y^*+z^*)}
		-
		\frac{\nu^{-1}}{\sigma_2}\E{\sqn{y_f^{k+1} + z_f^{k+1} - (y^*+z^*)}}
		\\&+
		\left(2\gamma\delta^2-\delta\right)\sqn{z_g^k - z^k}
		-
		2\E{\<x^{k+1} - x^*, y^{k+1} - y^*>}
        +
        \bb_F(x_g^k, x^*)
		\\&+
		\frac{\oL}{Lb}\left(\bb_F(\omega^k, x^*) - \bb_F(x_g^k, x^*) - \<\g F(x_g^k) - \g F(x^*) -\nu x_g^k + \nu x^*, \omega^k - x_g^k>\right).
	\end{align*}
	Using $\gamma$ definition \eqref{scary:gamma} and $\delta$ definition \eqref{scary:delta} we get
	\begin{align*}
		\MoveEqLeft[4]\left(\frac{1}{\theta} + \frac{\beta}{2}\right)\E{\sqn{y^{k+1} - y^*}}
		+
		\frac{\beta}{2\sigma_2}\E{\sqn{y_f^{k+1} - y^*}}
		+
		\frac{1}{\gamma}\sqn{\z^{k+1} - z^*}
		+
		\frac{4}{3\gamma}\sqn{m^{k+1}}_\mP
		\\&\leq
		\left(\frac{1}{\gamma} - \delta\right)\sqn{\z^k - z^*}
		+
		\left(1-(8\chi)^{-1}\right)\frac{4}{3\gamma}\sqn{m^k}_\mP
		+
		\frac{1}{\theta}\sqn{y^k - y^*}
		+
		\frac{\beta(1-\sigma_2/2)}{2\sigma_2}\sqn{y_f^k - y^*}
		\\&+
		\frac{\nu^{-1}(1-\sigma_2/2)}{\sigma_2}\sqn{y_f^k + z_f^k - (y^*+z^*)}
		-
		\frac{\nu^{-1}}{\sigma_2}\E{\sqn{y_f^{k+1} + z_f^{k+1} - (y^*+z^*)}}
		\\&-
		2\E{\<x^{k+1} - x^*, y^{k+1} - y^*>}
        +
        \bb_F(x_g^k, x^*)
        \\&+
		\frac{\oL}{Lb}\left(\bb_F(\omega^k, x^*) - \bb_F(x_g^k, x^*) - \<\g F(x_g^k) - \g F(x^*) -\nu x_g^k + \nu x^*, \omega^k - x_g^k>\right).
	\end{align*}
	After rearranging and using $\Psi_{yz}^k$ definition \eqref{scary:Psi_yz} we get
	\begin{align*}
		\E{\Psi_{yz}^{k+1}} 
		&\leq
		\max\left\{(1 + \theta\beta/2)^{-1}, (1-\gamma\delta), (1-\sigma_2/2), (1-(8\chi)^{-1})\right\}\Psi_{yz}^k
		\\&-
		2\E{\<x^{k+1} - x^*, y^{k+1} - y^*>}
		+
        \bb_F(x_g^k, x^*)
        \\&+
		\frac{\oL}{Lb}\left(\bb_F(\omega^k, x^*) - \bb_F(x_g^k, x^*) - \<\g F(x_g^k) - \g F(x^*) -\nu x_g^k + \nu x^*, \omega^k - x_g^k>\right)
		\\&\leq
		\left(1 - \frac{\sqrt{\mu}}{32\chi\sqrt{L}}\right)\Psi_{yz}^k
		\\&-
		2\E{\<x^{k+1} - x^*, y^{k+1} - y^*>}
        +
        \bb_F(x_g^k, x^*)
		\\&+
		\frac{\oL}{Lb}\left(\bb_F(\omega^k, x^*) - \bb_F(x_g^k, x^*) - \<\g F(x_g^k) - \g F(x^*) -\nu x_g^k + \nu x^*, \omega^k - x_g^k>\right).
	\end{align*}
\end{proof}

\begin{lemma}
	Let $\lambda$ be defined as follows:
	\begin{equation}\label{scary:lambda}
		\lambda = \frac{n}{b}\left(\frac{1}{2}+\frac{\oL}{Lb\tau_1}\right).
	\end{equation}
	Let $p_1$ be defined as follows:
	\begin{equation}\label{scary:p1}
		p_1 = \frac{1}{2\lambda}.
	\end{equation}
	Let $p_2$ be defined as follows:
	\begin{equation}\label{scary:p2}
		p_2 = \frac{\oL}{\lambda Lb\tau_1}.
	\end{equation}
	Then the following inequality holds:
	\begin{equation}
		\begin{split}
			\E{\Psi_x^k + \Psi_{yz}^k+\lambda\bb_F(\omega^{k+1}, x^*)} \leq\left(1 - \frac{1}{32}\min\left\{\frac{b}{n}, b\sqrt{\frac{\mu}{nL}}, \frac{b^2 L}{n \oL}\sqrt{\frac{\mu}{L}}, \frac{\sqrt{\mu}}{\chi\sqrt{L}}\right\}\right)(\Psi_x^0 + \Psi_{yz}^0+\lambda\bb_F(\omega^k, x^*)).
		\end{split}
	\end{equation}
\end{lemma}

\begin{proof}
	Combining \eqref{scary:eq:x} and \eqref{scary:eq:yz} gives
	\begin{equation}
		\begin{split}\label{scary:lyapunov_part_1}
			\E{\Psi_x^{k+1} + \Psi_{yz}^{k+1}}
			&\leq
			\left(1 - \frac{1}{20}\min\left\{\sqrt{\frac{\mu}{L}}, b\sqrt{\frac{\mu}{nL}}\right\}\right)\Psi_x^k
			+
			\left(1 - \frac{\sqrt{\mu}}{32\chi\sqrt{L}}\right)\Psi_{yz}^k
			\\&-
			\frac{\oL}{Lb\tau_1}\bb_F(x_g^k, x^*)
			+
			\frac{\oL}{Lb\tau_1}\bb_F(\omega^k, x^*)
			-
			\frac{1}{2}\bb_F(x_f^k, x^*)
			\\&\leq
			\left(1 - \frac{1}{32}\min\left\{b\sqrt{\frac{\mu}{nL}}, \frac{\sqrt{\mu}}{\chi\sqrt{L}}\right\}\right)(\Psi_x^k + \Psi_{yz}^k)
			\\&-
			\frac{\oL}{Lb\tau_1}\bb_F(x_g^k, x^*)
			+
			\frac{\oL}{Lb\tau_1}\bb_F(\omega^k, x^*)
			-
			\frac{1}{2}\bb_F(x_f^k, x^*).
		\end{split}
	\end{equation}
	Using \eqref{scary:line_omega} we get the following inequality:
	\begin{equation}\label{scary:lyapunov_part_2}
		\E{\bb_F(\omega^{k+1}, x^*)}
		\leq
		p_1\bb_F(x_f^k, x^*) + p_2\bb_F(x_g^k, x^*)+(1-p_1-p_2)\bb_F(\omega^k, x^*).
	\end{equation}
	Multiplying \eqref{scary:lyapunov_part_2} on $\lambda$ and combining with \eqref{scary:lyapunov_part_1} we get
	\begin{align*}
		\E{\Psi_x^{k+1} + \Psi_{yz}^{k+1}+\lambda\bb_F(\omega^{k+1}, x^*)}
		&\leq
		\left(1 - \frac{1}{32}\min\left\{b\sqrt{\frac{\mu}{nL}}, \frac{\sqrt{\mu}}{\chi\sqrt{L}}\right\}\right)(\Psi_x^k + \Psi_{yz}^k)
		+
		\lambda(1-p_1)\bb_F(\omega^k, x^*).
	\end{align*}
	Estimating $p_1$, using $\tau_1$ and $\tau_0$ definitions \eqref{scary:tau1}, \eqref{scary:tau0}
	\begin{align*}
		p_1
		&=
		\frac{b}{n}\left(2\left(\frac{1}{2}+\frac{\oL}{L b\tau_1}\right)\right)^{-1}
        =
        \frac{b}{n}\left(1+\frac{2\oL}{L b\tau_1}\right)^{-1}
        \geq \frac{b}{2n}\min\left\{1, \left(\frac{2\oL}{L b\tau_1}\right)^{-1}\right\}
        =
        \min\left\{\frac{b}{2n}, \frac{b^2 L\tau_1}{4n\oL}\right\}
        \\&\geq
        \min\left\{\frac{b}{2n}, \frac{b^2 L\tau_2}{10n\oL}\right\}
        =
        \min\left\{\frac{b}{2n}, \frac{b^2 L}{10n\oL}\min\left\{\frac{1}{2}, \max\left\{1, \frac{\sqrt{n}}{b}\right\}\sqrt{\frac{\mu}{L}}\right\}\right\}
        \\&\geq
        \min\left\{\frac{b}{2n}, \frac{b^2 L}{20n\oL}, \frac{b^2 L}{10n\oL}\max\left\{1, \frac{\sqrt{n}}{b}\right\}\sqrt{\frac{\mu}{L}}\right\}
        \geq
        \min\left\{\frac{b}{20n}, \frac{b^2 L}{10n \oL}\sqrt{\frac{\mu}{L}}\right\}.
	\end{align*}
	Therefore we conclude
	\begin{align*}
		\E{\Psi_x^{k+1} + \Psi_{yz}^{k+1}+\lambda\bb_F(\omega^{k+1}, x^*)}
		&\leq
		\left(1 - \frac{1}{32}\min\left\{\frac{b}{n}, b\sqrt{\frac{\mu}{nL}}, \frac{b^2 L}{n \oL}\sqrt{\frac{\mu}{L}}, \frac{\sqrt{\mu}}{\chi\sqrt{L}}\right\}\right)(\Psi_x^k + \Psi_{yz}^k+\lambda\bb_F(\omega^k, x^*)).
	\end{align*}

	This implies
	\begin{align*}
		\E{\Psi_x^k + \Psi_{yz}^k+\lambda\bb_F(\omega^k, x^*)}
		&\leq
		\left(1 - \frac{1}{32}\min\left\{\frac{b}{n}, b\sqrt{\frac{\mu}{nL}}, \frac{b^2 L}{n \oL}\sqrt{\frac{\mu}{L}}, \frac{\sqrt{\mu}}{\chi\sqrt{L}}\right\}\right)^k(\Psi_x^0 + \Psi_{yz}^0+\lambda\bb_F(x^0, x^*)).
	\end{align*}
	Using $\Psi_x^k$ definition \eqref{scary:Psi_x} we get
	\begin{align*}
		\E{\sqn{x^k - x^*}} &\leq \eta\E{\Psi_x^k} \leq \eta \E{\Psi_x^k + \Psi_{yz}^k + \lambda\bb_F(\omega^k, x^*)} 
		\\&\leq
		\left(1 - \frac{1}{32}\min\left\{\frac{b}{n}, b\sqrt{\frac{\mu}{nL}}, \frac{b^2 L}{n \oL}\sqrt{\frac{\mu}{L}}, \frac{\sqrt{\mu}}{\chi\sqrt{L}}\right\}\right)^k\eta(\Psi_x^0 + \Psi_{yz}^0+\lambda\bb_F(\omega^0, x^*)).
	\end{align*}
	Choosing $C = \eta(\Psi_x^0 + \Psi_{yz}^0 + \lambda\bb_F(\omega^k, x^*))$ and using the number of iterations
	\begin{equation*}
		k = 32\max\left\{\frac{n}{b}, \frac{\sqrt{n}}{b}\sqrt{\frac{L}{\mu}}, \frac{n\oL}{b^2 L}\sqrt{\frac{L}{\mu}}, \chi\sqrt{\frac{L}{\mu}}\right\}\log \frac{C}{\varepsilon} = \cO \left( \max\left\{\frac{n}{b}, \frac{\sqrt{n}}{b}\sqrt{\frac{L}{\mu}}, \frac{n\oL}{b^2 L}\sqrt{\frac{L}{\mu}}, \chi\sqrt{\frac{L}{\mu}}\right\}\log \frac{1}{\varepsilon}\right)
	\end{equation*}
	we get
	\begin{equation*}
		\sqn{x^k - x^*} \leq \epsilon.
	\end{equation*}
	Therefore the number of iterations of Algorithm~\eqref{scary:alg} is bounded by
	\begin{equation*}
		k = \cO \left(\left(\frac{n}{b}+\frac{\sqrt{n}}{b}\sqrt{\frac{L}{\mu}}+\frac{n\oL}{b^2 L}\sqrt{\frac{L}{\mu}}+\chi\sqrt{\frac{L}{\mu}}\right)\log \frac{1}{\epsilon}\right),
	\end{equation*}
	which concludes the proof.
\end{proof}

Let's prove the Corollary~\ref{corollary:adom+_conv}.
\begin{proof}
    The choice of the number of communication iterations $\sim \chi$ per algorithm iteration and a specific choice of $b=\max\{\sqrt{n\oL/L}, n\sqrt{\mu/L}\}$ provides the following upper bound on the number of algorithm iterations:
    \begin{equation*}
        N = \cO \left(\sqrt{\frac{L}{\mu}}\log \frac{1}{\epsilon}\right).
    \end{equation*}
    From this, it immediately follows that the upper bound on the number of communications is as follows:
    \begin{equation*}
        \cO \left(\chi\sqrt{\frac{L}{\mu}}\log \frac{1}{\epsilon}\right).
    \end{equation*}

    Now, let's estimate the number of oracle calls at each node. It is not difficult to show the following upper bound:
    \begin{equation*}
        Nb = \cO \left(\left(n+\sqrt{n}\sqrt{\frac{L}{\mu}}+b\sqrt{\frac{L}{\mu}}+\frac{n\oL}{b L}\sqrt{\frac{L}{\mu}}\right)\log \frac{1}{\epsilon}\right) = \cO \left(\left(n+\sqrt{n}\sqrt{\frac{\oL}{\mu}}\right)\log \frac{1}{\epsilon}\right),
    \end{equation*}
    which completes the proof.
    
\end{proof}

\section{Proof of Theorem~\ref{thm:lower_strong_conv}}\label{appendix:B}
As our graph counterexample, we will use the graph from \citep{metelev2024decentralized} because it allows us to obtain a lower bound not only in the setting of ``changing graphs'' but also in the setting of ``slowly changing graphs'', which will be a good addition.

Let's define $T_{a, b}$ as a graph consisting of two ``stars'' with sizes $a+1$ and $b+1$, whose centers are connected to an isolated vertex. In total, the graph will have $a+b+3$ vertices.

Let's say the left part of the graph $\mathcal{P}_1$ is the set of $a + 1$ vertices of the first star, and the right part $\mathcal{P}_2$ is correspondingly the set of $b + 1$ vertices of the second star. The middle vertex $v_m$ is the vertex connected to the centers $v_l$ and $v_r$ of the left and right stars, respectively.

If $v \in \mathcal{P}_1$, we define the ``hop to the right'' operation as follows: remove the edge $(v_l, v_m)$ and add the edge $(v, v_r)$. As a result, $v_m$ ceases to be the middle vertex, being replaced by the vertex $v$. The operation ``hop to the left'' is defined in the same way.

Now, let's describe the sequence of graphs that will make up the changing network. The first graph will be of the form $T_{0, m-3}$, followed by a series of ``hops to the left'', which increase the left part $\mathcal{P}_1$ of the graph and decrease the right. This will continue until the graph $T_{m-3, 0}$ appears. After this, a series of ``hops to the right'' occur until the network returns to its original form. Then, the cycle repeats.

\begin{lemma}
    For this sequence of graphs, there exists a corresponding sequence of positive weights \((A_k)_{k=0}^{\infty}\) and a sequence of Laplacian matrices \((W(k))_{k=0}^{\infty}\) for these weighted graphs, such that it satisfies \ref{assum:gossip_matrix_sequence} with
    \begin{equation}\label{scary_lower:chi_upper}
        \chi \leq 8m.
    \end{equation}
\end{lemma}
\begin{proof}
    This is a direct consequence of Lemma~8 from \citep{metelev2024decentralized}.
\end{proof}

Note that vertices $v_l$ and $v_r$ in the process of changing the network are always on the left and right parts, respectively. Denote by $\{g_i\}_{i=1}^m$:$y\in$ $\ell_2 \rightarrow \R$ the set of auxiliary functions corresponding to the vertices:

\begin{equation}\label{scary_lower:func}
g_i(y)=\begin{cases}
 \frac{\mu}{2}\sqn{y}+\frac{(L-\mu)}{4}\left[(y_1-1)^2+\sum_{k=1}^{\infty}(y_{2k}-y_{2k+1})^2\right], & i=v_l,
 \\
 \frac{\mu}{2}\sqn{y}+\frac{(L-\mu)}{4}\sum_{k=1}^{\infty}(y_{2k-1}-y_{2k})^2, &  i=v_r,
 \\
 \frac{\mu}{2(m-2)}\sqn{y}, & i\not\in \{v_l,v_r\}.
\end{cases}
\end{equation}

Let's describe the local functions on the nodes: let $x \in \ell_2^n$, then define $f_{ij}(x) = g_i(x_j)$, where $x_j \in \ell_2$. Accordingly, it turns out that $f_{ij}: x \in \ell_2^n \rightarrow \mathbb{R}$, but its gradient affects only the $j$th subspace of $\ell_2^n$, in which $x_k = 0$ for $k \neq j$. Hence, $F_i(x) = \frac{1}{n}\sum_{j=1}^n g_i(x_j)$.

Such a structure allows achieving that the ``transfer'' of the gradient to the next dimension in each subspace occurs once every $\Omega(m) = \Omega(\chi)$ communication iterations.

The solution to this optimization problem will be the vector $(x^*, \ldots, x^*) \in \ell_2^n$, $x^* = (1, q, q^2, \ldots)\in\ell_2$, $q=\frac{\sqrt{\frac{2}{3}L/\mu+\frac{1}{3}}-1}{\sqrt{\frac{2}{3}L/\mu+1}+\frac{1}{3}}$ .

Let $(e_1, e_2, \ldots, e_n)$ be sets of vectors that form a basis in the space $\ell_2^n$. Let $x_{ij}$ denote the coordinates along a set of vectors $e_j$ on the variable on the $i$th node.

Following the ideas of \citep{hendrikx2021optimal}, consider the expression
\begin{align*}
    A \overset{\Delta}{=} \sum_{i=1}^m\sum_{j=1}^n\sqn{x_{ij} - x^*}.
\end{align*}

Let's define the quantities $k_j = \min\{k \in \N_0 | \forall l \geq k,\forall i\in\{1,\ldots,m\} \rightarrow x_{ijl} = 0\}$. Using this definition and the convexity of $q^{2x}$ we get
\begin{equation}\label{scary_lower:A}
    A \geq \frac{m}{1-q^2}\sum_{i=1}^n q^{2k_j} \geq \frac{nm}{1-q^2}q^{\frac{2}{n}\sum_{j=1}^n k_j}.
\end{equation}

Let $T_c$ and $T_s$ be the number of communication rounds and the number of oracle calls at node $v_l$, respectively. Between the network state $T_{0, m-3}$ and the next such state there are $2m-6$ communication iterations, during which two ``transfers'' of the gradient from an odd position to an even one cannot occur. Therefore we get
\begin{equation}\label{scary_lower:t_c}
    k_j \leq 1 + \frac{T_c}{m-3}.
\end{equation}

Note that each $j$ corresponds to at least $\lceil k_j/2\rceil$ oracle calls to the function $f_{ij}$ for $i=v_l$, hence we get
\begin{equation}\label{scary_lower:t_s}
    \sum_{j=1}^n k_j \leq 2T_s.
\end{equation}  
Using \eqref{scary_lower:A}, \eqref{scary_lower:t_c} and \eqref{scary_lower:t_s} we get
\begin{equation}
    A \geq \frac{nm}{1-q^2}\max\left\{\left(1 - \frac{2}{\sqrt{\frac{2}{3}L/\mu+\frac{1}{3}} + 1}\right)^{2 + 2t_c/(m-3)}, \left(1 - \frac{2}{\sqrt{\frac{2}{3}L/\mu+\frac{1}{3}} + 1}\right)^{4t_s/n}\right\}.
\end{equation}
Based on the form of the function we can conclude that $\kappa_s = \frac{nL}{\mu}=n\kappa_b$, then using $x^0=0$, $\sqn{x^0_{ij}-x^*_{ij}}=(1-q^2)^{-1}$ and \eqref{scary_lower:chi_upper} we get
\begin{align*}
    \frac{1}{nm}\sum_{i=1}^m\sum_{j=1}^n\frac{\sqn{x_{ij} - x^*}}{\sqn{x^0_{ij} - x^*}} \geq \max\left\{\left(1 - \frac{2}{\sqrt{\frac{2}{3}\kappa_b+\frac{1}{3}} + 1}\right)^{2 + 16t_c/(\chi-24)}, \left(1 - \frac{2n}{\sqrt{n}\sqrt{\frac{2}{3}\kappa_s+n/3} + n}\right)^{4t_s/n}\right\},
\end{align*}
which concludes the proof.

{
	\renewcommand{\E}{\mathbb{E}}

\section{Proofs for Algorithm~\ref{alg:yup}}\label{appendix:C}
Before we start, let us denote
\begin{align}
     \label{eq: mixing-matrix}
     \mM(k) = (\mI_m - \mW(k))\otimes \mI_d
\end{align}
and
\begin{align}
	\label{eq:rho_and_chi}
	\rho = \frac{1}{\chi}
\end{align}
for the convenient analysis. Moreover, we need to introduce some definitions as
\begin{align*}
    \bar\vx^k &= \frac{1}{m} (\vone^\top_m \otimes \mI_d) x^k,\\
    \bar\vv^k &= \frac{1}{m} (\vone^\top_m \otimes \mI_d) v^k,\\
    S^k &= (S_1^k, \ldots, S_m^k),\\
    \nabla_{S^k}F(x^k) &= \big(\nabla_{S_1^k}F_1(x_1^k), \ldots, \nabla_{S_m^k}F_m(x_m^k)\big) \in \mathbb{R}^{md},\\
    \nabla_{S_i^k}F_i(x_i^k) &= \frac{1}{b}\sum\limits_{j \in S^k_i} \nabla f_{ij}(x_i^{k})
\end{align*}
Also we need to formulate some useful propositions:
\begin{proposition}
	\label{lem:gt}
	If $\bar\vv^0 = \frac{1}{m} (\vone^\top_m \otimes \mI_d) \mY^0$, then for any $k \ge 1$, according to \cref{alg:yup}, we get
	\begin{equation}\label{eq:avg_gradient}
		\bar\vv^k = \frac{1}{m}(\vone^\top_m \otimes \mI_d)\mY^k,
	\end{equation} 
	and
	\begin{equation}\label{eq:avg_model}
		\bar\vx^{k+1} = \bar\vx^k - \frac{\eta}{m}(\vone^\top_m \otimes \mI_d)\mY^k.
	\end{equation}
\end{proposition}
\begin{proof}
	We prove it using the induction. For $k=0$ it is trivial because of start point. Now suppose that at the $k$-th iteration, the relation \eqref{eq:avg_gradient} is true:
	\[\bar\vv^k = \frac{1}{m}(\vone^\top_m \otimes \mI_d)\mY^k.\]
	Hence, at the $(k+1)$-th iteration, we have
	\begin{align}
		\bar\vv^{k+1} & = \frac{1}{m}(\vone^\top_m \otimes \mI_d)\mV^{k+1} \nonumber \\
		& = \frac{1}{m}(\vone^\top_m \otimes \mI_d)\mV^k + \frac{1}{m} (\vone^\top_m \otimes \mI_d)(\mM(k) - \mI_{md})\mV^{k}  + \frac{1}{m}(\vone^\top_m \otimes \mI_d)\left(\mY^{k+1} - \mY^k\right) \nonumber \\
		& = \frac{1}{m}(\vone^\top_m \otimes \mI_d)\mV^k + \frac{1}{m}(\vone^\top_m \otimes \mI_d)\left(\mY^{k+1} - \mY^k\right) \nonumber \\
		& = \frac{1}{m}(\vone^\top_m \otimes \mI_d)\mY^{k+1}, \nonumber
	\end{align}
	where the third line follows from \cref{assum:gossip_matrix_sequence}:
        \begin{align*}
        (\vone^\top_m \otimes \mI_d)(\mM(k) - \mI_{md}) = -(\vone^\top_m \otimes \mI_d)(\mW(k) \otimes \mI_{d}) = -(\vone^\top_m\mW(k) \otimes \mI_d) = 0.
        \end{align*}
        Thus, we complete the proof of \eqref{eq:avg_gradient}. 
	For \eqref{eq:avg_model},
        \begin{align*}
		\bar\vx^{k+1} & = \bar\vx^{k} + \frac{1}{m}(\vone^\top_m \otimes \mI_d)(\mM(k) - \mI_{md})\mX^k - \frac{\eta}{m}(\vone^\top_m \otimes \mI_d)\mV^k \\
		& = \bar\vx^{k} - \eta\bar\vv^k = \bar\vx^k - \frac{\eta}{m} (\vone^\top_m \otimes \mI_d)\mY^k.
	\end{align*}
\end{proof}
\begin{proposition}
    If $\mW(k)$ satisfy \cref{assum:gossip_matrix_sequence} and $\mM(k)$ is taken from \eqref{eq: mixing-matrix}, then $\forall x \in \mathbb{R}^{md}$, we have
    \begin{align}
        \label{eq: matr-recur}
        \norm{\mM(k)x - \frac{1}{m}(\vone_m \otimes \mI_d)(\vone^\top_m \otimes \mI_d)x}^2 \leq (1 - \rho)\norm{x - \frac{1}{m}(\vone_m \otimes \mI_d)(\vone^\top_m \otimes \mI_d)x}^2,
    \end{align}
\end{proposition}
\begin{proof}
Note that
\begin{align*}
    \mM(k)(\vone_m \otimes \mI_d) = ((\mI_m - \mW(k))\otimes \mI_d)(\vone_m \otimes \mI_d) = ((\mI_m - \mW(k))\vone_m\otimes \mI_d) = \vone_m \otimes \mI_d.
\end{align*}
Therefore, 
\begin{align*}
     \norm{\mM(k)x - \frac{1}{m}(\vone_m \otimes \mI_d)(\vone^\top_m \otimes \mI_d)x}^2 =  \norm{\mM(k)\left(x - \frac{1}{m}(\vone_m \otimes \mI_d)(\vone^\top_m \otimes \mI_d)x\right)}^2.
\end{align*}
Decomposing $x - \frac{1}{m}(\vone_m \otimes \mI_d)(\vone^\top_m \otimes \mI_d)x$ by eigenvectors of $\mM(k)$ and using that 
\begin{align*}
    \vone_{md}^\top\left(\mI_{md} - \frac{1}{m}(\vone_m\vone_m^\top \otimes \mI_d)\right) = 0,
\end{align*}
we claim the final result.
\begin{remark}
    The proposition above is equivalent to
    \begin{align*}
        \norm{\mM(k)x^k - (\vone_m \otimes \mI_d)\bar\vx^k}^2 \leq (1 - \rho)\norm{x^k - (\vone_m \otimes \mI_d)\bar\vx^k}^2.
    \end{align*}
\end{remark}
\end{proof}
\subsection{Descent lemma}
\begin{lemma}{\textbf{(Descent lemma)}}
    Let \cref{assum:smoothness_F} and \cref{assum:gossip_matrix_sequence} hold. Then, after $k$ iterations of \cref{alg:yup}, we get
	\begin{align}
		\label{eq: descent-final}
		\E F(\bar\vx^{k+1}) &\le \E F(\bar\vx^{k}) - \frac{\eta}{2}\E \norm{\nabla F(\bar\vx^{k})}^2 + \frac{\eta }{m}\E \norm{\nabla F(\mX^{k}) - \mY^k}^2 + \frac{\eta L^2}{m}\E \norm{\mX^{k} - (\vone_m \otimes \mI_d)\bar\vx^k}^2 \nonumber \\
		&- \left(\frac{\eta}{2} - \frac{\eta^2 L}{2}\right)\E\norm{\bar\vv^{k}}^2.
	\end{align}
\end{lemma}
\begin{proof}
Starting with $L$-smoothness:
	\begin{align}
		F(\bar\vx^{k+1}) &\le F(\bar\vx^{k}) - \eta\dotp{\bar\vv^k}{\nabla F(\bar\vx^{k})} + \frac{\eta^2 L}{2}\norm{\bar\vv^{k}}^2 \nonumber \\
		& = F(\bar\vx^{k}) - \frac{\eta}{2}\norm{\nabla F(\bar\vx^{k})}^2 - \frac{\eta}{2}\norm{\bar\vv^{k}}^2 + \frac{\eta}{2}\norm{\nabla F(\bar\vx^{k}) - \bar\vv^k}^2 + \frac{\eta^2 L}{2}\norm{\bar\vv^{k}}^2 \nonumber \\
		& = F(\bar\vx^{k}) - \frac{\eta}{2}\norm{\nabla F(\bar\vx^{k})}^2 + \frac{\eta}{2}\norm{\nabla F(\bar\vx^{k}) - \bar\vv^k}^2 - \left(\frac{\eta}{2} - \frac{\eta^2 L}{2}\right)\norm{\bar\vv^{k}}^2 \nonumber \\
		&\le F(\bar\vx^{k}) - \frac{\eta}{2}\norm{\nabla F(\bar\vx^{k})}^2 + \frac{\eta}{2}\norm{\nabla F(\bar\vx^k)- \frac{1}{m}(\vone^\top_m \otimes \mI_d)\mY^k}^2 - \left(\frac{\eta}{2} - \frac{\eta^2 L}{2}\right)\norm{\bar\vv^{k}}^2 \nonumber \\
        &\le F(\bar\vx^{k}) - \frac{\eta}{2}\norm{\nabla F(\bar\vx^{k})}^2 + \frac{\eta}{2m}\norm{(\vone_m \otimes \mI_d)\nabla F(\bar\vx^k) - \nabla F(\mX^k) + \nabla F(\mX^k) - \mY^k}^2 \nonumber \\
        &- \left(\frac{\eta}{2} - \frac{\eta^2 L}{2}\right)\norm{\bar\vv^{k}}^2 \nonumber \\
		&\le F(\bar\vx^{k}) - \frac{\eta}{2} \norm{\nabla F(\bar\vx^{k})}^2 + \frac{\eta }{m} \norm{\nabla F(\mX^{k}) - \mY^k}^2 + \frac{\eta L^2}{m} \norm{\mX^{k} - (\vone_m \otimes \mI_d)\bar\vx^k}^2 \nonumber \\
		&- \left(\frac{\eta}{2} - \frac{\eta^2 L}{2}\right)\norm{\bar\vv^{k}}^2,
	\end{align}
	where in the last inequality we use $(a+b)^2 \leq 2a^2 + 2b^2$. Taking the expectation, we claim the final result.
\end{proof}
\subsection{Auxiliary lemmas}
\begin{lemma}
	\label{lem: recursion-grad-lemma}
	Let \cref{assum:smoothness_F_average} holds. Hence, after $k$ iterations the following is fulfilled:
	\begin{align*}
		\E\norm{\nabla F(\mX^{k+1}) - \mY^{k+1}}^2 \leq (1 - p)\E\norm{\nabla F(\mX^{k}) - \mY^{k}}^2 + \frac{(1 - p)\hat{L}^2}{b}\E\norm{\mX^{k+1} - \mX^k}^2.
	\end{align*}
\end{lemma}
\begin{proof}
	\begin{align}
		\label{eq: recursion-as-PAGE}
		\E&\norm{\nabla F(\mX^{k+1}) - \mY^{k+1}}^2 = p\E\norm{\nabla F(\mX^{k+1}) - \nabla F(\mX^{k+1})}^2 \nonumber \\ &+ (1 - p)\E\norm{\nabla F(\mX^{k+1}) - \mY^k - \nabla_{S^k} F(\mX^{k+1}) + \nabla_{S^k}F(\mX^k)}^2 \nonumber 
		\\&= (1 - p) \E\norm{\nabla F(\mX^{k+1}) - \nabla F(\mX^{k}) + \nabla F(\mX^{k}) - \mY^k - \nabla_{S^k} F(\mX^{k+1}) + \nabla_{S^k}F(\mX^k)}^2 \nonumber \\&= (1 - p)\E\norm{\nabla F(\mX^{k+1}) - \nabla F(\mX^{k}) - \nabla_{S^k} F(\mX^{k+1}) + \nabla_{S^k}F(\mX^k)}^2 \nonumber \\&+ (1 - p)\E\norm{\nabla F(\mX^{k}) - \mY^k}^2, 
	\end{align}
	Rewriting $\nabla_{S^k}F(\mX)$ as claimed before, using that $\E\norm{X - \E X}^2 \leq \E\norm{X}^2$, clarifying that indices in one batch are chosen independently and using the $\hat{L}$-average smoothness, one can obtain
	\begin{align}
		\label{eq: rec-with-b}
		\E\norm{\nabla F(\mX^{k+1}) - \mY^{k+1}}^2 \leq (1 - p)\E\norm{\nabla F(\mX^{k}) - \mY^{k}}^2 + \frac{(1 - p)\hat{L}^2}{b}\E\norm{\mX^{k+1} - \mX^k}^2,
	\end{align}
	what ends the proof.
\end{proof}
\begin{remark}
		The proof is similar to the proof of Lemma 3 in \cite{li2021page}, but we write it for each node in the same time.
	\end{remark}
\noindent Now we need to bound some extra terms for our Lyapunov's function. We use the next notation
\begin{align*}
	\Omega_1^k & = \E\norm{\mX^k - (\vone_m \otimes \mI_d)\bar\vx^k}^2, \nonumber \\
	\Omega_2^k & = \E\norm{\mV^k - (\vone_m \otimes \mI_d)\bar\vv^k}^2. \nonumber
\end{align*}
\begin{lemma}
	\label{lem: lyap-omegas}
	Let \cref{assum:gossip_matrix_sequence} holds. Therefore, for the \cref{alg:yup}, we have
	\begin{align*}
		\Omega_1^{k+1} &\leq \left(1 - \frac{\rho}{2}\right) \Omega_1^k + \frac{3\eta^2}{\rho}\Omega_2^k, \\
		\Omega_2^{k+1} &\leq \left(1 - \frac{\rho}{2}\right) \Omega_2^k + \frac{3}{\rho}\E\norm{\mY^{k+1} - \mY^k}^2. 
	\end{align*}
\end{lemma}
\begin{proof}
	Substituting the iteration of \cref{alg:yup} into $\Omega_1^{k+1}$, we get 
	\begin{align}
		\label{eq:omega-1}
		&   \norm{\mX^{k+1}-(\vone_m \otimes \mI_d)\bar\vx^{k+1}}^2 \nonumber \\
		& = \norm{\mM(k)\mX^k - \eta \mV^k - (\vone_m \otimes \mI_d)\bar\vx^k + (\vone_m \otimes \mI_d)\eta\bar\vv^k }^2 \nonumber \\
		& \le (1+\beta)(1-\rho)\norm{\mX^k - (\vone_m \otimes \mI_d)\bar\vx^k}^2  + \left(1+\frac{1}{\beta}\right)\eta^2\norm{\mV^k-(\vone_m \otimes \mI_d)\bar\vv^k}^2 \nonumber \\
		& \le  \left(1-\frac{\rho}{2}\right)\norm{\mX^k - (\vone_m \otimes \mI_d)\bar\vx^k}^2  + \left(1+\frac{2}{\rho}\right)\eta^2\norm{\mV^k-(\vone_m \otimes \mI_d)\bar\vv^k}^2 \nonumber \\
		& \le \left(1-\frac{\rho}{2}\right)\norm{\mX^k - (\vone_m \otimes \mI_d)\bar\vx^k}^2 + \frac{3\eta^2}{\rho}\norm{\mV^k-(\vone_m \otimes \mI_d)\bar\vv^k}^2,
	\end{align}
	where we choose $\beta = \frac{\rho}{2}$. For $\Omega_2^{k+1}$ respectively
	\begin{align*}
		\norm{\mV^{k+1} - (\vone_m \otimes \mI_d)\bar\vv^{k+1}}^2 & = \norm{\mV^{k+1} - (\vone_m \otimes \mI_d)\bar\vv^{k} + (\vone_m \otimes \mI_d)\bar\vv^{k} - (\vone_m \otimes \mI_d)\bar\vv^{k+1}}^2 \\
		& = \norm{\mV^{k+1} - (\vone_m \otimes \mI_d)\bar\vv^{k}}^2 - m\norm{\bar\vv^{k+1} - \bar\vv^{k}}^2 \\
		& \le \norm{\mV^{k+1} - (\vone_m \otimes \mI_d)\bar\vv^{k}}^2.
	\end{align*}
	Thus by the update rule of \cref{alg:yup}, one can obtain
	\begin{align}
		\label{eq:omega-2}
		\norm{\mV^{k+1} - (\vone_m \otimes \mI_d)\bar\vv^{k+1}}^2 
		& \le \norm{\mV^{k+1} - (\vone_m \otimes \mI_d)\bar\vv^{k}}^2 \nonumber \\
		& = \norm{\mM(k)\mV^{k} + \mY^{k+1} - \mY^k - (\vone_m \otimes \mI_d)\bar\vv^{k}}^2 \nonumber \\
		& \le  \left(1-\frac{\rho}{2}\right)\norm{\mV^{k} - (\vone_m \otimes \mI_d)\bar\vv^{k}}^2 + \left(1+\frac{2}{\rho}\right) \norm{\mY^{k+1} - \mY^k}^2 \nonumber \\
		& \le \left(1-\frac{\rho}{2}\right)\norm{\mV^{k} - (\vone_m \otimes \mI_d)\bar\vv^{k}}^2 + \frac{3}{\rho} \norm{\mY^{k+1} - \mY^k}^2.
	\end{align}
	Taking the expectation in both bounds, we claim the final result.
\end{proof}
\noindent As a consequence of \cref{lem: recursion-grad-lemma} and \cref{lem: lyap-omegas}, we need to bound some redundant expressions.
\begin{lemma}
	\label{lem: extra-expr-bound}
	Let Assumptions \ref{assum:smoothness_F}, \cref{assum:smoothness_F_average} and \ref{assum:gossip_matrix_sequence} hold. Then, after $k$ iterations of \cref{alg:yup}, we get
	\begin{align*}
		\E\norm{\mY^{k+1} - \mY^k}^2 &\leq (1 + p)\hat{L}^2\E\norm{\mX^{k+1} - \mX^k}^2 + 2p\E\norm{\nabla F(\mX^k) - \mY^k}^2,\\
		\E\norm{\mX^{k+1} - \mX^{k}}^2 &\leq 2\widetilde{C}\E\norm{\mX^k - (\vone_m \otimes \mI_d)\bar\vx^k}^2 + 2\eta^2\E\norm{\mV^k - (\vone_m \otimes \mI_d)\bar\vv^k}^2 + 2\eta^2 m \E\norm{\bar\vv^k}^2,
	\end{align*}
        where $\widetilde{C} = \max_{k} \norm{\mM(k) - \mI_{md}}^2 = \max_{k}\sigma_{\max}(\mM(k) - \mI_{md})^2 \leq 4$.
\end{lemma}
\begin{proof}
	Start with substituting $\mY^{k+1}$:
	\begin{align}
		\label{eq: adjac-y}
		\E\norm{\mY^{k+1} - \mY^k}^2 &= p\E\norm{\nabla F(\mX^{k+1}) - \mY^k}^2 + (1-p)\E\norm{\nabla_{S^k} F(\mX^{k+1}) - \nabla_{S^k} F(\mX^{k})}^2 \nonumber \\ 
		&= p\E\norm{\nabla F(\mX^{k+1}) - \nabla F(\mX^{k}) + \nabla F(\mX^{k}) - \mY^k}^2 \nonumber \\
		&+ (1-p)\E\norm{\nabla_{S^k} F(\mX^{k+1}) - \nabla_{S^k} F(\mX^{k})}^2 \nonumber \\
		&\leq p(1 + \beta)L^2\E\norm{\mX^{k+1} - \mX^k}^2 + p\left(1 + \frac{1}{\beta}\right)\E\norm{\nabla F(\mX^k) - \mY^k}^2 \nonumber \\
		&+ (1-p)\E\norm{\nabla_{S^k} F(\mX^{k+1}) - \nabla_{S^k} F(\mX^{k})}^2.
	\end{align}
	Let us bound the last term in \eqref{eq: adjac-y}. We have
	\begin{align}
		\label{eq: batch-grad}
		\E\norm{\nabla_{S^k} F(\mX^{k+1}) - \nabla_{S^k} F(\mX^{k})}^2 &= \E\sum\limits_{i=1}^m\norm{\nabla_{S_i^k} F_i(x_i^{k+1}) - \nabla_{S_i^k} F_i(x_i^{k})}^2 \nonumber \\
		&= \E\sum\limits_{i=1}^m\norm{\frac{1}{b} \sum\limits_{\ell \in \{S_i^k\}} \nabla f_{i\ell}(x_i^{k+1}) - \nabla f_{i\ell}(x_i^{k})}^2 \nonumber \\
		&= \E\sum\limits_{i=1}^m \frac{1}{b^2}\norm{\sum\limits_{\ell \in \{S_i^k\}} \nabla f_{i\ell}(x_i^{k+1}) - \nabla f_{i\ell}(x_i^{k})}^2 \nonumber \\
		&\leq \E\sum\limits_{i=1}^m \frac{1}{b}\sum\limits_{\ell \in \{S_i^k\}}\norm{\nabla f_{i\ell}(x_i^{k+1}) - \nabla f_{i\ell}(x_i^{k})}^2 \nonumber \\
		&\leq \E\sum\limits_{i=1}^m \frac{\hat{L}^2}{b}\sum\limits_{\ell \in \{S_i^k\}}\norm{x_i^{k+1} - x_i^{k}}^2 \nonumber \\
		&= \E\sum\limits_{i=1}^m \hat{L}^2\norm{x_i^{k+1} - x_i^{k}}^2 \nonumber \\
		&= \hat{L}^2\E\norm{\mX^{k+1} - \mX^k}^2.
	\end{align}
	Hence, substituting \eqref{eq: batch-grad} into \eqref{eq: adjac-y}, choosing $\beta$ as $1$ and using $L \leq \hat{L}$ (because of Jensen's inequality), one can obtain
	\begin{align}
		\label{eq: adjust-y-end} 
		\E\norm{\mY^{k+1} - \mY^k}^2 \leq (1 + p)\hat{L}^2\E\norm{\mX^{k+1} - \mX^k}^2 + 2p\E\norm{\nabla F(\mX^k) - \mY^k}^2.
	\end{align}
	The second expression can be bounded in the following way:
	\begin{align}
		\label{eq: adjac-x}
		\norm{\mX^{k+1} - \mX^{k}}^2 & = \norm{(\mM(k) - \mI_{md})\mX^k - \eta \mV^k}^2 \nonumber \\
		& = \norm{(\mM(k) - \mI_{md})(\mX^k - (\vone_m \otimes \mI_d)\bar\vx^k) - \eta \mV^k}^2  \nonumber \\
		& \le 2\widetilde{C}\norm{\mX^k - (\vone_m \otimes \mI_d)\bar\vx^k}^2 + 2\eta^2\norm{\mV^k}^2  \nonumber\\
		& = 2\widetilde{C}\norm{\mX^k - (\vone_m \otimes \mI_d)\bar\vx^k}^2 + 2\eta^2\norm{\mV^k - (\vone_m \otimes \mI_d)\bar\vv^k}^2 + 2\eta^2 m \norm{\bar\vv^k}^2.
	\end{align}
	 Taking the expectation, we claim the final result.
\end{proof}
\noindent Now we denote some expressions from \cref{lem: recursion-grad-lemma} and \cref{lem: extra-expr-bound} as follows
\begin{align*}
	\Delta^k &= \E\norm{\nabla F(\mX^{k}) - \mY^{k}}^2,\\
	\Delta_x^k &= \E\norm{\mX^{k+1} - \mX^k}^2.
\end{align*}
Consequently, substituting the bound of a first expression from \cref{lem: extra-expr-bound} in \cref{lem: lyap-omegas}, we get
\begin{align}
	\label{eq: bounds-letters}
	\Omega_1^{k+1} &\leq \left(1 - \frac{\rho}{2}\right)\Omega_1^k + \frac{3\eta^2}{\rho}\Omega_2^k, \nonumber \\
	\Omega_2^{k+1} &\leq \left(1 - \frac{\rho}{2}\right)\Omega_2^k + \frac{3}{\rho}(2p\Delta^k + (1 + p)\hat{L}^2\Delta_x^k).
\end{align}
Moreover, we can write
\begin{align*}
	\Delta^{k+1} &\leq (1 - p)\Delta^k + \frac{(1-p)\hat{L}^2}{b}\Delta_x^k,\\
	\Delta_x^k &\leq 2\widetilde{C}\Omega_1^k + 2\eta^2\Omega_2^k + 2\eta^2 m \E\norm{\bar\vv^k}^2.
\end{align*}
\subsection{Proof of Theorem \ref{thm:nonconvex-setup}} \label{upper}
\begin{proof}
	Rewriting the descent lemma in new notation, we have
	\begin{align*}
		\E F(\bar\vx^{k+1}) \le \E F(\bar\vx^{k}) - \frac{\eta}{2}\E \norm{\nabla F(\bar\vx^{k})}^2 + \frac{\eta }{m}\Delta^k + \frac{\eta L^2}{m}\Omega_1^k - \left(\frac{\eta}{2} - \frac{\eta^2 L}{2}\right)\E\norm{\bar\vv^{k}}^2.
	\end{align*}
	Also we can construct a Lyapunov's function in the following way:
	\begin{align}
		\label{eq: lyap-function}
		\Phi_k = \E F(\bar\vx^{k}) - F^* + C_0\Delta^k + s_1\Omega_1^k + s_2\Omega_2^k.
	\end{align}
	Then, adding some terms to the left-hand side of descent lemma mentioned above, one can obtain
	\begin{align*}
		\Phi_{k+1} &= \E F(\bar\vx^{k+1}) - F^* + C_0\Delta^{k+1} + s_1\Omega_1^{k+1} + s_2\Omega_2^{k+1} \nonumber \\
		&\leq \E F(\bar\vx^{k}) - F^* - \frac{\eta}{2}\E \norm{\nabla F(\bar\vx^{k})}^2 + \frac{\eta }{m}\Delta^k + \frac{\eta L^2}{m}\Omega_1^k - \left(\frac{\eta}{2} - \frac{\eta^2 L}{2}\right)\E\norm{\bar\vv^{k}}^2 \nonumber \\
		&+ C_0\left((1 - p)\Delta^k + \frac{(1 - p)\hat{L}^2}{b}\Delta_x^k\right) + s_1\left(\left(1 - \frac{\rho}{2}\right)\Omega_1^k + \frac{3\eta^2}{\rho}\Omega_2^k\right)\nonumber \\
		&+ s_2\left(\left(1 - \frac{\rho}{2}\right)\Omega_2^k + \frac{3}{\rho}(2p\Delta^k + (1 + p)\hat{L}^2\Delta_x^k)\right).
	\end{align*}
	Grouping the terms, we get
	\begin{align}
		\label{eq: lyap-descent-start}
		\Phi_{k+1} &\leq \E F(\bar\vx^{k}) - F^* - \frac{\eta}{2}\E \norm{\nabla F(\bar\vx^{k})}^2 + \Delta^k\left((1-p)C_0 + \frac{\eta}{m} + \frac{6ps_2}{\rho}\right) \nonumber \\
		&+ \Omega_1^k \left(\frac{\eta L^2}{m} + \left(1 - \frac{\rho}{2}\right)s_1\right) + \Omega_2^k\left(\frac{3\eta^2 s_1}{\rho} + \left(1 - \frac{\rho}{2}\right)s_2\right) \nonumber \\
		&+ \Delta_x^k\left(\frac{(1-p)\hat{L}^2C_0}{b} + \frac{3(1+p)\hat{L}^2s_2}{\rho}\right) - \left(\frac{\eta}{2} - \frac{\eta^2 L}{2}\right)\E\norm{\bar\vv^{k}}^2.
	\end{align}
	Hence, denoting
	\begin{align*}
		A &= (1-p)C_0 + \frac{\eta}{m} + \frac{6ps_2}{\rho}, \nonumber\\
		B &= \frac{(1-p)\hat{L}^2C_0}{b} + \frac{3(1+p)\hat{L}^2s_2}{\rho}, \nonumber\\
		C &= \frac{\eta L^2}{m} + \left(1 - \frac{\rho}{2}\right)s_1 \nonumber, \\
		D &= \frac{3\eta^2 s_1}{\rho} + \left(1 - \frac{\rho}{2}\right)s_2,
	\end{align*}
	and substituting these constants into \eqref{eq: lyap-descent-start}, we get
	\begin{align}
		\label{eq: lyap-in-const-descent-with-dx}
		\Phi_{k+1} \leq \E F(\bar\vx^{k}) - F^* - \frac{\eta}{2}\E \norm{\nabla F(\bar\vx^{k})}^2 + A\Delta^k 
		+ C\Omega_1^k  + D\Omega_2^k
		+ B\Delta_x^k - \left(\frac{\eta}{2} - \frac{\eta^2 L}{2}\right)\E\norm{\bar\vv^{k}}^2.
	\end{align}
	Using the definition of $\Delta_x^k$ and \cref{lem: extra-expr-bound} in \eqref{eq: lyap-in-const-descent-with-dx}, we finally have
	\begin{align}
		\label{eq: lyap-in-const-descent}
		\Phi_{k+1} &\leq \E F(\bar\vx^{k}) - F^* - \frac{\eta}{2}\E \norm{\nabla F(\bar\vx^{k})}^2 + A\Delta^k 
		+ (C + 2\widetilde{C}B)\Omega_1^k  + (D + 2\eta^2B)\Omega_2^k \nonumber \\
		&- \left(\frac{\eta}{2} - \frac{\eta^2 L}{2} - 2\eta^2nB\right)\E\norm{\bar\vv^{k}}^2 \nonumber \\
		&= \E F(\bar\vx^{k}) - F^* + s_1\Omega_1^k + s_2\Omega_2^k + A\Delta^k - \frac{\eta}{2}\E \norm{\nabla F(\bar\vx^{k})}^2 \nonumber \\  
		&+ (C + 2\widetilde{C}B - s_1)\Omega_1^k  + (D + 2\eta^2B -s_2)\Omega_2^k - \left(\frac{\eta}{2} - \frac{\eta^2 L}{2} - 2\eta^2mB\right)\E\norm{\bar\vv^{k}}^2.
	\end{align}
	Looking at the form of the descent lemma, we want to require the following:
	\begin{enumerate}
		\item $C_0 = A$.
		\item $\frac{\eta}{2} - \frac{\eta^2 L}{2} - 2\eta^2mB \geq 0$.
		\item $C + 2\widetilde{C}B - s_1 \leq 0$.
		\item $D + 2\eta^2B - s_2 \leq 0$.
	\end{enumerate}
	Before we start to solve this system relative to $\eta$, we assume the following form of constants $s_1$ and $s_2$:
	\begin{align}
		\label{eq: const-form}
		s_1 &= \frac{c_1(\rho, p, b)\hat{L}^2}{mL}, \nonumber \\
		s_2 &= \frac{c_2(\rho, p, b)L}{m\hat{L}^2}.
	\end{align}
	\textbf{First part}\\
	From the first requirement we get
	\begin{align}
		\label{eq: C_0-est}
		C_0 = \frac{\eta}{mp} + \frac{6s_2}{\rho}.
	\end{align}
	\textbf{Second part}\\
	From the second requirement:
	\begin{align}
		\label{eq: 2nd-start}
		\frac{\eta}{2} - \frac{\eta^2 L}{2} - 2\eta^2m \left(\frac{(1-p)\hat{L}^2C_0}{b} + \frac{3(1+p)\hat{L}^2s_2}{\rho}\right)\geq 0.
	\end{align}
	After substituting $C_0$ into \eqref{eq: 2nd-start}, we have
	\begin{align*}
		\frac{\eta}{2} - \frac{\eta^2 L}{2} -  \frac{2\eta^3(1-p)\hat{L}^2}{bp} - \frac{12\eta^2m(1-p)\hat{L}^2s_2}{b\rho}- \frac{6\eta^2m(1+p)\hat{L}^2s_2}{\rho}\geq 0.
	\end{align*}
	Using \eqref{eq: const-form}, one can obtain
	\begin{align*}
		\frac{\eta}{2} - \frac{\eta^2 L}{2} -  \frac{2\eta^3(1-p)\hat{L}^2}{bp} - \frac{12\eta^2(1-p)Lc_2(\rho, p, b)}{b\rho }- \frac{6\eta^2(1+p)Lc_2(\rho, p, b)}{\rho}\geq 0.
	\end{align*}
	Dividing both sides by $\eta$:
	\begin{align*}
		\frac{1}{2} - \frac{\eta L}{2} -  \frac{2\eta^2(1-p)\hat{L}^2}{bp} - \frac{12\eta (1-p)L c_2(\rho, p, b)}{b\rho }- \frac{6\eta (1+p)L c_2(\rho, p, b)}{\rho }\geq 0.
	\end{align*}
	Multiplying the left side by $2$ and entering a variable $r = \eta L$,
	\begin{align}
		\label{eq: unmod-2}
		1 - r -  \frac{4(1-p)r^2 \hat{L}^2}{bpL^2} - \frac{24(1-p) c_2(\rho, p, b)r}{b\rho}- \frac{12(1+p) c_2(\rho, p, b)r}{\rho}\geq 0.
	\end{align}
	Consequently, we could consider the next inequality
	\begin{align}
		\label{eq: mod-2}
		1 - r -  \frac{4(1-p)r^2\hat{L}^2}{bpL^2} - \frac{36c_2(\rho, p, b)r}{\rho}\geq 0.
	\end{align}
	Since $\frac{24(1-p)}{b} + 12(1+p) \leq 36$, if $r_0 = \eta_0 L$ satisfies \eqref{eq: mod-2}, then $r_0$ satisfies \eqref{eq: unmod-2} too. Hence, we could solve \eqref{eq: mod-2} to find a bound on $r$. Therefore,
	\begin{align*}
		r &\leq \frac{-\left(1 + \frac{36c_2(\rho, p, b)}{\rho}\right) + \sqrt{\left(1 + \frac{36c_2(\rho, p, b)}{\rho}\right)^2 + \frac{16(1-p)\hat{L^2}}{bpL^2}}}{\frac{8(1-p)\hat{L}^2}{bpL^2}} \\
		&= \frac{2}{\left(1 + \frac{36c_2(\rho, p, b)}{\rho}\right) + \sqrt{\left(1 + \frac{36c_2(\rho, p, b)}{\rho}\right)^2 + \frac{16(1-p)\hat{L^2}}{bpL^2}}}.
	\end{align*}
	Then,
	\begin{align*}
		\eta \leq \frac{2}{L\left(\left(1 + \frac{36c_2(\rho, p, b)}{\rho}\right) + \sqrt{\left(1 + \frac{36c_2(\rho, p, b)}{\rho}\right)^2 + \frac{16(1-p)\hat{L^2}}{bpL^2}}\right)}.
	\end{align*}
	Using $(a+b)^2 \leq 2a^2 + 2b^2$, we claim that
	\begin{align}
		\label{eq: eta-bound-2}
		\eta \leq \frac{2}{L\left(\left(1 + \frac{36c_2(\rho, p, b)}{\rho}\right) + \sqrt{2 + \frac{2592c_2^2(\rho, p, b)}{\rho^2} + \frac{16(1-p)\hat{L^2}}{bpL^2}}\right)}.
	\end{align}
	\textbf{Third part}\\
	From the third requirement one can obtain
	\begin{align}
		\label{eq: 3rd-start}
		\frac{\eta L^2}{m} + \left(1 - \frac{\rho}{2}\right)s_1 + 2\widetilde{C}\left(\frac{(1-p)\hat{L}^2C_0}{b} + \frac{3(1+p)\hat{L}^2s_2}{\rho}\right) - s_1 \leq 0.
	\end{align}
	Substituting $C_0$ in \eqref{eq: 3rd-start}:
	\begin{align*}
		\frac{\eta L^2}{m} - \frac{\rho}{2}s_1 + \frac{2\widetilde{C}(1-p)\hat{L}^2}{b}\left(\frac{\eta}{mp} + \frac{6s_2}{\rho}\right) + \frac{6\widetilde{C}(1+p)\hat{L}^2s_2}{\rho}\leq 0.
	\end{align*}
	Hence, we get
	\begin{align*}
		\frac{\eta L^2}{m} - \frac{\rho}{2}s_1 + \frac{2\widetilde{C}(1-p)\hat{L}^2\eta}{bmp} + \frac{12s_2\widetilde{C}(1-p)\hat{L}^2}{b\rho} + \frac{6\widetilde{C}(1+p)\hat{L}^2s_2}{\rho}\leq 0.
	\end{align*}
	Combining two last terms:
	\begin{align*}
		\frac{\eta L^2}{m} - \frac{\rho}{2}s_1 + \frac{2\widetilde{C}(1-p)\hat{L}^2\eta}{bmp} + \frac{\widetilde{C}\hat{L}^2s_2}{\rho}\left(\frac{12(1-p)}{b} + 6(1+p)\right)\leq 0.
	\end{align*}
	Grouping terms with $\eta$:
	\begin{align*}
		\eta\left(\frac{L^2}{m} + \frac{2\widetilde{C}(1-p)\hat{L}^2}{bmp}\right) - \frac{\rho}{2}s_1 + \frac{\widetilde{C}\hat{L}^2s_2}{\rho}\left(\frac{12(1-p)}{b} + 6(1+p)\right)\leq 0.
	\end{align*}
	Using the \eqref{eq: const-form}, one can obtain
	\begin{align*}
		\eta\left(\frac{L^2}{m} + \frac{2\widetilde{C}(1-p)\hat{L}^2}{bmp}\right) + \frac{\widetilde{C}L c_2(\rho, p, b)}{\rho m}\left(\frac{12(1-p)}{b} + 6(1+p)\right)\leq \frac{c_1(\rho, p, b)\hat{L}^2\rho}{2mL}.
	\end{align*}
	Consequently,
	\begin{align*}
		\frac{2\eta L^2}{\rho m}\left(1 + \frac{2\widetilde{C}(1-p)\hat{L}^2}{bpL^2}\right) + \frac{2\widetilde{C}L c_2(\rho, p, b)}{\rho^2 m}\left(\frac{12(1-p)}{b} + 6(1+p)\right)\leq \frac{c_1(\rho, p, b)\hat{L}^2}{mL}.
	\end{align*}
	Multiplying both sides by $\frac{m}{L}$:
	\begin{align}
		\label{eq: unmod-3rd}
		\frac{2\eta L}{\rho}\left(1 + \frac{2\widetilde{C}(1-p)\hat{L}^2}{bpL^2}\right) + \frac{2\widetilde{C}c_2(\rho, p, b)}{\rho^2}\left(\frac{12(1-p)}{b} + 6(1+p)\right)\leq \frac{c_1(\rho, p, b)\hat{L}^2}{L^2}.
	\end{align}
	Then, we can consider next inequality
	\begin{align}
		\label{eq: mod-3rd}
		\frac{2\eta L}{\rho}\left(1 + \frac{2\widetilde{C}(1-p)\hat{L}^2}{bpL^2}\right) + \frac{36\widetilde{C}c_2(\rho, p, b)}{\rho^2}\leq c_1(\rho, p, b),
	\end{align}
	where we use $\frac{12(1-p)}{b} + 6(1+p) \leq 18 - 6p \leq 18$. Hence, if we choose $\eta$ equal to some $\eta_0$ at which \eqref{eq: mod-3rd} holds, then \eqref{eq: unmod-3rd} holds too. Therefore, we can bound $\eta$:
	\begin{align*}
		\eta \leq \frac{\frac{\rho c_1(\rho, p, b)\hat{L}^2}{L^2} - \frac{36\widetilde{C}c_2(\rho, p, b)}{\rho}}{2L\left(1 + \frac{2\widetilde{C}(1-p)\hat{L}^2}{bpL^2}\right)}.
	\end{align*}
        Using $\hat{L} \geq L$:
        \begin{align}
		\label{eq: eta-bound-3}
		\eta \leq \frac{\rho c_1(\rho, p, b) - \frac{36\widetilde{C}c_2(\rho, p, b)}{\rho}}{2L\left(1 + \frac{2\widetilde{C}(1-p)\hat{L}^2}{bpL^2}\right)}.
	\end{align}
	\textbf{Fourth part}\\
	From the fourth requirement, we get:
	\begin{align}
		\label{eq: 4th-start}
		\frac{3\eta^2 s_1}{\rho} + \left(1 - \frac{\rho}{2}\right)s_2 + 2\eta^2\left(\frac{(1-p)\hat{L}^2C_0}{b} + \frac{3(1+p)\hat{L}^2s_2}{\rho}\right) - s_2 \leq 0.
	\end{align}
	Substituting the \eqref{eq: C_0-est}, we have
	\begin{align*}
		\frac{3\eta^2 s_1}{\rho} - \frac{\rho}{2}s_2 + \frac{2(1-p)\eta^3\hat{L}^2}{bmp} + \frac{12(1-p)\eta^2\hat{L}^2s_2}{b\rho} + \frac{6\eta^2(1+p)\hat{L}^2s_2}{\rho} \leq 0.
	\end{align*}
	Then, after combining last two terms, we get
	\begin{align*}
		\frac{3\eta^2 s_1}{\rho} - \frac{\rho}{2}s_2 + \frac{2(1-p)\eta^3\hat{L}^2}{bmp} + \frac{\eta^2\hat{L}^2s_2}{\rho}\left(\frac{12(1-p)}{b} + 6(1+p)\right) \leq 0.
	\end{align*}
	Using the \eqref{eq: const-form}, one can obtain
	\begin{align*}
		\frac{3\eta^2 c_1(\rho, p, b)\hat{L}^2}{m\rho L} - \frac{\rho c_2(\rho, p, b)L}{2m\hat{L}^2} + \frac{2(1-p)\eta^3\hat{L}^2}{bmp} + \frac{\eta^2Lc_2(\rho, p, b)}{\rho m}\left(\frac{12(1-p)}{b} + 6(1+p)\right) \leq 0.
	\end{align*}
	Consequently, 
	\begin{align}
		\label{eq: unmod-4th}
		\frac{\eta L}{\rho} \left(\frac{3\eta c_1(\rho, p, b)\hat{L}^2}{mL^2} + \frac{2\rho(1-p)\eta^2\hat{L}^2}{bmpL} + \frac{\eta c_2(\rho, p, b)}{m}\left(\frac{12(1-p)}{b} + 6(1+p)\right)\right) - \frac{\rho c_2(\rho, p, b)L}{2m\hat{L}^2} \leq 0.
	\end{align}
	If we choose $\eta \leq \frac{\rho}{L}$, then we could consider next inequality
	\begin{align}
		\label{eq: mod-4th}
		\frac{3\eta c_1(\rho, p, b)\hat{L}^2}{mL^2} + \frac{2\rho(1-p)\eta^2\hat{L}^2}{bmpL} + \frac{\eta c_2(\rho, p, b)}{m}\left(\frac{12(1-p)}{b} + 6(1+p)\right) - \frac{\rho c_2(\rho, p, b)L}{2m\hat{L}^2} \leq 0.
	\end{align}
	If \eqref{eq: mod-4th} holds for some $\eta_0$, where $\eta_0 L \leq \rho$, then \eqref{eq: unmod-4th} holds respectively. Hence, we could solve \eqref{eq: mod-4th} relative to $\eta$. For convenience, multiply both sides of the equation by $m$:
	\begin{align*}
		\frac{3\eta c_1(\rho, p, b)\hat{L}^2}{L^2} + \frac{2\rho(1-p)\eta^2\hat{L}^2}{bpL} + {\eta c_2(\rho, p, b)}\left(\frac{12(1-p)}{b} + 6(1+p)\right) - \frac{\rho c_2(\rho, p, b)L}{2\hat{L}^2} \leq 0.
	\end{align*}
	Moreover, we could use $\frac{12(1-p)}{b} + 6(1+p) \leq 18$ and $\rho \leq 1$. Therefore, using $L \leq \hat{L}$, we can consider
	\begin{align}
		\label{eq: mod-4th-2}
		\frac{\hat{L}^2}{L^2}({3\eta c_1(\rho, p, b)} + 18\eta c_2(\rho, p, b)) + \frac{2(1-p)\eta^2\hat{L}^2}{bpL} - \frac{\rho c_2(\rho, p, b)L}{2\hat{L}^2} \leq 0.
	\end{align}
	Then, if $\eta_0$ satisfies \eqref{eq: mod-4th-2}, consequently it satisfies \eqref{eq: mod-4th} and \eqref{eq: unmod-4th}. So, we could solve \eqref{eq: mod-4th-2}:
	\begin{align*}
		\frac{2(1-p)\eta^2\hat{L}^2}{bpL} + \frac{\eta\hat{L^2}}{L^2}\left(3 c_1(\rho, p, b)+{18 c_2(\rho, p, b)}\right) - \frac{\rho c_2(\rho, p, b)L}{2\hat{L}^2} \leq 0.
	\end{align*}
	Solving the inequality, we get
	\begin{align*}
		\eta &\leq \frac{-\left(3 c_1(\rho, p, b)+{18 c_2(\rho, p, b)}\right) + \sqrt{\left(3 c_1(\rho, p, b)+{18 c_2(\rho, p, b)}\right)^2 + \frac{8(1-p)}{bp}\frac{\rho c_2(\rho, p, b)}{2}}}{\frac{4(1-p)\hat{L}^2}{bpL}} \\
		&= \frac{4{\rho c_2(\rho, p, b)}}{{4L}\left(\left(3 c_1(\rho, p, b)+{18 c_2(\rho, p, b)}\right) + \sqrt{\left(3 c_1(\rho, p, b)+{18 c_2(\rho, p, b)}\right)^2 + \frac{8(1-p)}{bp}\frac{\rho c_2(\rho, p, b)}{2}\frac{L^4}{\hat{L}^4}}\right)} \\
		&= \frac{{\rho c_2(\rho, p, b)}}{{L}\left(\left(3 c_1(\rho, p, b)+{18 c_2(\rho, p, b)}\right) + \sqrt{\left(3 c_1(\rho, p, b)+{18 c_2(\rho, p, b)}\right)^2 + \frac{8(1-p)}{bp}\frac{\rho c_2(\rho, p, b)}{2}\frac{L^4}{\hat{L}^4}}\right)}.
	\end{align*}
	Using that $(a+b)^2 \leq 2a^2 + 2b^2$ and $\frac{L^4}{\hat{L}^4} \leq \frac{\hat{L^2}}{L^2}$, we can give a bit rough estimate of $\eta$:
	\begin{align}
		\label{eq: eta-bound-4}
		\eta \leq \frac{\rho c_2(\rho, p, b)}{L\left(3 c_1(\rho, p, b)+18 c_2(\rho, p, b) + \sqrt{18 c_1^2(\rho, p, b)+648 c_2^2(\rho, p, b) + \frac{4(1-p)\rho c_2(\rho, p, b)}{bp}\frac{\hat{L^2}}{L^2}}\right)}.
	\end{align}
	\textbf{Selection of $c_1(\rho, p, b)$ and $c_2(\rho, p, b)$}\\
	Let us take these parameters in the following way
	\begin{align*}
		c_1(\rho, p, b) &= 2\widetilde{C}(1+\rho)\left(\sqrt{\frac{(1-p)\hat{L}^2}{bpL^2}} + \frac{1}{\widetilde{C}}\right), \nonumber \\
		c_2(\rho, p, b) &= \frac{\rho^2}{18\widetilde{C}}.
	\end{align*}
	From \eqref{eq: eta-bound-2}, we get
	\begin{align*}
		\eta \leq \frac{2}{L\left(\left(1 + \frac{2\rho}{\widetilde{C}}\right) + \sqrt{2 + \frac{8\rho^2}{\widetilde{C}^2} + \frac{16(1-p)\hat{L}^2}{bpL^2}}\right)}.
	\end{align*}
	Consequently, we could roughen the estimate by $\rho \leq 1$:
	\begin{align}
		\label{eq: final-bound-2}
		\eta \leq \frac{2}{L\left(\left(1 + \frac{2}{\widetilde{C}}\right) + \sqrt{2 + \frac{8}{\widetilde{C}^2} + \frac{16(1-p)\hat{L}^2}{bpL^2}}\right)}.
	\end{align}
	From \eqref{eq: eta-bound-3}, one can obtain
	\begin{align*}
		\eta \leq \frac{2\widetilde{C}(\rho + \rho^2)\left(\sqrt{\frac{(1-p)\hat{L}^2}{bpL^2}} +\frac{1}{\widetilde{C}}\right)  - 2\rho}{2L\left(1 + \frac{2\widetilde{C}(1-p)\hat{L}^2}{bpL^2}\right)}.
	\end{align*}
	Hence, final bound is
	\begin{align}
		\label{eq: final-bound-3}
		\eta \leq \frac{2\rho^2 + 2\widetilde{C}(\rho^2 + \rho)\sqrt{\frac{(1-p)\hat{L}^2}{bpL^2}}}{L\left(1 + \frac{2\widetilde{C}(1-p)\hat{L}^2}{bpL^2}\right)}.
	\end{align}
	From \eqref{eq: eta-bound-4}, we have
	\begin{align*}
		\eta \leq \frac{\rho^3}{18\widetilde{C}L\left(6\widetilde{C}(1+\rho)\left(\sqrt{\frac{(1-p)\hat{L}^2}{bpL^2}} + \frac{1}{\widetilde{C}}\right)+\frac{\rho^2}{\widetilde{C}} + \sqrt{72\widetilde{C}^2(1+\rho)^2\left(\sqrt{\frac{(1-p)\hat{L}^2}{bpL^2}} + \frac{1}{\widetilde{C}}\right)^2+\frac{2\rho^4}{\widetilde{C}^2}  + \frac{2(1-p)\rho^3\hat{L}^2}{9\widetilde{C}bp}L^2}\right)}.
	\end{align*}
	Using that $(a+b)^2 \leq 2a^2 + 2b^2$ and $\rho \leq 1$, we claim
	\begin{align}
		\label{eq: final-bound-4}
		\eta \leq \frac{\rho^3}{18\widetilde{C}L\left(12 + \frac{1}{\widetilde{C}} + 12\widetilde{C}\sqrt{\frac{(1-p)\hat{L}^2}{bpL^2}} + \sqrt{288 + \frac{2}{\widetilde{C}^2} + \frac{288\widetilde{C}^2(1-p)\hat{L}^2}{bpL^2} + \frac{2(1-p)\hat{L}^2}{9\widetilde{C}bpL^2}}\right)}.
	\end{align}
	From $\eta \leq \frac{\rho}{L}$ and bounds \eqref{eq: final-bound-2}, \eqref{eq: final-bound-3} and \eqref{eq: final-bound-4} the next result follows:
	\begin{align*}
		\Phi_{k+1} \leq \Phi_k -  \frac{\eta}{2}\E \norm{\nabla F(\bar\vx^{k})}^2.
	\end{align*}
	Summarizing over $t$, we claim
	\begin{align*}
		\frac{1}{N}\sum\limits_{k=0}^{N-1}\E \norm{\nabla F(\bar\vx^{k})}^2 \leq \frac{2(\Phi_0 - \Phi_k)}{\eta N},
	\end{align*}
        where $\Phi_0 = F(x^0) - F^* = \Delta$ because of initialization. Hence, for reaching $\frac{1}{N}\sum\limits_{k=0}^{N-1}\E \norm{\nabla F(\bar\vx^{k})}^2 \leq  \epsilon^2$, we need 
	\begin{align*}
		N = \mathcal{O}\left(\frac{L\Delta\left(1 + \sqrt{\frac{(1-p)\hat{L}^2}{bpL^2}}\right)}{\rho^3 \epsilon^2}\right)
	\end{align*}
	iterations. Choosing $\hat{x}^N$ uniformly from $\{\bar\vx^{k}\}_{k=0}^{N-1}$, we claim the final result.
\end{proof}
\subsection{Proof of Corollary \ref{cor:nonconvex_setup}} \label{upper-separate}
\begin{proof}
    First, we need to clarify that multi-stage consensus technique allows to avoid $\chi^3$ factor in \cref{thm:nonconvex-setup}, but apply $\chi$ to a number of communications. Hence, choosing $b = \frac{\sqrt{n}\hat{L}}{L}, p = \frac{b}{n + b}$, we get
    \begin{align*}
        N_{comm} = \mathcal{O}\left(\frac{\chi L\Delta\left(1 + \sqrt{\frac{n\hat{L^2}}{b^2L^2}}\right)}{\epsilon^2}\right) = \mathcal{O}\left(\frac{\chi L\Delta}{\epsilon^2}\right), 
    \end{align*}
    while the number of iterations is 
    \begin{align*}
        N_{iter} = \mathcal{O}\left(\frac{L\Delta\left(1 + \sqrt{\frac{n\hat{L^2}}{b^2L^2}}\right)}{\epsilon^2}\right) = \mathcal{O}\left(\frac{L\Delta}{\epsilon^2}\right).
    \end{align*}
    Moreover, number of local computations (in average) is equal to
    \begin{align*}
        n + N_{iter}(pn + (1-p)b) = n + C\frac{L\Delta}{\epsilon^2}\left(\frac{2n\sqrt{n}\frac{\hat{L}}{L}}{n + \sqrt{n}\frac{\hat{L}}{L}}\right) \leq n + C\frac{\sqrt{n}\hat{L}\Delta}{\epsilon^2} = \mathcal{O}\left(n + \frac{\sqrt{n}\hat{L}\Delta}{\epsilon^2}\right),
    \end{align*}
    where $C$ is a constant from $\mathcal{O}(\cdot)$. This finishes the proof.
\end{proof}
\subsection{Lower bounds for nonconvex setting}\label{proof_lower}
The main idea of lower bound construction is to provide an example of a bad function for which we can estimate the minimum required number of iterations or oracle calls to solve the problem. Hence, we need to consider some class of problems, oracles, and algorithms among which we shall dwell.\\
\noindent Before we start, let us propose some additional facts for a clear proof. \\ \noindent Consider the next function:
\begin{align}
    \label{eq: bad-func}
    l(x) = -\Psi(1)\Phi(\left[x\right]_1) + \sum\limits_{j=2}^{d}\left(\Psi(-\left[x\right]_{j-1})\Phi(-\left[x\right]_{j})  - \Psi(\left[x\right]_{j-1})\Phi(\left[x\right]_{j})\right), 
\end{align}
where
\begin{align}
    \label{eq: extra-bad}
    \Psi(z) &= \begin{cases}
        0 \qquad z \leq \frac{1}{2};\\
        \exp\left(1 - \frac{1}{(2z - 1)^2}\right) \qquad z > \frac{1}{2},
    \end{cases} \nonumber\\
    \Phi(z) &= \sqrt{e}\int\limits_{-\infty}^{z} e^{-\frac{t^2}{2}}\, dt.
\end{align}
\textbf{Properties}.
It has already been shown in \cite{arjevani2023lower} (see Lemma 2) that $l(x)$ satisfies the following properties:
\begin{enumerate}
    \item $\forall x \in \mathbb{R}^d \ l(x) - \inf_{x} l(x) \leq \Delta_0 d \ $ with $\Delta_0 = 12$.
    \item $l(x)$ is $L_0$-smooth with $L_0 = 152$.
    \item $\forall x \in \mathbb{R}^d \ \norm{\nabla l(x)}_{\infty} \leq G_0$ with $G_0 = 23$.
    \item $\forall x \in \mathbb{R}^d: \left[x\right]_d = 0 \ \norm{\nabla l(x)}_{\infty} \geq 1$.
\end{enumerate}
Moreover, let us introduce the next definition
\begin{align}
    \label{eq: prog}
    \text{prog}(x) = \begin{cases}
        0 \qquad x=0;\\
        \max_{1 \leq j \leq d}\{j: \left[x\right]_j \neq 0\} \qquad \text{otherwise}.
    \end{cases}
\end{align}
Hence, the function $f$ is called zero-chain, if 
\begin{align*}
    \text{prog}(\nabla f(x)) \leq \text{prog}(x) + 1.
\end{align*}
This means that if we start at point $x = 0$, after a gradient estimation we earn at most one non-zero coordinate of $x$. What is more, $l(x)$ is zero-chain function.\\
\noindent Let us formulate auxiliary lemmas which help to estimate the lower bound.
\begin{lemma}
    \label{lem: smooth-prop}
    Consider the function $l(x)$ which is defined in \eqref{eq: bad-func}. Moreover, let us define 
    \begin{align*}
        l_j(x) = l_j([x]_{j-1}, [x]_j) = \begin{cases}
            -\Psi(1)\Phi(\left[x\right]_1), \qquad &j=1,\\
            \Psi(-\left[x\right]_{j-1})\Phi(-\left[x\right]_{j})  - \Psi(\left[x\right]_{j-1})\Phi(\left[x\right]_{j}), \qquad &j > 1.
        \end{cases}
    \end{align*}
    Therefore, $l_j(x)$ is $L_0$-smooth for all $j$, where $L_0$ is defined above.
    \begin{proof}
        For the case $j = 1$ it is obvious, since it is enough to consider $l(x)$ with $d = 1$. For $j > 1$ we consider $l(x)$ with $d = j$. It is known that $l(x)$ is $L_0$-smooth (see \textbf{Property} 2). It means that for all $x, y \in \mathbb{R}^j$ the next inequality holds:
        \begin{align*}
            \norm{\nabla l(y) - \nabla l(x)} \leq L_0\norm{y - x}.
        \end{align*}
        Hence, we can get $x$ and $y$ as follows: $[x]_k = [y]_k = 0$ for all $k = 1, \ldots, j - 2$. Consequently,
        \begin{enumerate}
            \item $[\nabla l(y)]_k = [\nabla l(x)]_k$ for all $k = 1, \ldots, j - 2$, since these components of gradient can depend on all coordinates of variable except $(j - 1)$-\textit{th} and $j$-\textit{th}.
            \item Due to the zero-chain property of $l(x)$, $(j - 2)$-\textit{th} coordinate of variable cannot change $(j - 1)$-\textit{th} one after computation of a gradient, that is, $\frac{\partial l_{j - 1}}{\partial x_{j - 1}}(0, x_{j-1}) = 0$.
        \end{enumerate}
        As a result, last two components of $\nabla l(x)$ are equal to $\frac{\partial l_{j}}{\partial x_{j - 1}}(x_{j-1}, x_{j})$ and $\frac{\partial l_{j}}{\partial x_{j}}(x_{j-1}, x_{j})$ respectively, and this is exactly the gradient of $l_j(x_{j-1}, x_j)$. Using that other coordinates of $x$ and $y$ are equal to zero, we conclude that $l_j(x_{j-1}, x_j)$ is $L_0$-smooth.
    \end{proof}
\end{lemma}
\begin{lemma}
\label{lem: effect-L}
    Consider the function $l(x)$ which is defined in \eqref{eq: bad-func}. Suppose that
    \begin{align*}
        \hat{l}_1(x) &= -\Psi(1)\Phi(\left[x\right]_1) + \sum\limits_{j \text{ odd; } j \geq 2}\left(\Psi(-\left[x\right]_{j-1})\Phi(-\left[x\right]_{j})  - \Psi(\left[x\right]_{j-1})\Phi(\left[x\right]_{j})\right), \\
        \hat{l}_2(x) &= \sum\limits_{j \text{ even}}\left(\Psi(-\left[x\right]_{j-1})\Phi(-\left[x\right]_{j})  - \Psi(\left[x\right]_{j-1})\Phi(\left[x\right]_{j})\right).
    \end{align*}
    Hence, if we divide $\hat{l}_i(x)$ into $n$ parts in the following way:
    \begin{align*}
        \hat{l}_i(x) = \frac{1}{n}\sum\limits_{k=1}^n \hat{l}_{ik}(x),
    \end{align*}
    where
    \begin{align*}
        \hat{l}_{1k}(x) &= \begin{cases}
             -n\Psi(1)\Phi(\left[x\right]_1) + \sum\limits_{j \geq 2, \ j \equiv 1 \text{ mod 2n}}n\left(\Psi(-\left[x\right]_{j-1})\Phi(-\left[x\right]_{j})  - \Psi(\left[x\right]_{j-1})\Phi(\left[x\right]_{j})\right) \qquad k=1;\\
             \sum\limits_{j \equiv 2k-1 \text{ mod 2n}}n\left(\Psi(-\left[x\right]_{j-1})\Phi(-\left[x\right]_{j})  - \Psi(\left[x\right]_{j-1})\Phi(\left[x\right]_{j})\right) \qquad k > 1;
        \end{cases}\\
        \hat{l}_{2k}(x) &= \sum\limits_{j \equiv 2k \text{ mod 2n}}n\left(\Psi(-\left[x\right]_{j-1})\Phi(-\left[x\right]_{j})  - \Psi(\left[x\right]_{j-1})\Phi(\left[x\right]_{j})\right),
    \end{align*}
    then 
    \begin{align*}
        \frac{1}{n}\sum\limits_{k=1}^n \norm{\nabla \hat{l}_{ik}(y) - \nabla \hat{l}_{ik}(x)}^2 \leq nL_0^2\norm{y - x}^2 
    \end{align*}
    for $i = 1, 2$ and for all $x, y \in \mathbb{R}^d$.
\end{lemma}
\begin{proof}
    Let us consider the structure of $\nabla \hat{l}_{ik}(x)$. This part of $\hat{l}_i(x)$ depends only on some coordinates of $x$. Hence, given the definition of each slice, we can identify which coordinates of $\hat{l}_{ik}(x)$ can be non-zero. For example, $\nabla \hat{l}_{11}(x)$ can be non-zero only in components $1, 2n, 2n+1, 4n, 4n+1, \ldots$ because this function depends only on these coordinates. \\
    \noindent Moreover, since $n \geq 2$ (when $n = 1$, the fact above is obvious), if we consider $\hat{l}_{ik}(x)$ and $\hat{l}_{ij}(x)$, then there is no intersection of sets of potentially non-zero coordinates of gradients of these functions due to the construction. Therefore, we can apply \cref{lem: smooth-prop} to obtain that $\hat{l}_i(x)$ is $L_0$-smooth. Using that full gradient is 
    \begin{align*}
        \nabla \hat{l}_i(x) = \frac{1}{n}\sum\limits_{k=1}^n \nabla \hat{l}_{ik}(x),
    \end{align*}
    one can obtain
    \begin{align*}
        \frac{1}{n}\sum\limits_{k=1}^n \norm{\nabla \hat{l}_{ik}(y) - \nabla \hat{l}_{ik}(x)}^2 = n \norm{\nabla \hat{l}_{i}(y) - \nabla \hat{l}_{i}(x)}^2 \leq nL_0^2\norm{y - x}^2.
    \end{align*}
\end{proof}
\begin{remark}
    \cref{lem: effect-L} asserts that in essence the function under consideration and its pieces satisfy the assumptions from \cref{thm: lower-bound-sq-norm}. The main effect consists of the scaling factor $\frac{1}{\sqrt{n}}$.
\end{remark}
\textbf{Proof of \cref{thm: lower-bound-sq-norm}}
\noindent \begin{proof} We need to introduce functions $F_i$, structure of a time-varying graphs and mixing matrices respectively to construct the lower bound. Then, we can consider next functions
\begin{align}
    l_1(x) &= \frac{m}{\left\lceil \frac{m}{3}\right\rceil}\left(-\Psi(1)\Phi(\left[x\right]_1) + \sum\limits_{j \text{ odd}}\left(\Psi(-\left[x\right]_{j-1})\Phi(-\left[x\right]_{j})  - \Psi(\left[x\right]_{j-1})\Phi(\left[x\right]_{j})\right)\right), \nonumber \\
    l_2(x) &= \frac{m}{\left\lceil \frac{m}{3}\right\rceil}\left(\sum\limits_{j \text{ even}}\left(\Psi(-\left[x\right]_{j-1})\Phi(-\left[x\right]_{j})  - \Psi(\left[x\right]_{j-1})\Phi(\left[x\right]_{j})\right)\right). \nonumber
\end{align}
As a sequence of graphs, we take star graphs, for each of which the center changes with time according some rules, which we explain later. We derive the mixing matrix from the Laplacian matrix of the graph at the moment $t$ in the next way:
\begin{align*}
    \mW(t) = \mI - \frac{1}{\lambda_{max}(L(t))}L(t).
\end{align*}
This matrix is obviously a mixing matrix by reason of symmetry and doubly stochasticity. Moreover, $\rho(t) = 1 - \mu_2(\mW(t))$, where $\mu_2(\mW(t))$ is the second largest eigenvalue of $\mW(t)$. Consequently, using the spectrum of $L(t)$, one can obtain that $\rho(t) = \rho = \frac{1}{m}$.\\
Let us specify the functions $F_i$ at each node:
\begin{align*}
    F_i(x) = \begin{cases}
        \frac{LC^2}{3L_0}l_1\left(\frac{x}{C}\right) \qquad &1 \leq i \leq {\left\lceil \frac{m}{3}\right\rceil} \Leftrightarrow i \in S_1, \\
        \frac{LC^2}{3L_0}l_2\left(\frac{x}{C}\right) \qquad &{\left\lceil \frac{m}{3}\right\rceil} + 1 \leq i \leq 2{\left\lceil \frac{m}{3}\right\rceil} \Leftrightarrow i \in S_2,\\
        0 \qquad &\text{otherwise} \Leftrightarrow i \in S_3,
    \end{cases}
\end{align*}
where we clarify $C$ later.\\
Also we need to separate each function into $n$ blocks. It is enough to divide $F_i(x)$ according to \cref{lem: effect-L} with corresponding multiplicative constants.
Therefore, since $l_1(x)$ and $l_2(x)$ are $3L_0$-smooth, $F_i(x)$ is $L$-smooth for every $C > 0$.\\
We also can bound $F(0) - \inf_{x} F(x)$
\begin{align*}
    F(0) - \inf_x F(x) \leq \frac{LC^2\Delta_0 d}{3L_0}
\end{align*}
according to the definition of $F_i$. Hence, we need
\begin{align*}
    \frac{LC^2\Delta_0 d}{3L_0} \leq \Delta.
\end{align*}
Now we are ready to  divide our proof into three parts.\\
\noindent \textbf{Number of communications}\\
We want the transfer of information between sets $S_1$ and $S_2$ to not occur for as long as possible. This requires that the center of the star graph is not a vertex from $S_1$ or $S_2$, or it is not a vertex of $S_3$ that already has information from other sets of vertices. Therefore, let us specify the changes of the graphs with time according to the following principle: first we go through all the vertices of the set $S_3$, and after that we choose the vertex that allows the exchange of information between $S_1$ and $S_2$.
Then, mentioning that $\frac{1}{m}\sum\limits_{i=1}^m F_i(x) = \frac{LC^2}{3L_0}l\left(\frac{x}{C}\right)$ and 
\begin{align*}
    \text{prog}(\nabla F_i(x))\begin{cases}
        = \text{prog}(x) + 1 \qquad (\text{prog}(x) \text{ is even and } i \in S_1) \text{ or } (\text{prog}(x) \text{ is odd and } i \in S_2);\\
        \leq \text{prog}(x) \qquad \text{otherwise},
    \end{cases}
\end{align*}
we claim that for increasing the $\text{prog}(x)$ at $1$ we need at least $m - 2{\left\lceil \frac{m}{3}\right\rceil} + 1$ iterations (without considering local computations). Therefore, after $N$ iterations
\begin{align*}
    \text{prog}(N) = \max_{1 \leq i \leq m, \ 0 \leq t \leq N}\text{prog}(x_i^t) \leq \left\lfloor \frac{N}{m - 2{\left\lceil \frac{m}{3}\right\rceil} + 1} \right\rfloor + 1. 
\end{align*}
Also it is easy to make sure that if $m \geq 3$, then $m - 2{\left\lceil \frac{m}{3}\right\rceil} + 1 \geq \frac{m}{4}$. Then
\begin{align*}
    \text{prog}(N) \leq \left\lfloor \frac{4N}{m} \right\rfloor + 1.
\end{align*}
\noindent \textbf{Number of local computations}\\
Here we use the same idea as in first part. Let us consider the next oracle computation: we take one of pieces on each node uniformly, i.e. $\mathbb{P}\{\text{block with index $k$ is chosen}\} = \frac{1}{n}$ for every $k = 1, \ldots, n$. Hence, at the current moment, we need a \textit{specific} piece of function, because according to structure of $l(x)$, each gradient estimation can "defreeze" at most one component and only a computation on a certain block makes it possible. Let us define the number of required gradient calculations as $n_{avg}$. Therefore, 
\begin{align*}
    \mathbb{E}\{n_{avg}\} = \sum\limits_{i=1}^{\infty} \frac{i}{n}\left(\frac{n-1}{n}\right)^{i-1} = n,
\end{align*}
where $\frac{1}{n}\left(\frac{n-1}{n}\right)^{i-1}$ is a is the probability that at $i$-th moment we take the correct piece. Thus, after $K$ local computations on each node we can change at most $\left\lfloor \frac{K}{n}\right\rfloor +1$ coordinates.\\
\textbf{Final result}\\
Hence, if considered algorithm makes $N$ communications and $K$ local computations on each node, then
\begin{align*}
    \text{prog}(N, K) = \max_{1 \leq i \leq m, \ 0 \leq t \leq N}\text{prog}(x_i^t) \leq \min\left(\left\lfloor \frac{4N}{m}\right\rfloor + 1, \left\lfloor \frac{K}{n}\right\rfloor +1  \right)
\end{align*}
Consequently, for every $N \geq \frac{m}{4}$ and $K \geq n$ consider
\begin{align*}
    d = 2 + \min\left(\left\lfloor\frac{4N}{m}\right\rfloor, \left\lfloor \frac{K}{n}\right\rfloor\right).
\end{align*}
It is easy to verify thar
\begin{align*}
    d < \min\left(\frac{16N}{m}, \frac{4K}{n}\right).
\end{align*}
Moreover, we choose $C$ as
\begin{align*}
    C = \left(\frac{3L_0\Delta }{L\Delta_0\min\left(\frac{16N}{m}, \frac{4K}{n}\right)}\right)^{\frac{1}{2}}.
\end{align*}
Hence, clarifying that $\text{prog}(N, K) < d$, we have
\begin{align*}
    \E \norm{\nabla F(\hat{x}_N)}^2 &\geq \min_{\left[x\right]_d = 0} \norm{\nabla F(\hat{x}_N)}^2 = \frac{L^2C^2}{9L_0^2} \min_{\left[x\right]_d = 0} \norm{\nabla l(\hat{x}_N)}^2 \geq \frac{L^2C^2}{9L_0^2} = \max\left(\frac{L\Delta m}{48NL_0\Delta_0}, \frac{L\Delta n}{12KL_0\Delta_0}\right) \\
    &\geq \frac{L\Delta m}{96NL_0\Delta_0} + \frac{L\Delta n}{24KL_0\Delta_0}= \Omega\left(\frac{L\Delta m}{N} + \frac{L\Delta n}{K}\right), 
\end{align*}
where the second inequality holds from fourth property of $l(x)$.\\
Consequently, applying \cref{lem: effect-L} to $\{F_i\}_{i=1}^m$ and noting that $\chi = \Theta\left(m\right)$, we finish the proof.
\end{proof}
}

\end{document}